\documentclass[12pt]{article}
 \usepackage{graphicx,amsthm,amsmath,amsfonts,amssymb}
 \usepackage[colorlinks]{hyperref}
\usepackage[numbers]{natbib}
 
\newcommand{\rs}[1]{{\mbox{\scriptsize \sc #1}}}

\newcommand{\sr}[1]{{\cal #1}}
 \newcommand{\dd}[1]{\mathbb{#1}}
 
 \newcommand{\ray}{{\rm r}}

\newcommand{\eq}[1]{(\ref{eq:#1})}
\newcommand{\eqn}[1]{(\ref{eqn:#1})}
\newcommand{\lem}[1]{Lemma~\ref{lem:#1}}

\newcommand{\thr}[1]{Theorem~\ref{thr:#1}}

\theoremstyle{definition}

\newcommand{\fig}[1]{Figure~\ref{fig:#1}}
\newcommand{\app}[1]{Appendix~\ref{app:#1}}
\newcommand{\sectn}[1]{Section~\ref{sect:#1}}

\newcommand{\lemt}[1]{\ref{lem:#1}}

\newcommand{\thrt}[1]{\ref{thr:#1}}

\newcommand{\appt}[1]{\ref{app:#1}}

\newcommand{\br}[1]{\left\langle #1 \right\rangle}
\newcommand{\brb}[1]{\big\langle #1 \big\rangle}
\newcommand{\ol}{\overline}

\newcommand{\pend}{\hfill \thicklines \framebox(6.6,6.6)[l]{}}
\newenvironment{proof*}[1]{\noindent {\sc  #1} \rm}{\pend}

\newtheorem{theorem}{Theorem}[section]
\newtheorem{lemma}{Lemma}[section]

\newtheorem{remark}{Remark}[section]

\newtheorem{corollary}{Corollary}[section]

 \newcommand{\setnewcounter} {
 \setcounter{subsection}{0}
 \setcounter{equation}{0}
 \setcounter{conjecture}{0}
 \setcounter{assumption}{0}
 \setcounter{question}{0}
 \setcounter{definition}{0}
 \setcounter{theorem}{0}
 \setcounter{corollary}{0}
 \setcounter{lemma}{0}
 \setcounter{proposition}{0}
 \setcounter{remark}{0}
}

\newenvironment{mylist}[1]{\begin{list}{}
{\setlength{\itemindent}{#1mm}}
{\setlength{\itemsep}{0ex plus 0.2ex}}
{\setlength{\parsep}{0.5ex plus 0.2ex}}
{\setlength{\labelwidth}{10mm}}
}{\end{list}}

\begin{document}

 \title{\bf  Stationary distribution of a two-dimensional SRBM:
   geometric views and boundary measures} 

\author{J. G. Dai\footnote{School of Operations Research and
    Information Engineering, Cornell University, 136 Hoy Road, Ithaca,
    NY 14853; jim.dai@cornell.edu; on leave from
    H. Milton Stewart School of Industrial and 
    Systems Engineering, Georgia Institute of Technology} { and} Masakiyo
  Miyazawa\footnote{Department of Information Sciences, Tokyo University of Science, Noda, Chiba 278-0017, Japan; miyazawa@is.noda.tus.ac.jp}}
\date{\normalsize November 2, 2012, Revised}

\maketitle

\begin{abstract}
  We present three sets of results for the stationary distribution of
  a two-dimensional semimartingale reflecting Brownian motion (SRBM)
  that lives in the nonnegative quadrant.  The SRBM data can
  equivalently be specified by three geometric objects,
 an ellipse and two lines, in the two-dimensional Euclidean
  space.  First, we revisit the variational problem (VP) associated with the
  SRBM. Building on Avram, Dai and Hasenbein (2001), we show that the
  value of the VP at a point in the quadrant is equal
  to the optimal value of a linear function over a convex
  domain. Depending on the location of the point, the convex domain is
  either ${\cal D}^{(1)}$ or ${\cal D}^{(2)}$ or ${\cal D}^{(1)}\cap
  {\cal D}^{(2)}$, where each ${\cal D}^{(i)}$, $i=1, 2$,
  can easily be described by the three geometric objects. Our results
  provide a geometric interpretation for the value function of the
  VP and allow one to see geometrically when one edge of the
  quadrant has influence on the optimal path traveling from the origin
  to a destination point. 
 Second, we provide a geometric condition
  that characterizes the existence of a product form stationary
  distribution. Third, we establish exact tail asymptotics
  of two boundary measures that are associated with the stationary
  distribution; a key step in our proof is to sharpen two
  asymptotic inversion lemmas in Dai and Miyazawa (2011) that allow
  one to
  infer the exact tail asymptotic of a boundary measure from the
  singularity of its moment 
  generating function.

\end{abstract}

\section{Introduction}
\label{sect:}
Multidimensional semimartingale reflecting Brownian motions (SRBMs)
have been extensively studied in literature. They 
serve as diffusion
models of multiclass queueing networks. See \cite{Williams95} for a
survey on SRBMs and 
\cite{Williams96} on diffusion approximations. In this paper, we focus
on a two-dimensional SRBM 
 $Z=\{Z(t); t \ge 0\}$ that lives on  the state space
 $\dd{R}_{+}^{2}$.  We adopt the definition of an SRBM used  in
 \citet{DaiMiyazawa2011a} 
and  follow the notational system there. In particular, the data for the SRBM
 consists of  a covariance matrix
$\Sigma$,  a drift vector $\mu$, and a reflection matrix $R$. They are
written as  
\begin{eqnarray*}
 \mu = \left(\begin{array}{cc} \mu_{1} \\ \mu_{2} \end{array}\right), \qquad \Sigma = \left(\begin{array}{cc} \Sigma_{11} & \Sigma_{12} \\ \Sigma_{21} & \Sigma_{22} \end{array}\right), \qquad R = \left(\begin{array}{cc} r_{11} & r_{12} \\ r_{21} & r_{22} \end{array}\right).
\end{eqnarray*}
In short, $Z$ is said to be a $(\Sigma, \mu, R)$-SRBM on $\dd{R}^2_+$.
The covariance matrix $\Sigma$ and drift vector $\mu$ uniquely
determine
the following ellipse passing through the origin in $\dd{R}^2$:
\begin{eqnarray*}
  \frac 12 \br{\theta, \Sigma \theta} + \br{\mu, \theta} = 0, \qquad \theta \in \dd{R}^{2},
\end{eqnarray*}
where $\br{x,y}$ denotes the usual inner product for vectors $x, y \in
\dd{R}^{2}$ (see its definition at the end of this section). 
 The first column $R^1$ of the reflection matrix $R$ determines a line,
called line 1 in this paper,  that
passes through the origin and is orthogonal to direction $R^1$. In
addition to the origin, the line intersects the ellipse at point
$\theta^{(2,\rm r)}\neq 0$. Similarly, the second column $R^2$ of $R$ defines
line 2 with intersection $\theta^{(1,\rm r)}\neq 0$ on the ellipse.
  These three geometric objects, the ellipse and two lines,
uniquely determine the SRBM data $(\Sigma, \mu, R)$. See \fig{ellipse-ray-stable} for an example of these three geometric
objects and \sectn{Geometrical view} for a precise
 definition of these objects.  All results in this paper
will have a geometric interpretation in 
terms of these three objects.

Throughout this paper, we assume conditions (\ref{eqn:P matrix}) and
(\ref{eqn:stability 1}) in Section~\ref{sect:SRBM and LDP} on the SRBM
data. These two conditions are necessary and sufficient for a
two-dimensional SRBM to have a unique stationary
distribution. \sectn{SRBM data} will give a geometric interpretation 
of these 
two conditions. In this paper, we present three sets of results for a
two-dimensional SRBM. We first give a geometric interpretation for the
value function of the variational problem that is associated with the
SRBM. The value function has been conjectured to be
 the large deviations rate function of the stationary distribution.
See inequalities (\ref{eqn:LDP 1}), (\ref{eqn:LDP 2}), and the text immediately
following these inequalities for discussion on large deviations
rate functions.
 We next give
a geometric condition for the SRBM to have a product form stationary
distribution.  We finally present exact tail asymptotics of two
boundary measures that are associated with the stationary distribution.

Closely related to the first set of our results,
\citet{AvramDaiHasenbein01} studied the variational problem that is
associated with a multidimensional SRBM.  In the two-dimensional case,
they obtained explicit expressions for the value function $I(v)$ for each
point $v$ in the state space $\dd{R}^2_+$. In particular, they proved
that $I(v)=\min(I^{(1)}(v), I^{(2)}(v))$, where $I^{(i)}(v)$ is the
value function of a constrained variational problem with constrained path
being restricted not to touch boundary 
\begin{displaymath}
F_i=\{x\in \dd{R}^2_+:
x_i=0\}  
\end{displaymath}
 of the state space, $i=1,2$. To compute $I^{(i)}(v)$, they
showed that an optimal path from the origin to $v$ is either a straight
line through the interior of the state space or a broken path whose
first segment travels on boundary $F_{3-i}$. In either case, they give
an explicit algebraic expression for $I^{(i)}(v)$. Central in their
analysis is a notion of ``exit velocity'' that dictates the direction
for the broken path to follow when it exits the boundary $F_{3-i}$. Theorem
\ref{thr:geometric view 1} and Lemma~\ref{lem:geometric} of this paper
give a geometric interpretation of $I^{(i)}(v)$. In particular, we
show that the exit velocity $\tilde a^{1}$ on boundary $F_1$
and the exit velocity $\tilde a^{2}$ on boundary $F_2$
have
the following expressions
\begin{equation}
  \label{eq:exitVelocity}
  \tilde a^{1} = n^\Gamma(\tilde \theta^{(2, r)})
\quad \text{ and } \quad 
  \tilde a^{2} = n^\Gamma(\tilde \theta^{(1, r)}),
\end{equation}
respectively, where $n^\Gamma(\theta)$, defined in (\ref{eqn:orthogonal
  to theta}) in Section~\ref{sect:Geometrical view}, is the normal
direction of the ellipse at point $\theta$ on the ellipse and $\tilde
\theta^{(i,r)}$ is the ``symmetry'' of $\theta^{(i,\rm r)}$ on the
ellipse; see \fig{ellipse-ray-normal} for an illustration
of $n^\Gamma(\theta)$ and \fig{SixPoints} for an
illustration of symmetry points. \thr{domain 1} gives a geometric
interpretation of $I(v)$. 
Harrison and Hasenbein \cite{HarrisonHasenbein09} re-express the
solution in \cite{AvramDaiHasenbein01} for $I(v)$ in an alternative form
that is somewhat more explicit, and offer further elaboration on the
structure of that solution.
  
A key step in our analysis is to connect the variational problem with the
optimization problems of a linear function over certain convex
domains. Those convex domains are closely related to the convergence
domain of the two-dimensional moment generating function of the
stationary distribution.  \citet{DaiMiyazawa2011a} studied the exact
tail asymptotics of the stationary marginal distribution for the two
dimensional SRBM.  In particular, they characterized the convergence
domain through the three geometric objects. In this paper, we show that the
convergence domain is the intersection of two convex domains and the
rate function $I(v)$ can be uniquely determined by these two convex
domains. Each of these convex domains can easily be determined by the three
geometric objects. In this way, we provide a geometric interpretation
of the large deviations rate function $I(v)$ (see Theorems \thrt{geometric view 1} and \thrt{domain 1}).

Our second set of results is to characterize geometrically the
existence of the product form stationary distribution. We prove in 
\ref{thm:productform} that the stationary distribution of a
two-dimensional SRBM has
a product form if and only if 
\begin{displaymath}
\tilde  \theta^{(1,\rm r)} = \tilde \theta^{(2,\rm r)}.
\end{displaymath}
See \fig{Product form0} for examples. This geometric
characterization contrasts with the algebraic characterization
developed in Harrison and Williams \cite{HarrisonWilliams87b}. Their
characterization, in terms of a skew symmetry condition, is valid for an
SRBM in an arbitrary dimension.

Our last set of results (\thr{boundary measure asymptotics}) concerns the exact tail asymptotics for two 
boundary measures $\nu_1$ and $\nu_2$ that are associated with the
stationary distribution. This set of results is closely related to 
the recent paper \cite{DaiMiyazawa2011a} in which moment generating
functions are used to obtain the exact tail asymptotic for stationary
marginal distributions. Their analysis relies on two asymptotic inversion
lemmas to obtain the exact tail asymptotics from the singularity
 of the moment generating functions. Unfortunately, these
inversion lemmas are not sharp enough for us to study exact tail
asymptotics for boundary measures. In this paper, we develop a series
of three lemmas, Lemmas \ref{lem:tail equivalence 1}, \ref{lem:handy
  1} and \ref{lem:handy 2}, to sharpen the asymptotic inversion
lemmas. These lemmas should  also be useful for other problems; see
\sectn{Exact asymptotics} for more discussions. 

This paper consists of six sections. In \sectn{SRBM and LDP}, we
discuss the variational problem of an SRBM and its relation to the
large deviations rate function for the stationary distribution. In
\sectn{Geometrical view}, we first define three geometric objects and two
convex domains ${\cal D}^{(1)}$ and ${\cal D}^{(2)}$. We then state the
first set of results, Theorems \thrt{geometric view 1} and \thrt{domain 1}.  In \sectn{ADH characterization}, we review some
results in \citet{AvramDaiHasenbein01} and prove \thr{geometric view 1}. In \sectn{Product form}, we give a
geometric characterization for the existence of a product form
stationary distribution. In Section \ref{sect:Exact asymptotics}, we
study exact tail asymptotics of two boundary measures. In the
  rest of this section, we summarize the notation used in this paper
  for the convenience  of readers.

\noindent {\bf {Notation}:}  Let $\dd{R}$ and $\dd{C}$ be the sets of
all real and complex numbers, respectively, and let $\dd{R}_{+}$ be
the set of all nonnegative real numbers. All vectors are envisioned as
column vectors, and $x^{\rs{t}}$ is the transpose of vector $x$. In
this paper, we adopt the standard inner product, that is, 
\begin{eqnarray*}
  \br{x,y} = x^{\rs{t}} y \quad \text{ for } x, y \in \dd{R}^{2}
\end{eqnarray*}
with the corresponding standard norm $\| x \| = \sqrt{\br{x,x}}$ for
each $x\in \dd{R}^2$.  We list
the notation for the 
primitive data and the stationary distribution of SRBM in Tables
\ref{tab:primitive}  and \ref{tab:sd} below.  
\begin{table}[htdp]
\caption{Notation associated with the primitive data ($i = 1,2$)}
\label{tab:primitive}
\begin{center}
\begin{tabular}{ll|ll}
  $\mu$ & drift vector of SRBM \hspace{5ex} & $\Sigma$ & covariance matrix of SRBM\\
  $R$ & reflection matrix & $R^{i}$ & the $i$-th column of $R$\\
  $\gamma(\theta)$ & $ - \frac 12 \br{\theta, \Sigma \theta} -
  \br{\mu, \theta}$  for $\theta\in \dd{R}^2$  & $\gamma_{i}(\theta)$
  & $ \br{R^{i}, \theta}$  for $\theta\in \dd{R}^2$\\
  $\Gamma$ & $ \{ \theta \in \dd{R}^{2}; \gamma(\theta) > 0\}$ & $\Gamma_{i}$ & $ \{ \theta \in \Gamma; \gamma_{3-i}(\theta) < 0\}$\\
  $\partial \Gamma$ & $ \{ \theta \in \dd{R}^{2}; \gamma(\theta) = 0\}$ & $\partial \Gamma_{i}$ & $ \{ \theta \in \Gamma; \gamma_{3-i}(\theta) = 0\}$\\
 $\Gamma_{\max}$ & $ \{\theta \in \dd{R}^{2}; \exists\,\theta' \in
 \Gamma, \theta < \theta'\}$ & $\theta^{(i,\max)}$ & $ \arg \max_{\theta \in \partial \Gamma} \theta_{i}$\\
  $\theta^{(i,\ray)}$ & the nonzero intersection & $\tilde{\theta}^{(i,\ray)}$ & $ \theta^{(i,\max)}$ if $\theta^{(i,\ray)} = \theta^{(i,\max)}$,\\
  &  of $\gamma(\theta) =0$ and $\gamma_{3-i}(\theta) = 0$ & & otherwise, $\theta$ on $\partial \Gamma$ such that \\
 & & & $\theta_{i} = \theta^{(i,\ray)}_{i}$ and $\theta_{3-i} \ne \theta^{(i,\ray)}_{3-i}$\\
 $\theta^{(i,\Gamma)}$ & $ \arg \max_{\theta \in \partial \Gamma_{i}} \theta_{i}$ & $\theta^{(v,\Gamma)}$ & $ \arg \max_{\theta \in \partial \Gamma} \br{v, \theta_{i}}$ for $v \in \dd{R}_{+}^{2}$\\
  $p^{(i)}$ & the $i$-th row of $(\det{R}) R^{-1}$ & ${\rm n}^{\Gamma}(\theta)$ & $ \Sigma \theta + \mu$ \\
  $a^{i}$ & $ {\rm n}^{\Gamma}(\theta^{(3-i, \ray)})$ & $\tilde{a}^{i}$ & $ {\rm n}^{\Gamma}(\tilde{\theta}^{(3-i, \ray)})$
   \end{tabular}
\end{center}
\label{notation 1}
\end{table}%

Let $Z(\infty)$ be a random variable following the stationary
distribution $\pi$ of the SRBM, and let $\nu_{i}$ be the measure on
the boundary face $F_{3-i} \equiv \{ x \in \dd{R}_{+}^{2}; x_{3-i} =
0\}$ associated with $\pi$ as defined in Section 2 of \cite{DaiMiyazawa2011a}. 

\begin{table}[htdp]
\caption{Notation associated with the stationary distribution of SRBM $(i=1,2)$}
\label{tab:sd}
\begin{center}
\begin{tabular}{ll|ll}
  $\varphi(\theta)$ & $ \dd{E}(e^{\br{\theta, Z(\infty)}})$ &
  $\varphi_{i}(\theta_{3-i})$ & the moment generating function of $\nu_{i}$\\
  $\sr{D}$ & the interior of $\{\theta \in \dd{R}^{2}; \varphi(\theta)
  < \infty\}$ & $\sr{D}^{(i)}$ & $ \{\theta \in \dd{R}^2; \exists\,
  \theta'\in \Gamma, \theta<\theta', \theta'_{i} < \theta_{i}^{(i,\Gamma)}\}$\\
  $\tau_{i}$ & $ \sup\{ \theta_{i} > 0; \theta \in \sr{D} \}$ & ${\rm e}^{(i)}$ & the unit vector for the $i$-th axis\\
  ${\rm e}^{i}$ & $ {\rm e}^{(3-i)}$ & ${\rm n}^{i}$ & $ {\rm e}^{(i)}$\\
  $\zeta$ & the joint density of $\pi$ & $\zeta_{i}$ & the marginal
  density of $\pi$\\
  VP & variational problem & $I(v)$ & the value function of  VP
\end{tabular}
\end{center}
\label{notation 2}
\end{table}%
\vspace{-2ex}

\section{Variational problem and large deviations principle}
\label{sect:SRBM and LDP}
\setnewcounter

In Section \ref{sect:ADH characterization}, we recover results
obtained by \citet{AvramDaiHasenbein01}, leveraging recent results of
\citet{DaiMiyazawa2011a}. Those two papers use different sets of
notation.  For example, \citet{AvramDaiHasenbein01} used $\Gamma$ for
covariance matrix and $\theta$ for drift vector instead of $\Sigma$ and
$\mu$ in this 
paper. Furthermore, $R$ in \cite{AvramDaiHasenbein01} is normalized in
such a way that $r_{11} = r_{22} = 1$. Following
\cite{DaiMiyazawa2011a}, we will not use any normalization in this
paper. In \citet{AvramDaiHasenbein01}, the inner product $\br{x,y}$
for vectors $x, y \in \dd{R}^{2}$ is defined as $x^{\rs{t}}
\Sigma^{-1} y$. Their inner product
certainly simplifies computation, but is also potentially confusing
because there is no indication of $\Sigma^{-1}$.

As in \cite{DaiMiyazawa2011a},  we assume that $\Sigma$ is non-singular
and the SRBM data $(\Sigma, \mu, R)$ satisfy
\begin{eqnarray}
\label{eqn:P matrix}
 && r_{11} > 0, \quad r_{22} > 0, \quad r_{11} r_{22} - r_{12} r_{21} > 0,\\
\label{eqn:stability 1}
 && r_{22} \mu_{1} - r_{12} \mu_{2} < 0, \quad r_{11} \mu_{2} - r_{21} \mu_{1} < 0.
\end{eqnarray} 
Under conditions (\ref{eqn:P matrix}) and (\ref{eqn:stability 1}), the
$(\Sigma, \mu, R)$-SRBM is well defined, and its stationary
distribution exists and is unique; see, for example,
\cite{HarrisonHasenbein09,HobsonRogers93}.

  We are going to introduce the variational problem associated with the
  SRBM shortly. For that, we first define the Skorohod problem
  associated with the reflection matrix $R$.  Let
  $\dd{D}^2=\dd{D}(\dd{R}_+, \dd{R}^2)$ be the set of functions from
  $\dd{R}_+$ to $\dd{R}^2$ that are right continuous on $[0, \infty)$
  and have left limits in $(0, \infty)$.  Assume that the reflection
  matrix $R$ satisfies   condition (\ref{eqn:P    matrix}). Then, $R$
  is  a completely-${\cal S}$ matrix (for a definition see, for
  example, the second 
  paragraph of the introduction
  in \cite{DaiMiyazawa2011a}).
  Therefore, for each $x(\cdot)\in
  \dd{D}^2$, there exists a $z(\cdot)\in \dd{D}^2$ such that $z(\cdot)$,
  together with some 
  $y(\cdot)\in \dd{D}^2$, 
  satisfies the following three conditions 
\begin{eqnarray}
\label{eqn:Skorohod1}
 &&  z(t) = x(t) + R y(t) \ge 0, \qquad t \ge 0, \\
 &&  y(0) =0 \text{ and  each component of $y$ is
   non-decreasing}, \label{eqn:Skorohod2}  \\
 &&  \int_0^\infty z_i(t) dy_i(t)=0, \quad i=1, 2.
\label{eqn:Skorohod3}
\end{eqnarray}
The path $z$ is called an $R$-regulation of $x$. It  is not unique in
general (e.g., see \cite{KozyMandVlad1993}).
Conditions (\ref{eqn:Skorohod1})-(\ref{eqn:Skorohod3}) define the
Skorohod problem associated with $R$. 

Assume conditions (\ref{eqn:P matrix}) and (\ref{eqn:stability 1}) so
that the stationary distribution is uniquely determined.  Let
$Z(\infty)$ be a random vector whose distribution is the stationary
distribution.  
 If there is a lower
semi-continuous function $I(\cdot)$ from $\dd{R}_{+}^{2}$ to
$\dd{R}_{+}$ that satisfies
\begin{eqnarray}
\label{eqn:LDP 1}
 && \limsup_{u \to \infty} \frac 1u \log \dd{P}( Z(\infty) \in u B) \le - \inf_{v \in \ol{B}} I(v),\\
\label{eqn:LDP 2}
 && \liminf_{u \to \infty} \frac 1u \log \dd{P}( Z(\infty) \in u B) \ge - \inf_{v \in B^{\rm o}} I(v)
\end{eqnarray}
for any measurable $B \subset \dd{R}_{+}^{2}$, where $\ol{B}$ and
$B^{\rm o}$ are closure and interior of $B$, respectively, it is said
that the large deviations principle (LDP) holds for the stationary
distribution, and $I(v)$ is called a large deviations rate function.

The LDP for the stationary distribution is verified by
\citet[Theorem~4]{maj96} 
for a multidimensional SRBM whose reflection matrix $R$ is an $\sr{M}$
matrix and the stability condition $R^{-1}\mu<0$ holds. Specializing
his result to two dimensions, we see that, under condition (\ref{eqn:P
  matrix}), $R^{-1}\mu<0$ is equivalent 
to \eqn{stability 1}.
 Furthermore, again specializing to two-dimensional case, for each
 $v\in\dd{R}_+^2$, 
he identifies  the rate function $I(v)$ as  the value function
of the following 
 variational 
problem (VP)
\begin{equation}
  \label{eqn:rate function 1}
  I(v) = \inf_{T\ge 0} \inf_{x\in {\cal H}, z(T)=v}
  \frac{1}{2}\int_0^T \langle \dot x(t)-\mu, \Sigma^{-1}(\dot
  x(t)-\mu)\rangle dt,
\end{equation}
where each $x(\cdot)\in {\cal H}\subset \dd{D}^2$ is an absolutely continuous
function on $\dd{R}_+$ such that its derivative function is locally
square integrable with respect to the Lebesgue measure, and
$z(\cdot)\in \dd{D}^2$ is an $R$-regulation 
of $x(\cdot)$. In general, an $R$-regulation $z$ in
  (\ref{eqn:rate function 1}) is not unique, and the inner inf is
  taken over all such $z$. 
\citet{DupuisRamanan-2002} proved an LDP when 
the $\sr{M}$ matrix condition in \cite{maj96} was relaxed.
However, an LDP for a general SRBM  remains unproven. Despite the lack of such
a general LDP, the  VP (\ref{eqn:rate function 1})  is well
defined  for a general SRBM.  The focus of this paper is  VP
(\ref{eqn:rate function 1}). We use terms
\emph{value function of the VP} and 
\emph{large deviations rate function}  interchangeably
even though the latter term does not make sense when the corresponding
LDP has not been proven.

For two-dimensional SRBMs, \citet{AvramDaiHasenbein01} solved VP
(\ref{eqn:rate function 1}) completely. In particular, they proved
that the value $I(v)$ of \eqn{rate function 1} can be obtained by
restricting paths $x$ to one of the three types. For each $T > 0$ and
$v \in \dd{R}_{+}^{2}$, define $\sr{H}^{(0)}(v, T)$ to be the set of
linear paths $x\in\dd{D}^2$ from $0$ to $v$ such that $x(T)=v$. Define
$\sr{H}^{(1)}(v, T)$ to be the set of continuous paths $x$ that have
two linear segments: during the first segment, a regulated path $z$
of $x$
stays on the  boundary $F_2$ of the state space $\dd{R}^2_+$ and
during the second segment, the regulated 
path travels linearly from the boundary to the end point $v$ at time $T$.
Namely,
\begin{eqnarray*}
  \lefteqn{\sr{H}^{(1)}(v, T) =  \biggl\{  x \in \sr{H}; x(t) = t c \text{ for } 
   t\in [0, s] \text{ and }  x(t) =   \frac {t - s}{T - s} (v - z(s))
   + sc  \text{ for }} \\
 &&  \hspace{11ex}  t \in (s,T] \text{ for some } s \in [0,T)
 \text{ and some } 
   c \in \dd{R}^{2} \text{ with } c_1>0 \text{ and }c_2<0 \biggr\}.
\end{eqnarray*}
Similarly, define
$\sr{H}^{(2)}(v, T)$ to be the set of two segment paths so that
during the first segment, an regulated path $z$ stays on
boundary  $F_1$ of the state space.

Using these sets, let
\begin{eqnarray}
\label{eqn:sub rate 1} 
  I^{(i)}(v) = \inf_{T > 0} \inf_{ x \in \sr{H}^{(i)}(v, T), z(T) = v}
  \frac{1}{2}\int_0^T \langle \dot x(t)-\mu, \Sigma^{-1}(\dot
  x(t)-\mu)\rangle dt,
 \qquad i = 0, 1, 2. 
\end{eqnarray}
\begin{remark} {\rm
\label{rem:}
  We here slightly changed the notation of
  \citet{AvramDaiHasenbein01}. Namely, $I^{(1)}(v)$ and $I^{(2)}(v)$
  are $\tilde{I}^{2}(v)$ and $\tilde{I}^{1}(v)$ of
  \cite{AvramDaiHasenbein01}, respectively, while $I^{(0)}(v)$ is the
  same as $\tilde I^0(v)$ in \cite{AvramDaiHasenbein01}. 
}\end{remark}
   
  We now present the following well known fact (see Theorem 5.1 of \citet{AvramDaiHasenbein01} for its proof).
\begin{lemma} {\rm
\label{lem:line optimality 1}
  For an SRBM that satisfies (\ref{eqn:P matrix}) and
  (\ref{eqn:stability 1}), the rate function $I(v)$ is obtained as 
 \begin{eqnarray}
\label{eqn:rate function 2}
  I(v) = \min(I^{(1)}(v), I^{(2)}(v)), \qquad v \in \dd{R}_{+}^{2}.
\end{eqnarray}
 }\end{lemma}
\begin{remark} {\rm
\label{rem:line optimality 1}
  $I(v)$ is often written as
\begin{eqnarray*}
  I(v) = \min(I^{(0)}(v), I^{(1)}(v), I^{(2)}(v)), \qquad v \in \dd{R}_{+}^{2}.
\end{eqnarray*}
  However, $\sr{H}^{(0)} \subset \sr{H}^{(1)} \cap \sr{H}^{(2)}$, and therefore \eqn{rate function 2} is also valid.
}\end{remark}

\citet{AvramDaiHasenbein01} explicitly obtained expressions for  $I(v)$ in
\eqn{rate function 2} in their Theorem~3.1 and Lemma~6.1. However, it
is difficult to see how the rate function $I(v)$ depends on the 
SRBM data $(\Sigma, \mu, R)$ from their results. In the next section,
we use a geometric 
method to compute $I(v)$ directly. In Section~\ref{sect:ADH
  characterization} we will show that the computation in Section
\ref{sect:Geometrical view} recovers the results in
\cite{AvramDaiHasenbein01}.

\section{Geometrical view}
\label{sect:Geometrical view}
\setnewcounter

Recall that $Z(\infty)$ is the random vector whose distribution is 
 the stationary distribution of the two-dimensional
 SRBM. 
Denote the moment generating function of $Z(\infty)$ by $\varphi$, that is,
\begin{equation}
  \label{eq:momentPhi}
  \varphi(\theta) = \dd{E}\left( e^{\br{\theta,Z(\infty)}}\right),
  \quad \text{ for } \theta \in \dd{R}^{2}
\end{equation}
  as long as $\varphi(\theta)<\infty$. \citet{DaiMiyazawa2011a}
  identified the 
  convergence domain $\sr{D}$, which is defined to be 
\begin{eqnarray*}
  \sr{D} = \mbox{the interior of } \{\theta \in \dd{R}^{2}; \varphi(\theta) < \infty \}.
\end{eqnarray*}
 In this paper, we provide an alternative characterization of the
 convergence domain $\sr{D}$; see \thr{domain 1} below in
 this section. The characterization allows us to compute the value
 functions $I^{(1)}(v)$ and $I^{(2)}(v)$ geometrically for variational
 problem (\ref{eqn:sub rate 1})

\subsection{SRBM data and conditions for existence and stability}
\label{sect:SRBM data}

To describe the convergence domain ${\cal D}$ and its alternative
characterization, we first introduce three functions on $\dd{R}^2$.
They are
\begin{eqnarray}
\label{eqn:gamma +}
 && \gamma(\theta) = -\frac 12 \br{\theta, \Sigma \theta} - \br{\mu, \theta},\\
\label{eqn:gamma i}
 && \gamma_{i}(\theta) = \langle R^i, \theta \rangle= \theta_{1} r_{1i} + \theta_{2} r_{2i}, \qquad i=1,2,
\end{eqnarray}
where $R^i$ is again the $i$th column of the reflection matrix $R$.  These
three functions determine three geometric objects on the plane
$\dd{R}^2$.  Because we assume $\Sigma$ is nondegenerate, the
quadratic equation $\gamma(\theta)=0$ determines an ellipse that
passes through the origin.  Two linear equations $\gamma_1(\theta)=0$
and $\gamma_2(\theta)=0$ determine two lines that pass through the
origin. See \fig{ellipse-ray-stable} for an example.

Let
\begin{eqnarray*}
  p^{(1)} = \left(\begin{array}{c} r_{22} \\ -
      r_{12} \end{array}\right), \qquad p^{(2)} =
  \left(\begin{array}{c} -r_{21}\\ r_{11} \end{array}\right). 
\end{eqnarray*}
One can check that the transpose of $p^{(i)}$ is the $i$th row of the matrix $R^{-1}\det(R)$, where $\det(R)$ is the determinant of $R$. Thus,  $p^{(1)}$ and $p^{(2)}$ are orthogonal to the second and first columns, respectively, of $R$. Note that $p^{(1)} = p^{2}$ and $p^{(2)} = p^{1}$, where  $p^{i}$ is defined in \cite{AvramDaiHasenbein01} for $i=1, 2$. Clearly, line $\gamma_1(\theta)=0$ parallels directional
vector $p^{(2)}$ and line $\gamma_2(\theta)=0$ parallels directional
vector $p^{(1)}$.  It is convenient to assign a direction for each
line and from now on a directional line is called a ray. We assume ray
$\gamma_1(\theta)=0$ points in the direction of $p^{(2)}$ and ray
$\gamma_2(\theta)=0$ points in the direction of $p^{(1)}$.
In \fig{ellipse-ray-stable}, these directions are denoted
by arrows on two lines.

Specifying the three geometric objects in \fig{ellipse-ray-stable} is equivalent to specifying the SRBM data
$(\Sigma, \mu, R)$. 
\begin{figure}[h]
 	\centering
	\includegraphics[height=2.0in]{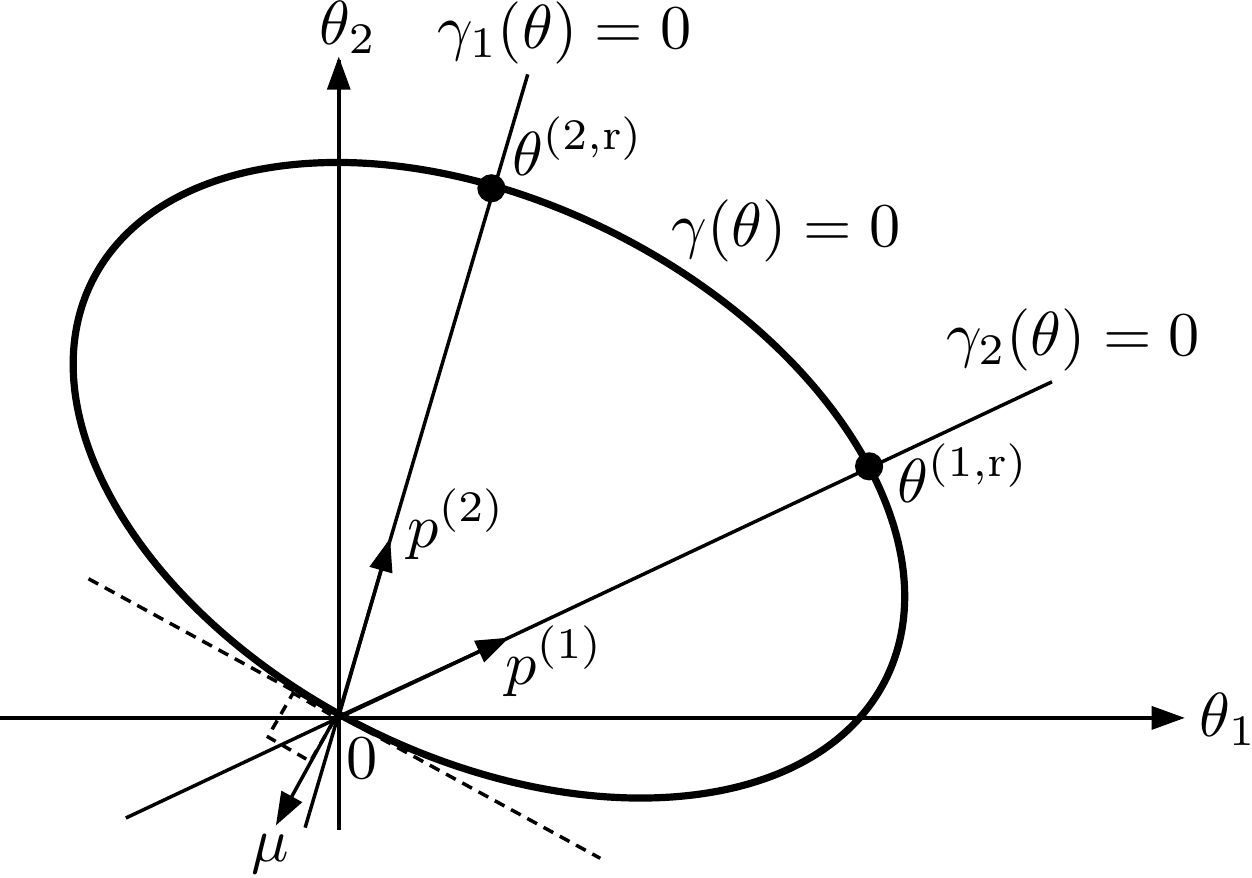} \hspace{5ex}
	\includegraphics[height=2.0in]{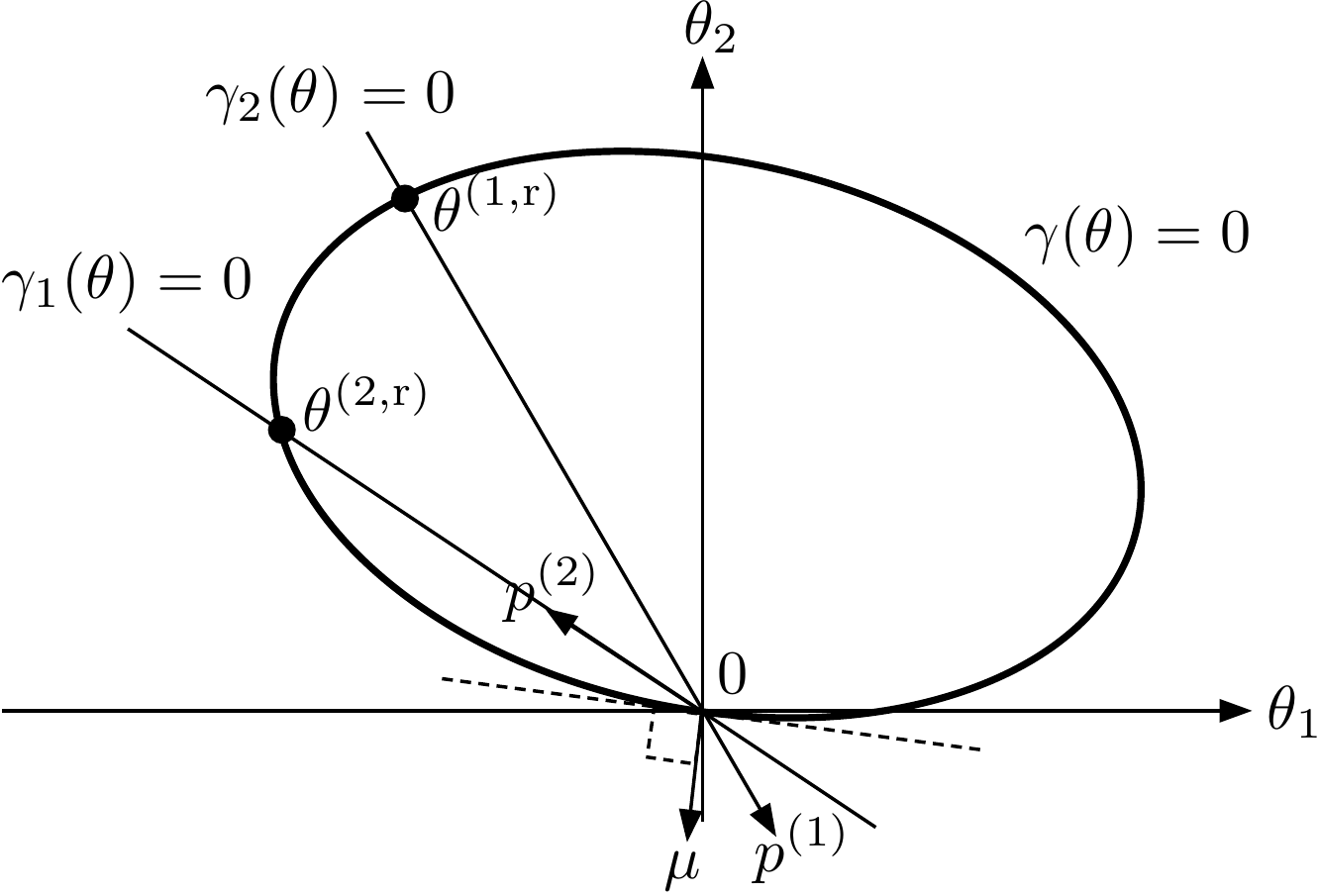} 
	\caption{Specifying the SRBM data $(\Sigma, \mu, R)$ is
          equivalent to specifying an ellipse and  two rays, all
          passing through the origin. The ellipse has normal direction
          $\mu$ at the origin. Ray $\gamma_{3-i}(\theta)=0$ has
          direction $p^{(i)}$ that  intersects with the ellipse at
          a point $\theta^{(i,r)}\neq 0$, $i=1,2$.
          Stability condition (\ref{eqn:stability 2}) is satisfied for
          the left figure and not satisfied for the right figure}
	\label{fig:ellipse-ray-stable}
\end{figure}
\begin{figure}[h]
 	\centering
	\includegraphics[height=2.0in]{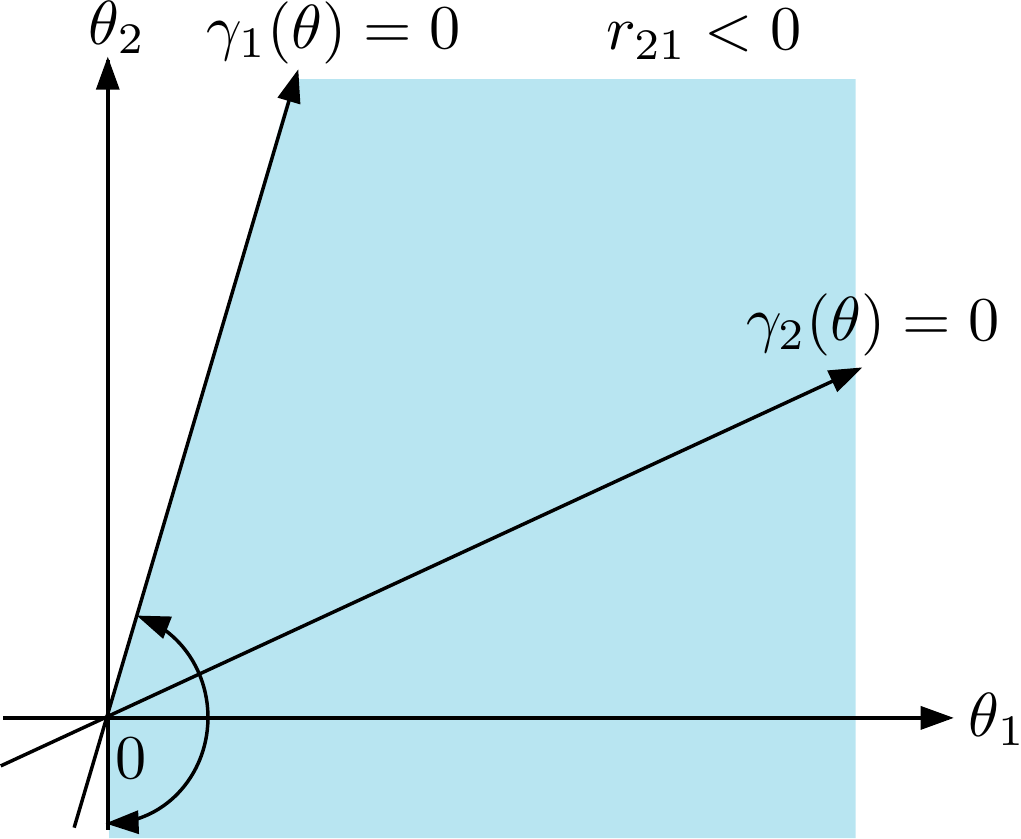} \hspace{5ex}
	\includegraphics[height=2.0in]{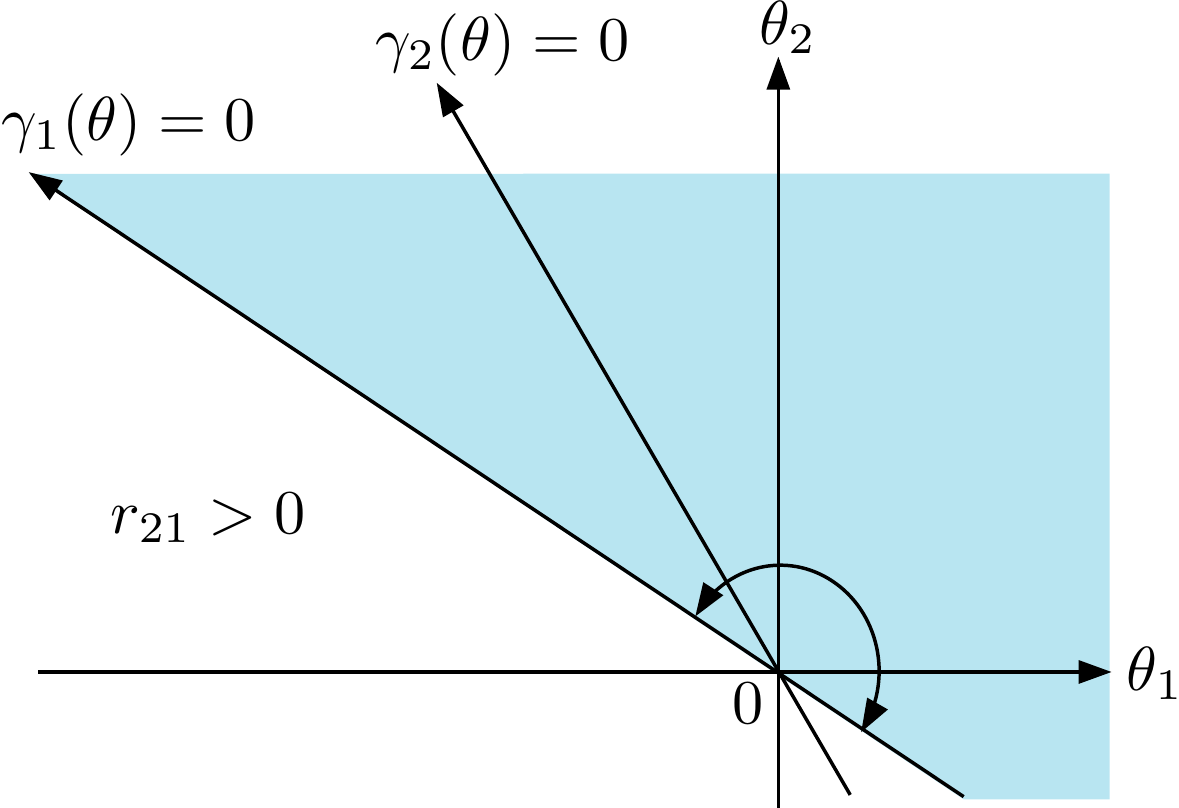} 
	\caption{The ${\cal P}$ matrix condition (\ref{eqn:P
            matrix}) is equivalent to that ray $\gamma_2(\theta)=0$
          must lie in the shaded region} 
	\label{fig:ray-P-matrix}
\end{figure}
A ray passing through the origin can be uniquely determined by its angle
in $(-\pi, \pi]$, measured counter-clockwise against the $\theta_1$
axis; the angle is positive if the ray points above the axis, and is
negative if it points below the axis. Ray one is said to be
\emph{above} (below) ray two if angle one is strictly larger (smaller)
than angle two. Ray one is said to be \emph{on} ray two if
the two angles are equal.

Geometrically, the ${\cal P}$ matrix condition (\ref{eqn:P matrix}) is
equivalent to ray $\gamma_2(\theta)=0$ being below
ray $\gamma_1(\theta)=0$; see \fig{ray-P-matrix} for an
example. 
With the definition of $p^{(i)}$ for $i=1, 2$, the stability condition
\eqn{stability 1} can 
be written as 
\begin{eqnarray}
\label{eqn:stability 2}
  \br{\mu, p^{(1)}} < 0, \quad \br{\mu, p^{(2)}} < 0.
\end{eqnarray}
  As discussed in \cite{DaiMiyazawa2011a}, this is an important
  observation in the geometrical view of the stability condition.
Condition (\ref{eqn:stability 2}) is
equivalent to the angle between vector $\mu$ and each 
ray being more than $\pi/2$. In \fig{ellipse-ray-stable}, 
the left panel corresponds a stable SRBM and the right panel an
unstable SRBM.
Since $\mu$ is the normal direction of the ellipse at the origin, 
condition (\ref{eqn:stability 2}) ensures that ray
$\gamma_i(\theta)=0$ is not tangent 
to the ellipse,  and therefore it must intersect with
the ellipse  at a point other than the origin.

\subsection{Convergence domain and rate function}
\label{sect:Convergence domain}

Let
\begin{displaymath}
    \Gamma = \{ \theta \in \dd{R}^{2}; \gamma(\theta) >  0 \}
\end{displaymath}
denote the interior of the ellipse. Its boundary $\partial \Gamma$ is
the ellipse itself. Since $\gamma(\theta)$ is a convex function,
$\Gamma$ is a convex set. We use $\theta^{(1,\max)}$
to denote  the right-most point on $\partial \Gamma$ and 
$\theta^{(2,\max)}$ the highest point; namely,
\begin{displaymath}
 \theta^{(i,\max)} = \arg \max_{\theta \in \partial \Gamma} \theta_{i},
 \quad i=1, 2.
\end{displaymath}
We also define two subsets $\Gamma_1$ and $\Gamma_2$ of 
$\Gamma$ via
\begin{equation}
  \label{eq:Gammai}
\Gamma_{i} = \{ \theta \in \Gamma; \gamma_{3-i}(\theta) <  0 \}
\quad \text{ for } i=1, 2.
\end{equation}
Clearly, $\Gamma_1$ is the region inside the ellipse and is below ray
$\gamma_2(\theta)=0$ and $\Gamma_2$ is the region inside the ellipse
and is above ray $\gamma_1(\theta)=0$.  Similar to $\partial \Gamma$,
the boundary of $\Gamma_{i}$ is denoted by $\partial \Gamma_{i}$ for
$i=1,2$.  
  Following \cite{DaiMiyazawa2011a}, we use $\theta^{(1, \rm r)}$ to denote the intersection of ray $\gamma_2(\theta)=0$ with the ellipse (other than the origin), and $\theta^{(2, \rm r)}$ the 
intersection of ray  $\gamma_1(\theta)=0$ with the ellipse; see \fig{ellipse-ray-stable}.
  We note that $\theta^{(1,\rm r)}$ is proportional to $p^{(1)}$ and therefore $\theta^{(1,\rm r)}_{1} p^{(1)} = r_{22} \theta^{(1,\rm r)}$. Hence,  $\gamma(\theta^{(1,\rm r)}) = 0$ implies that
  \begin{displaymath}
  \theta^{(1,\rm r)}_{1} = - \frac {2 r_{22} \br{\mu, p^{(1)}}}
  {\br{p^{(1)}, \Sigma p^{(1)}}} \quad  \text{ and } \quad
  \theta^{(1,\rm r)}_{2} =  \frac {2 r_{12} \br{\mu, p^{(1)}}}
  {\br{p^{(1)}, \Sigma p^{(1)}}}. 
  \end{displaymath}
One can similarly compute $\theta^{(2,\rm r)}_1$ and $\theta^{(2,\rm r)}_2$. Thus,
we have 
\begin{equation}
\label{eqn:ray and p 1}
  \theta^{(i,r)} = - \frac {2  \br{\mu, p^{(i)}}}
  {\br{p^{(i)}, \Sigma p^{(i)}}}p^{(i)}, \quad i=1, 2.
\end{equation}
\begin{figure}[tb]
 	\centering
	\includegraphics[height=2.0in]{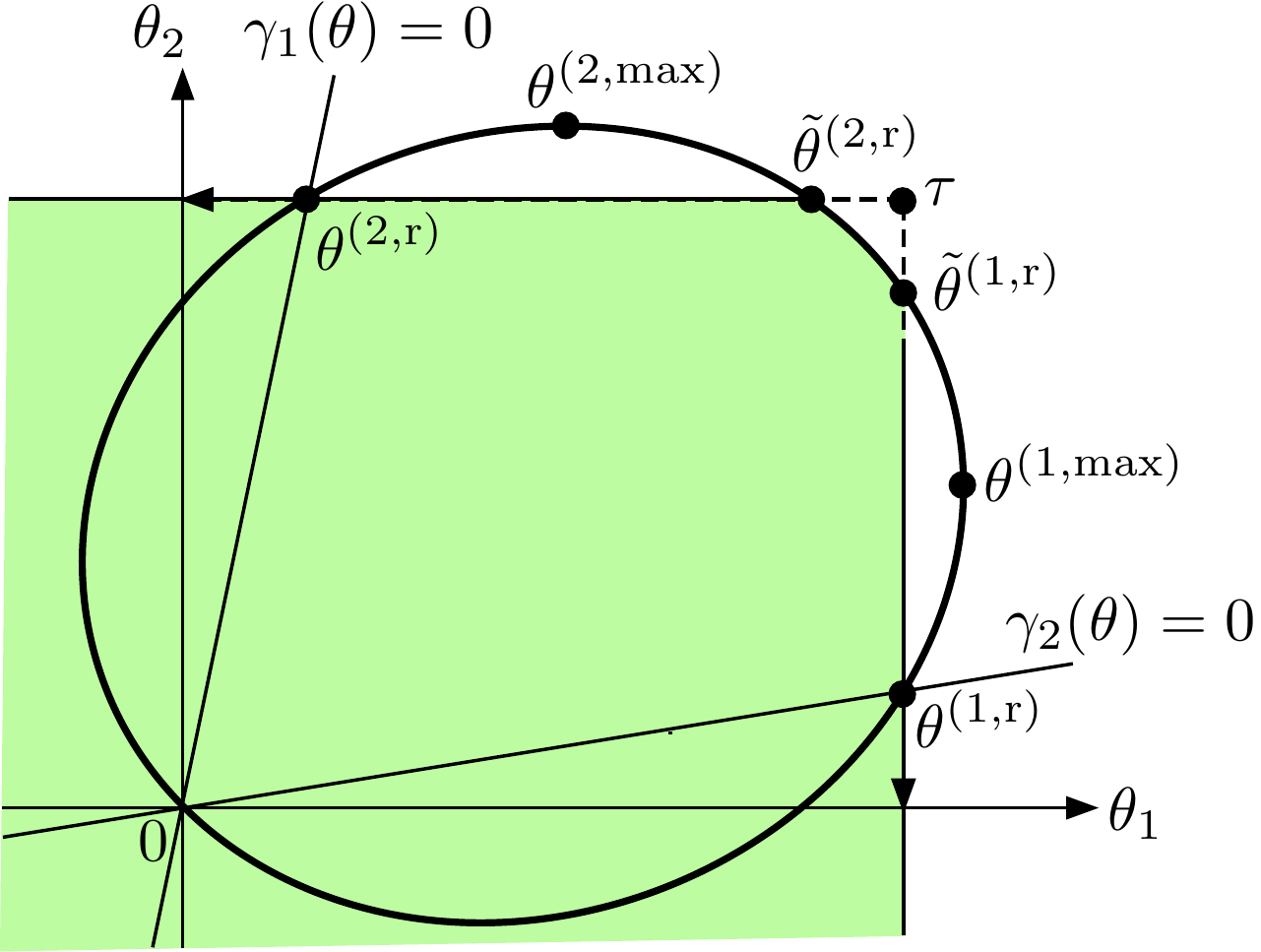} \hspace{5ex}
	\includegraphics[height=2.0in]{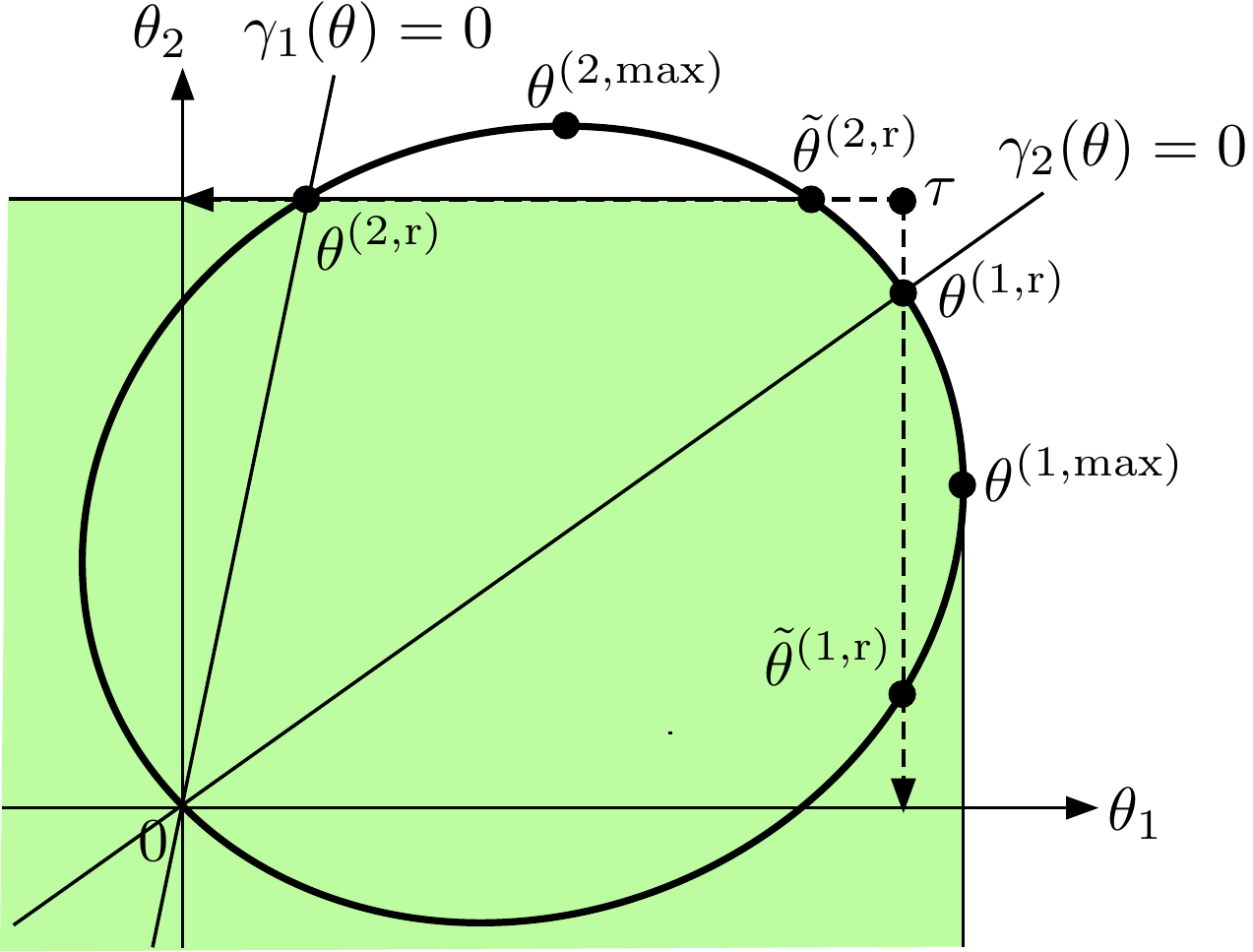} 
	\caption{Three pairs of points:
          $\theta^{(i,r)}$, $\tilde \theta^{(i,r)}$ and
          $\theta^{(i,\max)}$ on the ellipse}
	\label{fig:SixPoints}
 \end{figure}
Condition (\ref{eqn:stability 1}) or condition (\ref{eqn:stability 2}) is
equivalent to  
\begin{equation}
  \label{eq:thetairpositive}
  \theta^{(i, r)}_i >0 \quad \text{ for } i=1, 2.
\end{equation}
  
Now we define
a  pair of points $\tilde \theta^{(1,\rm r)}$ and $\tilde
\theta^{(2,\rm r)}$ on $\partial \Gamma$. The point $\tilde
\theta^{(1,\rm r)}$ is defined to be the  
``symmetry''  of $\theta^{(1, r)}$  on $\partial \Gamma$.
When $\theta^{(1,\rm r)}=\theta^{(1,\max)}$, we define $\tilde \theta^{(1,\rm r)}$
to be $\theta^{(1,\max)}$ itself.  When 
$\theta^{(1,\rm r)}\neq \theta^{(1,\max)}$, we define $\tilde
\theta^{(1,\rm r)}$ to be the unique point on $\partial \Gamma$ such that
\begin{equation}
  \label{eq:thetaSymmetry}
\tilde {\theta}^{(1,r)}_1= \theta^{(1,\rm r)}_1 \quad \text{ and } \quad \tilde
{\theta}^{(1,r)}_{2}\not = \theta^{(1,\rm r)}_{2}.
\end{equation}
Similarly, we can define $\tilde \theta^{(2,\rm r)}$. 
See \fig{SixPoints} for examples of three pairs of points
$\theta^{(i,r)}$, $\tilde \theta^{(i,r)}$ and $\theta^{(i,\max)}$,
$i=1, 2$. As illustrated in the figure,  $\tilde \theta^{(1,\rm r)}$ can
be higher or lower than 
$\theta^{(1,\rm r)}$.

\begin{figure}[tb]
 	\centering
	\includegraphics[height=2.5in]{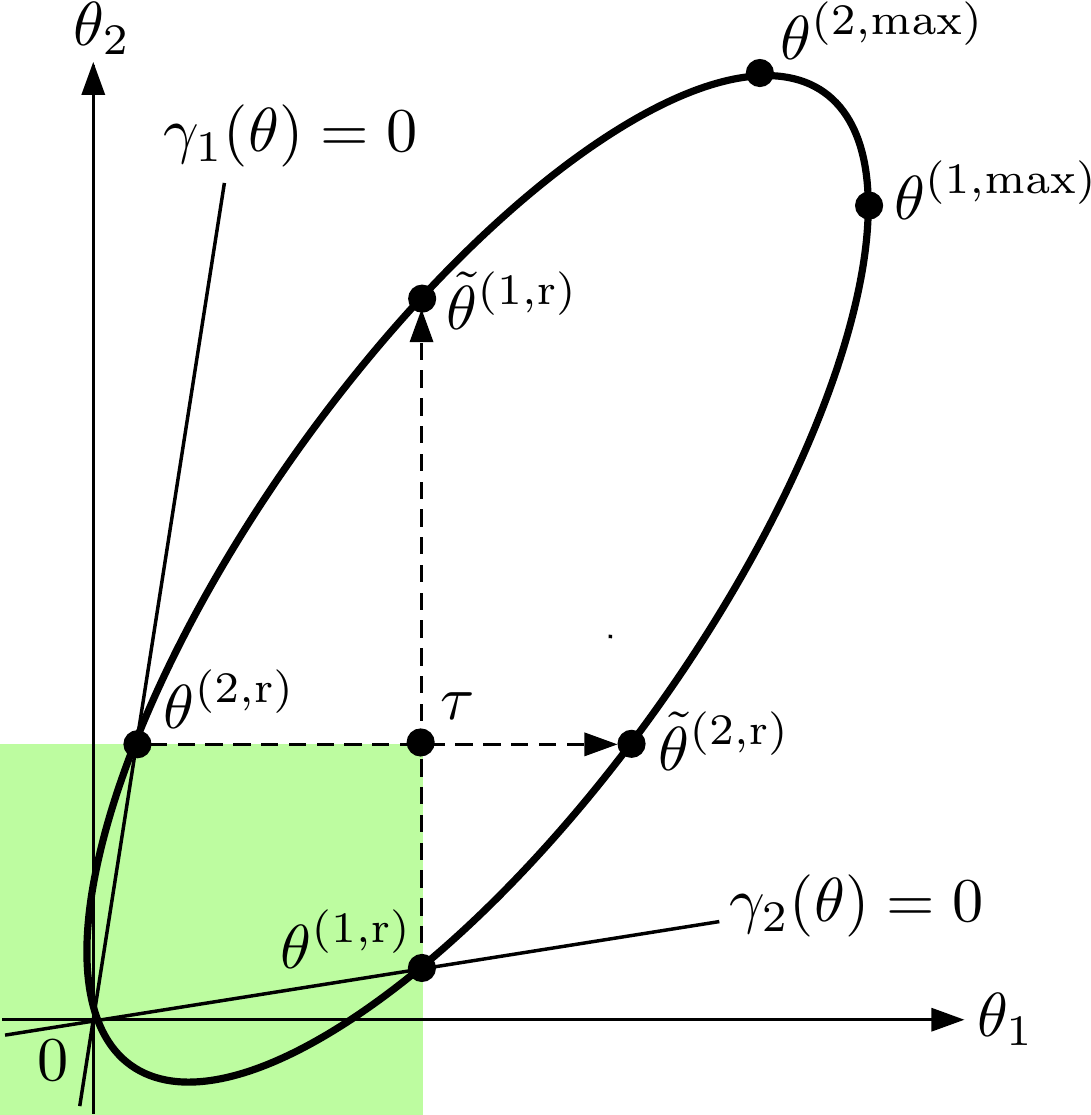} \hspace{5ex}
	\includegraphics[height=2.5in]{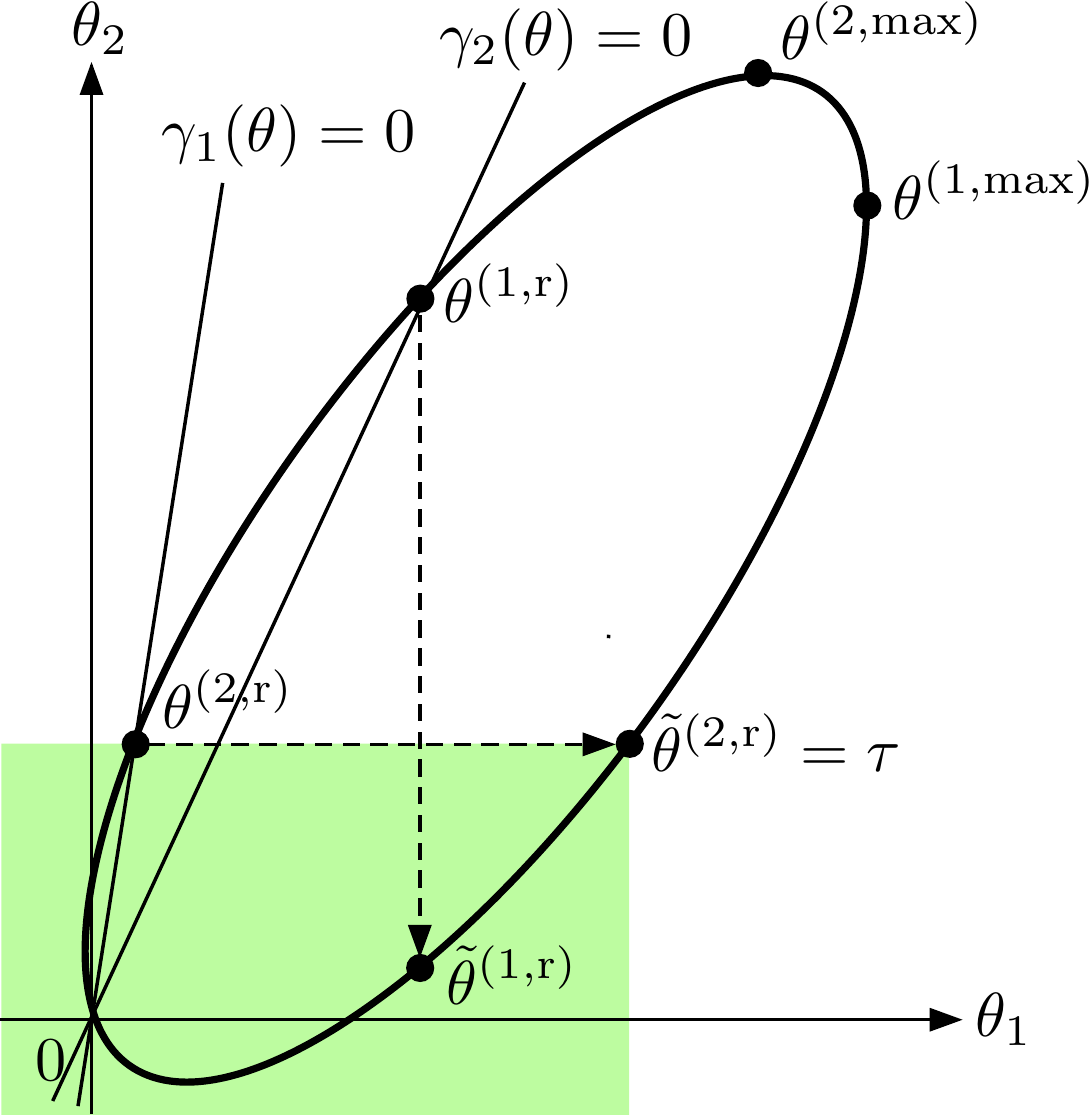} 
	\caption{An example in which both
          $\theta^{(1,\rm r)}$ and $\tilde \theta^{(1,\rm r)}$ are lower than
          $\theta^{(1,\max)}$}
	\label{fig:boththeta1rlowerthantheta1max}
 \end{figure}
The following is an important condition for the SRBM data:
\begin{equation}
  \label{eq:theta1maxinGamma1}
  \theta^{(1,\max)} \in \partial \Gamma_1
\end{equation}
In the right panel of \fig{SixPoints}, condition
(\ref{eq:theta1maxinGamma1}) is satisfied, while in the left panel,
the condition is not satisfied. As it will be explained in
Section~\ref{sect:ADH characterization}, condition
(\ref{eq:theta1maxinGamma1}) is equivalent to the fact that the face
$F_2= \{x \in \dd{R}^2_+; x_2=0\}$ is \emph{not reflective}, an
important term introduced in \cite{AvramDaiHasenbein01}. Its meaning is explained in \cite[page 264]{AvramDaiHasenbein01};
``When $F_i$ is not
reflective, \ldots, the face $F_i$ has no boundary influence on
solutions to the VP [in (\ref{eqn:rate function 1})] for any $v\in 
\dd{R}^2_+$''.

Under condition (\ref{eq:theta1maxinGamma1}), $\tilde\theta^{(1,\rm r)}$
is always at most as high as $\theta^{(1,\rm r)}$.
 Readers are warned  that condition
(\ref{eq:theta1maxinGamma1}) is \emph{not} equivalent to 
\begin{equation}
  \label{eq:theta1rlesstheta1max}
  \theta^{(1,\rm r)}_2 \ge \theta^{(1,\max)}_2.
\end{equation}
The right panel of \fig{boththeta1rlowerthantheta1max}
gives an example that satisfies condition
(\ref{eq:theta1maxinGamma1}), but not condition
(\ref{eq:theta1rlesstheta1max}). Conditions
(\ref{eq:theta1maxinGamma1}) and 
(\ref{eq:theta1rlesstheta1max}) are equivalent if and only if 
$\theta^{(1,\rm r)}$ or $\tilde \theta^{(1,\rm r)}$ is on the right side of
$\theta^{(2,\max)}$, namely, 
\begin{equation}
  \label{eq:theta1rrighttheta2max}
\theta^{(1,\rm r)}_1=\tilde \theta^{(1,\rm r)}_1 \ge \theta^{(2,\max)}_1.
\end{equation}

\begin{figure}[h]
 	\centering
	\includegraphics[height=2.2in]{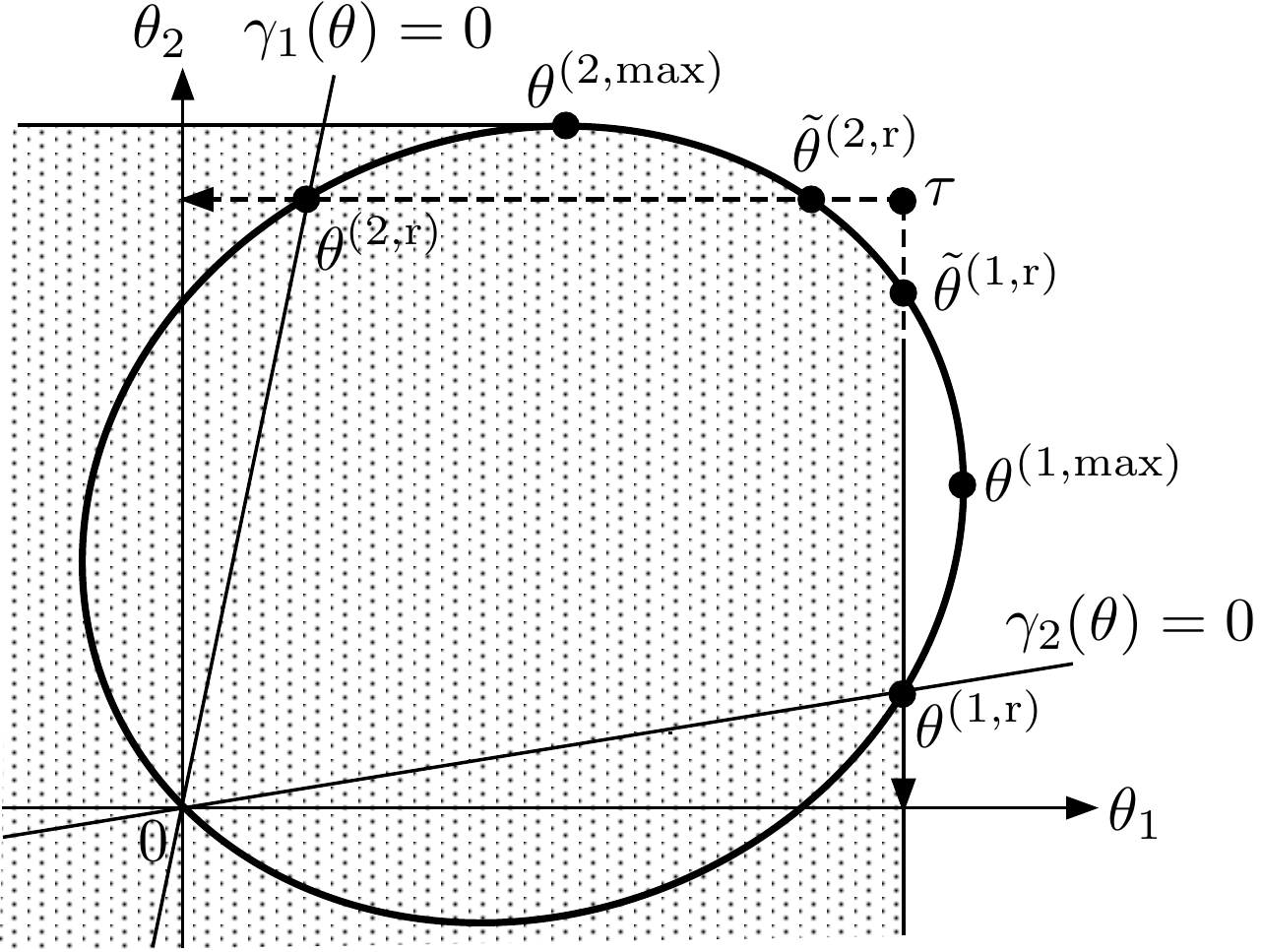} \hspace{5ex}
	\includegraphics[height=2.2in]{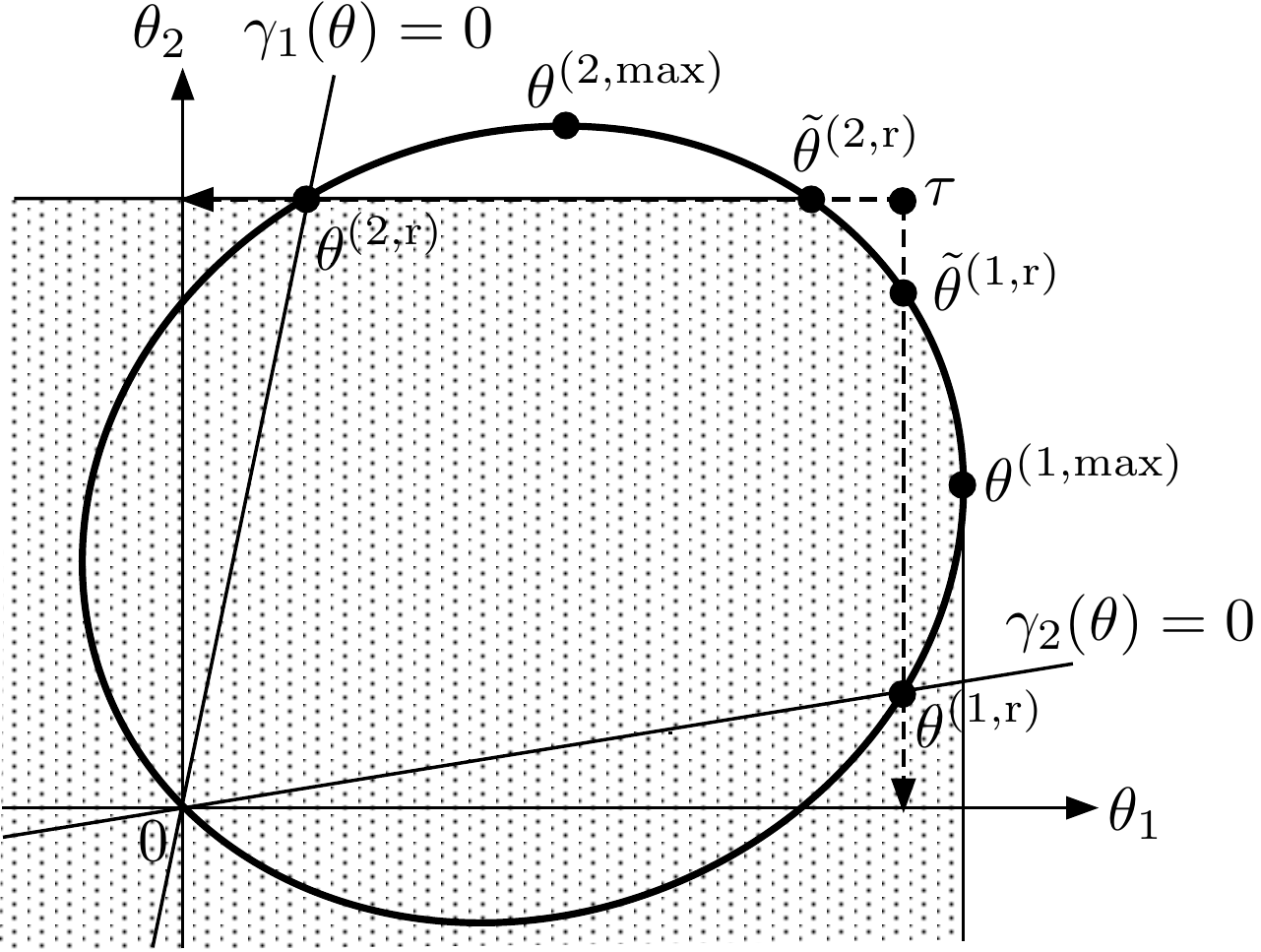} 
	\caption{An example of ${\cal D}^{(1)}$ (left) and ${\cal
            D}^{(2)}$ (right)}
	\label{fig:D1D2CategoryI}
\end{figure}
Now we define two subsets $ \sr{D}^{(1)}$ and $ \sr{D}^{(2)}$ of
$\dd{R}^{2}$.  For this, the following points on the ellipse $\partial \Gamma$ are convenient:
\begin{eqnarray*}
  \theta^{(i,\Gamma)} = \left\{\begin{array}{ll}
  \theta^{(i,\rm r)}, \quad & \mbox{if } \theta^{(i,\max)} \not\in \partial \Gamma_{i},\\
  \theta^{(i,\max)}, \quad & \mbox{if } \theta^{(i,\max)} \in \partial\Gamma_{i},
  \end{array} \right.
  \qquad i =1,2.
\end{eqnarray*}
  We define
  \begin{equation}
    \label{eq:Di}
  \sr{D}^{(i)} =
\Bigl\{ \theta \in \dd{R}^2; \text{ there exists a }
  \theta'\in \Gamma \text{ such that }  \theta < \theta' \text{ and
  }\theta_{i}' < 
  \theta^{(i,\Gamma)}_{i}\Bigr\}, \quad i =1,2.
  \end{equation}
 In what follows, we mainly consider $\sr{D}^{(1)}$ because results for $\sr{D}^{(2)}$ are symmetric. Note that the definition of $\sr{D}^{(1)}$ depends on whether
condition (\ref{eq:theta1maxinGamma1}) is satisfied.  We have the
following lemma. 

\begin{lemma}
\label{lem:D1characterization}
The domain  $\sr{D}^{(1)}$ has the following form:
\begin{equation}
  \label{eq:D1characterization}
  \sr{D}^{(1)}=
  \begin{cases}
\bigl     \{ \theta\in \dd{R}^2; \theta<
    \tilde \theta^{(1,\rm r)}\bigr\}
    & \mbox{if } \theta^{(1,\max)} \not\in \partial\Gamma_{1} \mbox{ and if } \theta^{(1,\rm r)}_{1} \le  \theta^{(2,\max)}_{1}, \\
 \Gamma_{\max}\cap\bigl\{ \theta\in \dd{R}^2; \theta_1<
    \tilde \theta^{(1,\rm r)}_1\bigr\}
 & \mbox{if } \theta^{(1,\max)} \not\in \partial\Gamma_{1} \mbox{ and if } \theta^{(1,\rm r)}_{1} > \theta^{(2,\max)}_{1}, \\
  \Gamma_{\max},
    & \text{if } \theta^{(1,\max)} \in \partial\Gamma_{1},
  \end{cases}
\end{equation}
where
\begin{equation}
  \label{eq:Gammamax}
 \Gamma_{\max} =\{ \theta\in \dd{R}^2; \text{ there exists a $\theta'\in \Gamma
   $  such that $\theta< \theta'$}\}.
\end{equation}
\end{lemma}
\begin{proof}
Assume that  condition (\ref{eq:theta1maxinGamma1})
is satisfied. Then, $\theta^{(1,\Gamma)} = \theta^{(1,\max)}$, and therefore $\sr{D}^{(1)} = \Gamma_{\max}$ (see the right panel of \fig{D1D2CategoryI}). We next assume that (\ref{eq:theta1maxinGamma1})
is not satisfied. Then, $\tilde \theta^{(1,\rm r)}$ higher than $\theta^{(1,\rm r)}$.
 If $\tilde \theta^{(1,\rm r)}$ is on the left side of
$\theta^{(2,\max)}$, all points $\theta \in \sr{D}^{(1)}$ are dominated by $\tilde \theta^{(1,\rm r)}$; see the left panel of
\fig{boththeta1rlowerthantheta1max}. Otherwise,
$\tilde\theta^{(1,\rm r)}$ must be higher than $\theta^{(1,\max)}$, and 
$\tilde \theta^{(1,\rm r)}$ is the lowest point on $\partial \Gamma
\cap\partial \sr{D}^{(1)}$; see the left panel of \fig{D1D2CategoryI}. Thus, \eq{D1characterization} is obtained.
\end{proof}

\fig{D1D2CategoryI} illustrates the domains
$\sr{D}^{(1)}$ and $\sr{D}^{(2)}$.
We are ready to present the first theorem on the value function $I^{(i)}(v)$.
It will proved in Section
 \ref{sect:ADH characterization}. 
\begin{theorem} {\rm
\label{thr:geometric view 1}
Assume  conditions (\ref{eqn:P matrix})  and \eqn{stability 1}.
For $v \in \dd{R}^{2}_+$, 
\begin{eqnarray}
\label{eqn:geometric view i}
  I^{(i)}(v) = \sup \{\br{v,\theta}; \theta \in \sr{D}^{(i)} \}, \qquad i=1,2.
\end{eqnarray}

}\end{theorem}

\begin{corollary}
  The value function $I^{(i)}(v)$ is a convex function of $v\in \dd{R}^2_+$.
\end{corollary}

\begin{remark} {\rm
\label{rem:geometric view 1b}
Equality  \eqn{geometric view i} corresponds to (44) in Theorem 5.1 of \cite{Miyazawa09} for the two-dimensional reflecting random walk studied in \cite{BorovkovMogulski01}.
}\end{remark}

\begin{figure}[h]
 	\centering
	\includegraphics[height=1.9in]{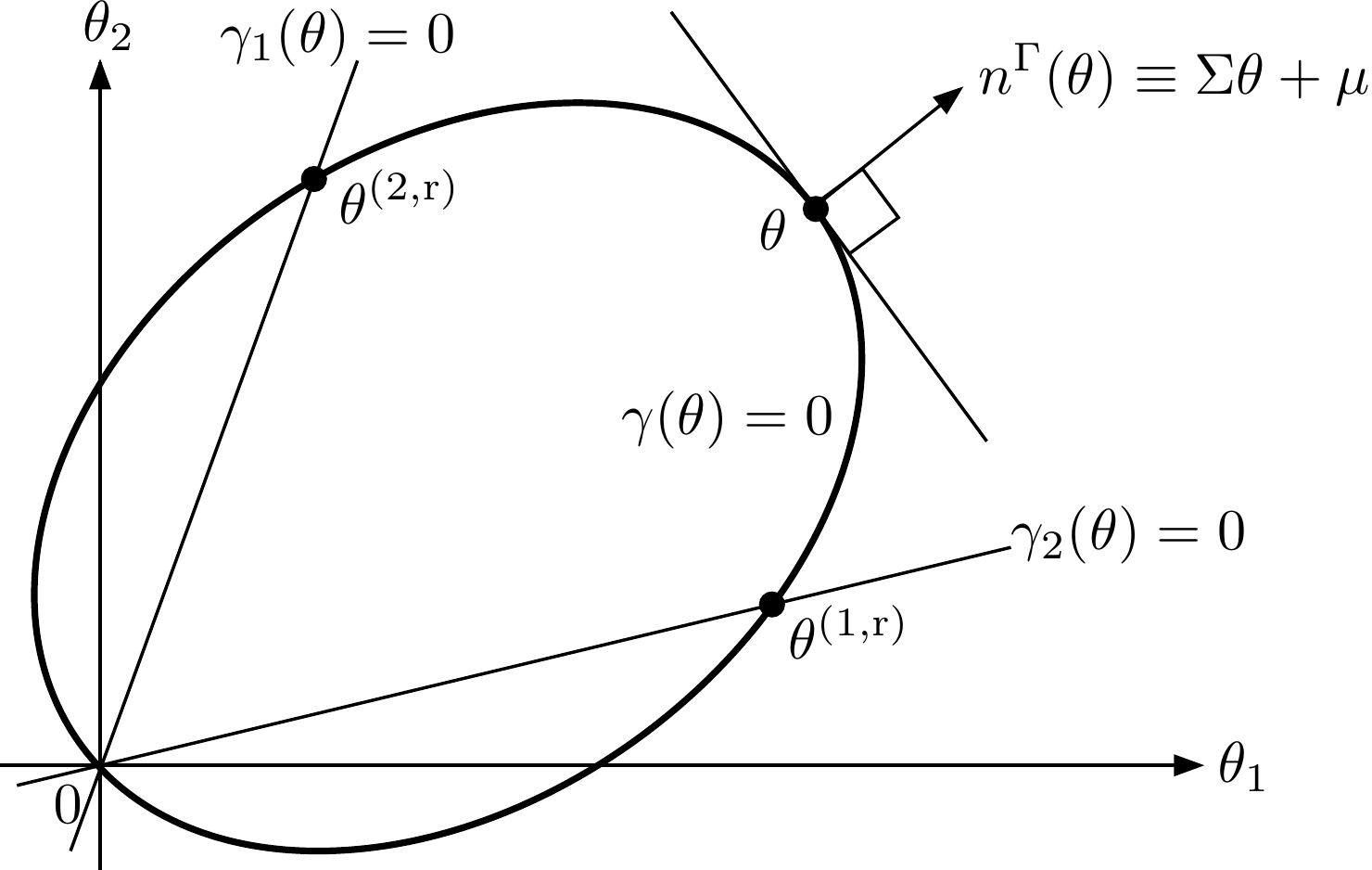}
	\caption{The normal vector $n^\Gamma(\theta)$
          that is orthogonal to the tangent line of the ellipse at
          $\theta\in \partial \Gamma$}
	\label{fig:ellipse-ray-normal}
\end{figure}
For each $\theta\in \dd{R}^2$, define
\begin{eqnarray}
\label{eqn:orthogonal to theta}
  {\rm n}^{\Gamma}(\theta) = \Sigma \theta + \mu.
\end{eqnarray}
It is easy to see that ${\rm n}^{\Gamma}(\theta)$ is the outward normal
direction of the ellipse at $\theta\in \partial \Gamma$. Clearly, 
${\rm n}^{\Gamma}(\theta)$ is 
orthogonal to the tangent line of the ellipse $\partial \Gamma$ at
$\theta$. The normal vector  is 
  illustrated in \fig{ellipse-ray-normal}. 
We now explain how to evaluate 
\begin{equation}
  \label{eq:supinGamma}
  \sup\{\langle v, \theta\rangle; \theta \in \Gamma \}
\end{equation}
that will be used Lemma \ref{lem:geometric} below.

Recall the normal vector $n^\Gamma(\theta)$ for a point
$\theta\in \partial \Gamma$.
The first component of $n^\Gamma(\theta^{(1,\max)})$ is zero  and
the second component of 
$n^\Gamma(\theta^{(2,\max)})$ is zero. Namely, the normal at
$\theta^{(1,\max)}$ is vertical and the normal at $\theta^{(2,\max)}$
is horizontal. Thus, for any $v\in \dd{R}^2_+$, there exists a unique
$\theta^{(v,\Gamma)}\in \partial \Gamma$ in the segment between
$\theta^{(1,\max)}$ 
and $\theta^{(2,\max)}$ such that $n^\Gamma(\theta^{(v,\Gamma)})$
parallels  $v$. Since $\br{v, \theta}$ is the distance from the
  origin to the point that  is
  projected from $\theta$ onto ray $\{t v, t\ge 0\}$, we have
\begin{equation}
  \label{eq:thetavGamma}
  \langle v,\theta^{(v,\Gamma)}\rangle = \sup \{ \br{v, \theta};
  {\theta\in \Gamma}\}.
\end{equation}

 Since $\sr{D}^{(i)}$ is convex and upper bounded,  
the supremum 
\begin{displaymath}
  \sup\{\langle
 v, \theta\rangle; \theta\in \sr{D}^{(i)}\}   
\end{displaymath}
is finite and  is achieved at an extreme
point $\theta^{(v, \sr{D}^{(i)})}$ on the boundary of
$\sr{D}^{(i)}$. Namely, 
\begin{equation}
  \sup\{\langle
 v, \theta\rangle; \theta\in \sr{D}^{(i)}\} = \langle v, \theta^{(v,
   \sr{D}^{(i)})} \rangle.
\label{eq:thetavD}
\end{equation}

\thr{geometric view 1} and Lemma 
\ref{lem:geometric} below give a geometric
interpretation of $I^{(1)}(v)$.
\begin{lemma}\label{lem:geometric}
For each $v\in \dd{R}^2_+$,
\begin{eqnarray}
  \label{eq:gvI1}
\lefteqn{\sup\{\langle
 v, \theta\rangle; \theta\in \sr{D}^{(1)}\} }\\
&=&  \begin{cases}
\langle v, \tilde \theta^{(1,\rm r)}\rangle
 & \text{if $  \theta^{(1,\max)} \not \in \partial \Gamma_1 $, $\tilde
   \theta^{(1,\rm r)}_1 \le  \theta^{(2,\max)}_1$}, \\
\langle v, \tilde \theta^{(1,\rm r)}\rangle  
 & \text{if $  \theta^{(1,\max)} \not \in \partial \Gamma_1 $, $\tilde
   \theta^{(1,\rm r)}_1 >  \theta^{(2,\max)}_1$, $v$ is below
or on $n^\Gamma(\tilde \theta^{(1,\rm r)})$ }, \\
  \sup\{\langle v, \theta\rangle; \theta \in \Gamma \} 
 & \text{if $  \theta^{(1,\max)} \not \in \partial \Gamma_1 $, $\tilde
   \theta^{(1,\rm r)}_1 >  \theta^{(2,\max)}_1$, $v$ is above
   $n^\Gamma(\tilde \theta^{(1,\rm r)})$ },\\
  \sup\{\langle v, \theta\rangle; \theta \in \Gamma \} & \text{if }
  \theta^{(1,\max)} \in \partial \Gamma_1.
  \end{cases}
\nonumber
\end{eqnarray}
\end{lemma}
\begin{proof}
 Lemma \ref{lem:D1characterization} characterizes the set
$\sr{D}^{(1)}$ in three separate cases. We now prove the lemma for
each of the three cases; (i) $\theta^{(1,\max)} \not \in \partial \Gamma_1 $, $\tilde \theta^{(1,\rm r)}_1 \le  \theta^{(2,\max)}_1$, (ii) $\theta^{(1,\max)} \not \in \partial \Gamma_1 $, $\tilde \theta^{(1,\rm r)}_1 >  \theta^{(2,\max)}_1$ and (iii) $\theta^{(1,\max)} \in \partial \Gamma_1 $. In the first case, 
$\theta^{(v, \sr{D}^{(1)})}$ in (\ref{eq:thetavD}) is equal to $\tilde
\theta^{(1,r)}$ because the boundary of $\sr{D}^{(1)}$ consists of two
straight lines with $\tilde\theta^{(1,r)}$ as the unique extreme
point. 
This proves (\ref{eq:gvI1}) for the first case. For the third case,
(\ref{eq:gvI1}) holds because $\sr{D}^{(1)}=\Gamma_{\max}$ and one can
argue that 
\begin{displaymath}
  \langle v,\theta^{(v,\Gamma)}\rangle = \sup \{ \br{v, \theta};
  {\theta\in \Gamma_{\max}}\}
\end{displaymath}
just as proving (\ref{eq:thetavGamma}).
Now consider the second case, which is illustrated in
the left panel of Figure \ref{fig:D1D2CategoryI}.
 One can check that 
\begin{equation}
  \label{eq:thatavGammalesstildetheta}
  \langle v, \theta^{(v, \Gamma)} \rangle \le 
  \langle v, \tilde \theta^{(1, r)} \rangle
\end{equation}
if and only if $v$ is below or on $n^\Gamma(\tilde
\theta^{(1,\rm r)})$. Therefore,
\begin{equation}
  \label{eq:case2}
   \theta^{(v, \sr{D}^{(1)})}=
   \begin{cases}
     \tilde \theta^{(1,r)}  & \text{ if $v$ is below or on $n^\Gamma(\tilde
\theta^{(1,\rm r)})$}, \\
\tilde \theta^{(v,\Gamma)} 
 & \text{ if $v$ is above $n^\Gamma(\tilde
\theta^{(1,\rm r)})$}, 
   \end{cases}
\end{equation}
from which (\ref{eq:gvI1}) holds for the two subcases of the second
case.
\end{proof}

The following characterization of the convergence domain $\sr{D}$
is in a form
different from the one in  \cite{DaiMiyazawa2011a}.
See \fig{domain 1} for an illustration of a convergence domain.
\begin{figure}[h]
 	\centering
	\includegraphics[height=2.0in]{D-Category-I_up} \hspace{5ex}
	\includegraphics[height=2.2in]{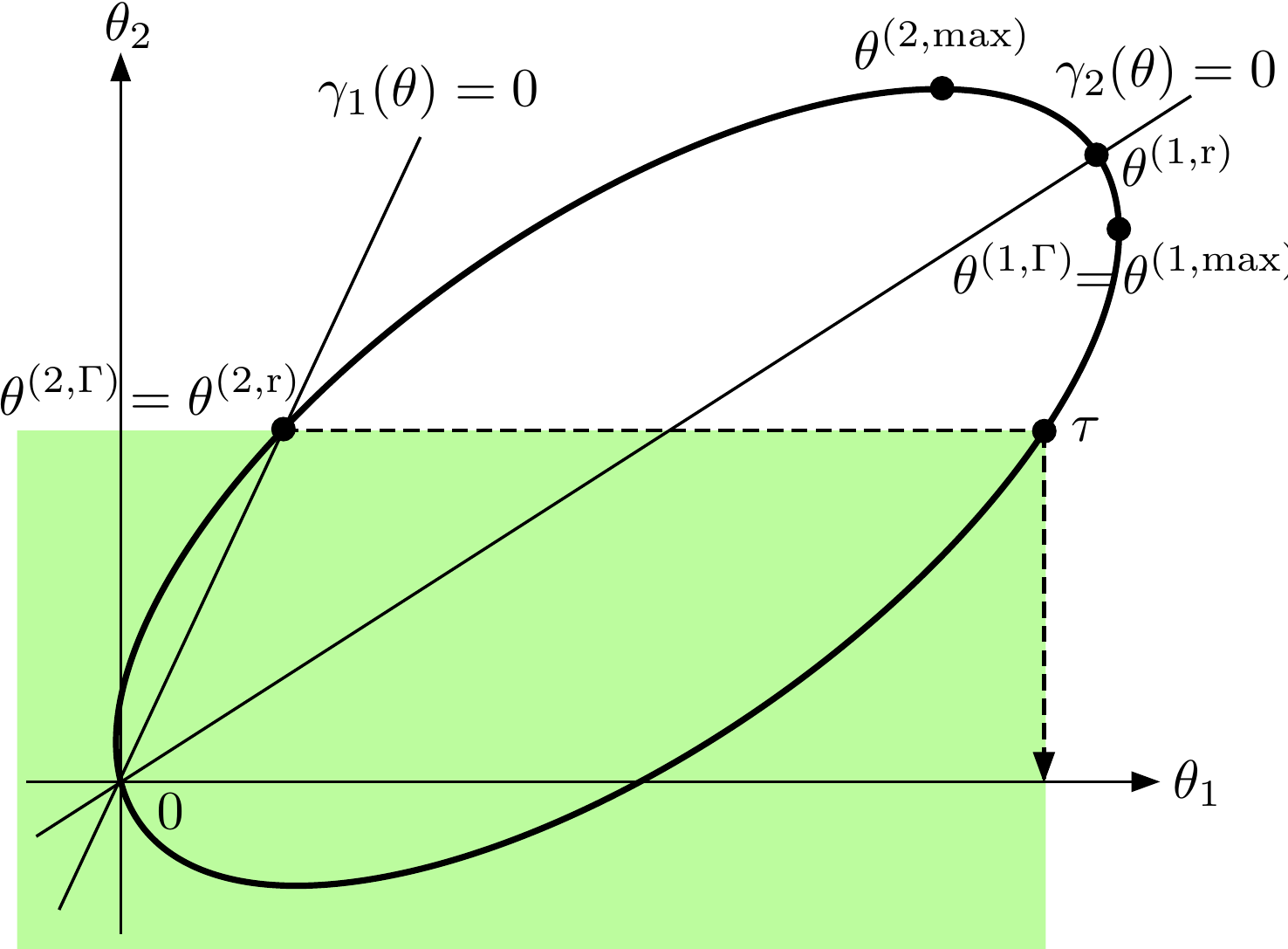} 
	\caption{$\sr{D}=\sr{D}^{(1)}\cap \sr{D}^{(2)}$ (left), where
          $\sr{D}^{(1)}$ and $\sr{D}^{(2)}$ are in 
          \fig{D1D2CategoryI}, and 
          $\sr{D}=\sr{D}^{(2)}$ (right)}
	\label{fig:domain 1}
\end{figure}
Define
\begin{eqnarray}
\label{eqn:tau i}
  \tau_{i} = \sup\{ \theta_{i} > 0; \theta \in \sr{D}_1\cap
  \sr{D}_2\}, \qquad i=1,2. 
\end{eqnarray} 

\begin{theorem}
\label{thr:domain 1}
Assume  conditions \eqn{P matrix} and \eqn{stability 1}. Then, 
(a)
\begin{eqnarray}
\label{eqn:D 1}
  \sr{D} = \sr{D}^{(1)} \cap \sr{D}^{(2)}.
\end{eqnarray}  
(b)
 The large deviations rate function $I$ is obtained as
\begin{eqnarray}
\label{eqn:geometric view 1}
  I(v) = \left\{\begin{array}{ll}
  \min( \brb{v, \tilde{\theta}^{(1, \rm r)}}, \brb{v,\tilde{\theta}^{(2, \rm r)}}), \quad & \tau \in \Gamma,\\
  \sup \{ \br{v, \theta}; \theta \in \sr{D} \}, \quad & \tau \not\in \Gamma,
  \end{array} \right.
\end{eqnarray}
  In particular,
\begin{eqnarray}
\label{eqn:I coordinate 1}
  I({\rm e}^{(i)}) = \tau_{i}, \qquad i=1,2,
\end{eqnarray}
where 
\begin{displaymath}
  {\rm e}^{(1)}=
  \begin{pmatrix}
    1 \\0
  \end{pmatrix}
\quad \text{ and } \quad
{\rm e}^{(2)} =
\begin{pmatrix}
  0 \\ 1
\end{pmatrix}.
\end{displaymath}
\end{theorem}
\begin{remark} {\rm
\label{rem:domain 1}
  Theorems \thrt{geometric view 1} and \thrt{domain 1} provide a
  geometrical interpretation for Theorem 3.1 of
  \cite{AvramDaiHasenbein01}. The latter is algebraic. This
  interpretation does not mean that the value function $I(v)$ is
  analytically easier to compute. However, it enables us to see how 
  $I(v)$ is changing with $v$ and how $I(v)$ is influenced by the
  primitive data through the domain $\sr{D}$ and 
  two points $\theta^{(1,\ray)}$ and $\theta^{(2,\ray)}$. 
}\end{remark}

\fig{rate function 1} gives an illustration of the rate function given in~(\ref{eqn:geometric view 1}). The following corollary is immediate.
 Inequality (\ref{eqn:I lower bound}) below is also proved in \cite[Section
 8]{DaiMiyazawa2011a}.

 \begin{figure}[h]
 	\centering
	\includegraphics[height=2.4in]{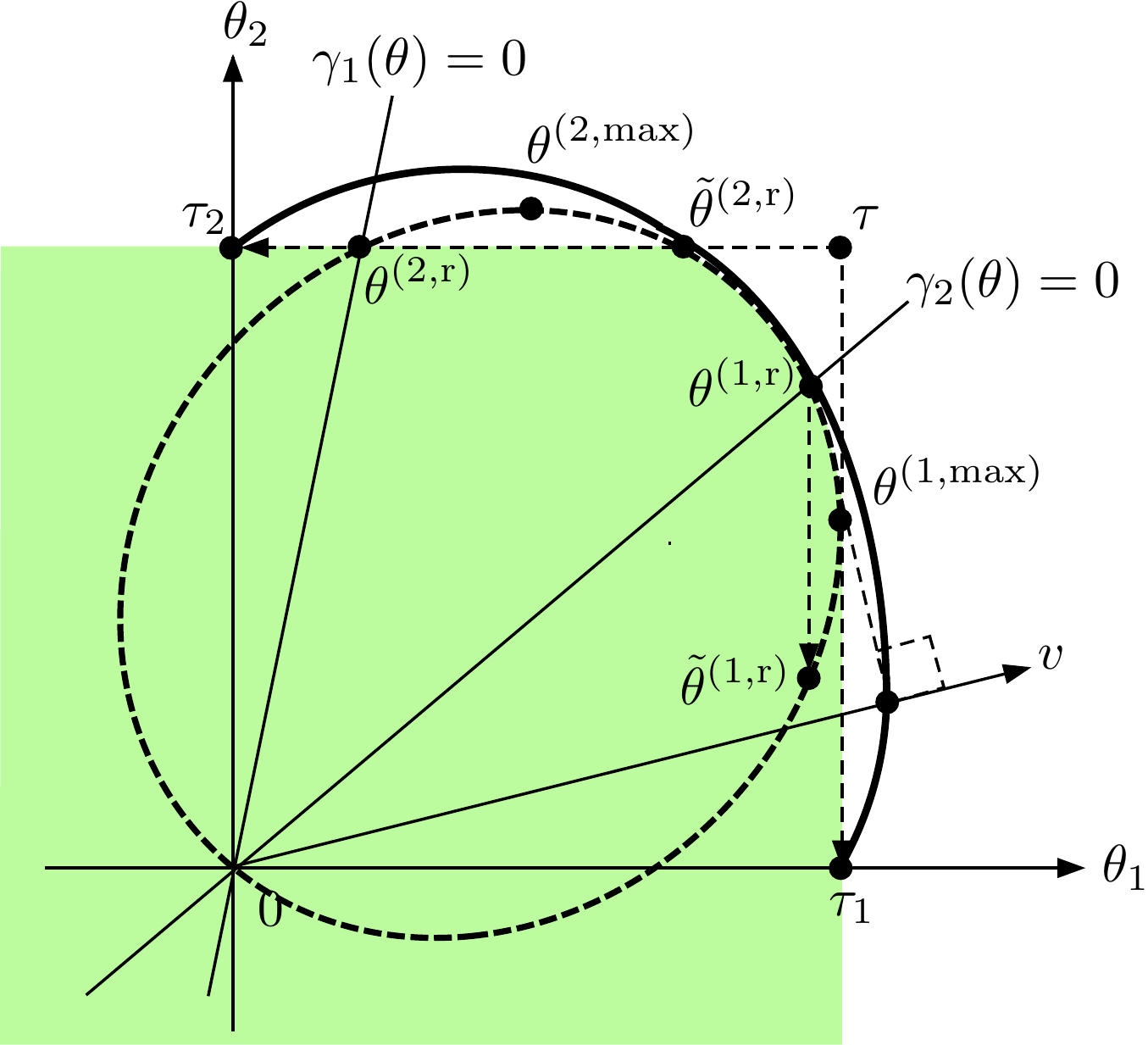} \hspace{5ex}
	\includegraphics[height=2.4in]{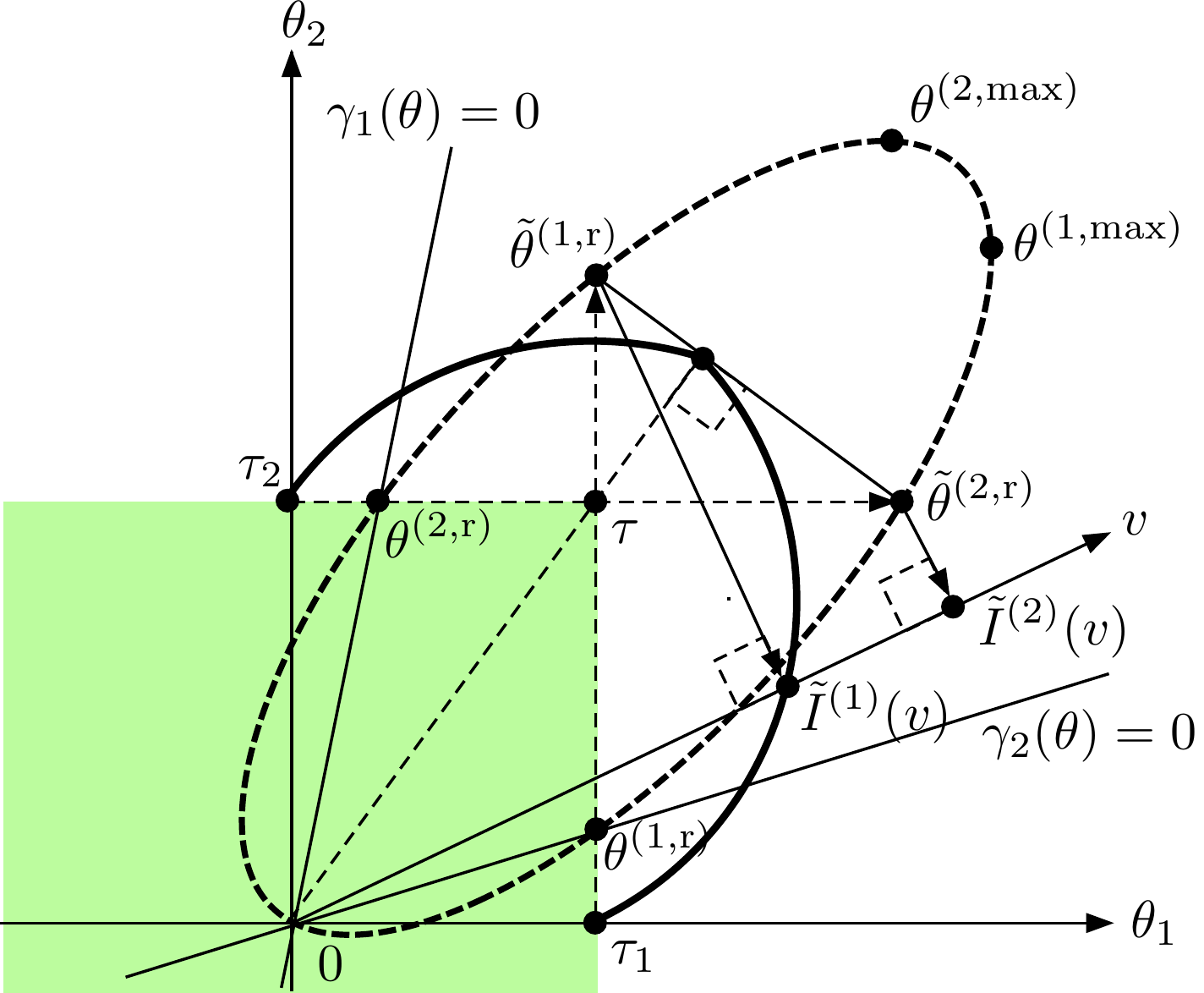}
	\caption{The shapes of the rate functions $I(v)$ with
          $\|v\|=1$ are drawn by thick curves; $\tau$ is inside the
          ellipse in the right panel and is outside in the left panel} 
	\label{fig:rate function 1}
\end{figure}

\begin{corollary} {\rm
\label{rem:geometric view 1a}
For each $v\in \dd{R}^2_+$,
\begin{eqnarray}
\label{eqn:I lower bound}
  I(v) \ge \sup \{ \br{v, \theta}; \theta \in \sr{D} \}.
\end{eqnarray}
}\end{corollary}

\begin{proof}[Proof of \thr{domain 1}]
We first prove (a). Here, we adopt the following three categories, which are introduced
in \cite{DaiMiyazawa2011a}. 
\begin{eqnarray}
&& \text{Category I:\ \  } \quad \theta^{(2,\Gamma)}_{1} < \theta^{(1,\Gamma)}_{1}
\mbox{ and } \theta^{(1,\Gamma)}_{2} < \theta^{(2,\Gamma)}_{2},  \label{eq:cateI} \\
&& \text{Category II: } \quad \theta^{(2,\Gamma)} \le \theta^{(1,\Gamma)},  \label{eq:cateII}\\
&& \text{Category III:} \quad \theta^{(1,\Gamma)} \le \theta^{(2,\Gamma)}.   \label{eq:cateIII}
\end{eqnarray}
It is immediate from the definitions of $\sr{D}^{(1)}$ and $\sr{D}^{(2)}$ that
\begin{eqnarray*}
  \lefteqn{\sr{D}^{(1)} \cap \sr{D}^{(2)}}\\
  &&  = 
\Bigl\{ \theta \in \dd{R}^2; \text{ there exists a }
  \theta'\in \Gamma \text{ such that }  \theta < \theta', \theta_{1}' < \theta^{(1,\Gamma)}_{1} \text{ and }  \theta_{2}' < \theta^{(2,\Gamma)}_{2} \Bigr\}.
\end{eqnarray*}
Hence, $\theta \in \sr{D}^{(1)} \cap \sr{D}^{(2)}$ implies that
\begin{equation}
  \label{eq:thetalessthetaGammas}
 \theta < \bigl(\theta^{(1,\Gamma)}_{1}, \theta^{(2,\Gamma)}_{2}\bigr)^T
\end{equation}
for all categories. For Category~I, (\ref{eq:thetalessthetaGammas})
and $\theta \in \Gamma_{\max}$ implies $\theta \in
\sr{D}^{(1)} \cap \sr{D}^{(2)}$. For
Category II, $\theta^{(2,\Gamma)} \le \theta^{(1,\Gamma)}$ implies
that $\theta^{(2,\Gamma)}_{2} < \theta^{(2,\max)}_{2}$, and therefore
$\theta^{(2,\max)} \not\in \Gamma_{2}$. Hence, 
$\sr{D}^{(1)} \cap
\sr{D}^{(2)}$ is equal to $
\{\theta \in \dd{R}^{2}; \theta <
\tilde{\theta}^{(2,\rm r)} \}$ 
 by \lem{D1characterization}. Thus, we
have the following expression by the symmetry of Categories II and
III. 
\begin{eqnarray*}
  \sr{D}^{(1)} \cap \sr{D}^{(2)} = \left\{\begin{array}{ll}
\Bigl  \{\theta \in \Gamma_{\max}; \theta < \bigl(\theta^{(1,\Gamma)}_{1},
  \theta^{(2,\Gamma)}_{2}\bigr)^T\Bigr \}, \quad & \mbox{for Category I},\\ 
  \{\theta \in \dd{R}^{2}; \theta < \tilde{\theta}^{(2,\rm r)} \}, \quad & \mbox{for Category II},\\
  \{\theta \in \dd{R}^{2}; \theta < \tilde{\theta}^{(1,\rm r)} \}, \quad & \mbox{for Category III}.
  \end{array} \right.
\end{eqnarray*}
This implies that
\begin{eqnarray}
\label{eqn:tau}
  \tau = \left\{\begin{array}{ll}
  \bigl(\theta^{(1,\Gamma)}_{1}, \theta^{(2,\Gamma)}_{2}\bigr)^T, \quad & \mbox{for in Category I} ,\\
  \tilde{\theta}^{(2,\rm r)}, \quad & \mbox{for Category II} ,\\
  \tilde{\theta}^{(1,\rm r)}, \quad & \mbox{for Category III} .
  \end{array} \right.
\end{eqnarray}
Thus, the $\tau$ is identical with the one that is obtained in \cite{DaiMiyazawa2011a}. Since Theorem 2.1 of \cite{DaiMiyazawa2011a} says that 
\begin{eqnarray}
\label{eqn:D 2}
  \sr{D} = \{ \theta \in \Gamma_{\max}; \theta < \tau \},
\end{eqnarray}
  \eqn{D 1} is proved. We next prove part (b). For this, we consider
  the two cases, $\tau \in \Gamma$ and $\tau \not\in \Gamma$,
  separately. First assume that $\tau \in \Gamma$ (see, e.g., the left
  picture of \fig{boththeta1rlowerthantheta1max}). In this case,
  ${\rm n}^{\Gamma}(\tilde{\theta}^{(1, \rm r)})$ is above
  ${\rm n}^{\Gamma}(\tilde{\theta}^{(2, \rm r)})$, and therefore we have the
  following three cases for nonzero $v \in \dd{R}_{+}^{2}$ by
  \lem{geometric} and its symmetric version for $\sr{D}^{(2)}$. If $v$ is
  below ${\rm n}^{\Gamma}(\tilde{\theta}^{(2, \rm r)})$ or on
  ${\rm n}^{\Gamma}(\tilde{\theta}^{(2, \rm r)})$, then 
\begin{eqnarray*}
  I^{(1)}(v) = \brb{v, \tilde{\theta}^{(1, \rm r)}} \mbox{ and } I^{(2)}(v) = \sup\{ \brb{v, \theta}; \theta \in \Gamma\}.
\end{eqnarray*}
If $v$ is between ${\rm n}^{\Gamma}(\tilde{\theta}^{(2, \rm r)})$ and ${\rm n}^{\Gamma}(\tilde{\theta}^{(1, \rm r)})$, then
\begin{eqnarray*}
  I^{(1)}(v) = \brb{v, \tilde{\theta}^{(1, \rm r)}} \mbox{ and } I^{(2)}(v) = \brb{v, \tilde{\theta}^{(2, \rm r)}}.
\end{eqnarray*}
If $v$ is above ${\rm n}^{\Gamma}(\tilde{\theta}^{(1, \rm r)})$ or on ${\rm n}^{\Gamma}(\tilde{\theta}^{(1, \rm r)})$, then
\begin{eqnarray*}
  I^{(1)}(v) = \sup\{ \brb{v, \theta}; \theta \in \Gamma\} \mbox{ and } I^{(2)}(v) = \brb{v, \tilde{\theta}^{(2, \rm r)}}.
\end{eqnarray*}
  Thus, we have \eqn{geometric view 1} for $\tau \in \Gamma$ because $I^{(1)}(v) \le I^{(2)}(v)$ for the first case, and $I^{(2)}(v) \le I^{(1)}(v)$ for the third case.
  
  We next suppose that $\tau \not\in \Gamma$ (see, e.g.,
  \fig{D1D2CategoryI}). In this case, ${\rm n}^{\Gamma}(\tilde{\theta}^{(1,
    \rm r)})$ is not above ${\rm n}^{\Gamma}(\tilde{\theta}^{(2, \rm r)})$,
  and we similarly have the following three cases. If $v$ is below
  ${\rm n}^{\Gamma}(\tilde{\theta}^{(1, \rm r)})$ or on
  ${\rm n}^{\Gamma}(\tilde{\theta}^{(1, \rm r)})$, then 
\begin{eqnarray*}
  I^{(1)}(v) = \brb{v, \tilde{\theta}^{(1, \rm r)}} \mbox{ and } I^{(2)}(v) = \sup\{ \brb{v, \theta}; \theta \in \Gamma\}.
\end{eqnarray*}
If $v$ is between ${\rm n}^{\Gamma}(\tilde{\theta}^{(1, \rm r)})$ and ${\rm n}^{\Gamma}(\tilde{\theta}^{(2, \rm r)})$, then
\begin{eqnarray*}
  I^{(1)}(v) = I^{(2)}(v) = \sup\{ \brb{v, \theta}; \theta \in \Gamma\}.
\end{eqnarray*}
If $v$ is above ${\rm n}^{\Gamma}(\tilde{\theta}^{(2, \rm r)})$ or on ${\rm n}^{\Gamma}(\tilde{\theta}^{(2, \rm r)})$, then 
\begin{eqnarray*}
  I^{(1)}(v) = \sup\{ \brb{v, \theta}; \theta \in \Gamma\} \mbox{ and } I^{(2)}(v) = \brb{v, \tilde{\theta}^{(2, \rm r)}}.
\end{eqnarray*}
  In the first case, $I^{(1)}(v) \le I^{(2)}(v)$, and $I^{(1)}(v) = \sup\{ \brb{v, \theta}; \theta \in \sr{D}\}$, while in the third case $I^{(2)}(v) \le I^{(1)}(v)$, and $I^{(2)}(v) = \sup\{ \brb{v, \theta}; \theta \in \sr{D}\}$. Hence, we always have
\begin{eqnarray*}
  \min(I^{(1)}(v), I^{(2)}(v)) = \sup\{ \brb{v, \theta}; \theta \in \sr{D}\}.
\end{eqnarray*}
Thus, we obtain \eqn{geometric view 1} for $\tau \not\in \Gamma$.

We finally prove \eqn{I coordinate 1}. This is immediate from \eqn{geometric view 1}, \eqn{tau i} and \eqn{D 1} if $\tau \not\in \Gamma$. Otherwise, $I(v) = \min( \brb{v, \tilde{\theta}^{(1, \rm r)}}, \brb{v,\tilde{\theta}^{(2, \rm r)}})$ by \eqn{geometric view 1}. This implies \eqn{I coordinate 1} since $\brb{{\rm e}^{(1)}, \tilde{\theta}^{(1, \rm r)}} < \brb{{\rm e}^{(1)},\tilde{\theta}^{(2, \rm r)}}$ and $\brb{{\rm e}^{(2)}, \tilde{\theta}^{(2, \rm r)}} < \brb{{\rm e}^{(2)},\tilde{\theta}^{(1, \rm r)}}$ for $\tau \in \Gamma$. This completes the proof.
\end{proof}

\section{Characterization by Avram, Dai and Hasenbein}
\label{sect:ADH characterization}
\setnewcounter

Theorem 6.3 of \cite{AvramDaiHasenbein01} gives explicit expressions
for $I^{(i)}(v)$ for each $v\in \dd{R}^2_+$. We will use these expressions
to prove \thr{geometric view 1}. For concreteness, we will
prove \thr{geometric view 1} for $i=1$. Recall $\theta^{(v, \Gamma)}$
defined in (\ref{eq:thetavGamma}).
The next lemma gives an analytical expression for $  \langle
v,\theta^{(v,\Gamma)}\rangle$.
\begin{lemma}
\label{lem:nGammaProperty}
For each $v\in \dd{R}^2_+ $, 
\begin{equation}
  \label{eqn:I 0a}
  \br{v, \theta^{(v,\Gamma)}} =
  \sqrt{\br{\mu, \Sigma^{-1} \mu} \br{v, \Sigma^{-1} v}} - \br{\mu,
    \Sigma^{-1} v}.
\end{equation}
\end{lemma}
\begin{proof}
 Obviously, the right side of (\ref{eq:thetavGamma}) is attained by the $\theta$ on the ellipse
 $\partial \Gamma$ such that $v$ parallels  the normal
 $n^\Gamma(\theta)$ at the $\theta$. That is, $v$ is proportional to
 ${\rm n}^{\Gamma}(\theta)$. Thus, following \eqn{orthogonal to theta}, we have
 for some $\alpha > 0$, 
\begin{eqnarray*}
  \Sigma \theta + \mu = \alpha v.
\end{eqnarray*}
 Substituting $\theta = \Sigma^{-1}(\alpha v - \mu)$ in $\gamma(\theta) = 0$, we have
\begin{eqnarray*}
  \alpha = \sqrt{\frac {\br{\Sigma^{-1} \mu, \mu}}{\br{\Sigma^{-1}v, v}}}.
\end{eqnarray*}
  Hence,
\begin{eqnarray*}
  \br{v, \theta} = \br{\Sigma^{-1}(\alpha v - \mu), v} = \sqrt{\br{\mu, \Sigma^{-1} \mu} \br{v, \Sigma^{-1} v}} - \br{ \Sigma^{-1} \mu, v}.
\end{eqnarray*}
  This proves (\ref{eqn:I 0a}).
\end{proof}

Central in the proofs of \citet{AvramDaiHasenbein01} is two vectors
$a^2$ and  $\tilde a^2$ 
defined in (3.2) and (3.4), respectively, of their paper. They call
$\tilde a^2$  the symmetry 
of $a^2$ around face $F_2=\{z=(z_1, z_2)\in \dd{R}^2_+; z_2=0
\}$, namely the nonnegative horizontal axis. (Readers are warned again
the swap of indexes between \cite{AvramDaiHasenbein01} and this
paper. Anything pertaining to the 
horizontal axis is indexed as $1$, with a pair of parentheses, in this
paper,  and is indexed as  $2$ in \cite{AvramDaiHasenbein01}.)

It follows from (\ref{eqn:ray and p 1}) and (\ref{eqn:orthogonal to
  theta}) of this paper and (3.2) of \cite{AvramDaiHasenbein01} that one
immediately has the following geometric interpretation
\begin{equation}
  \label{eq:a2normaltheta1r}
  a^2 = n^\Gamma(\theta^{(1,\rm r)}). 
\end{equation}
(Note that our $\theta^{(1,\rm r)}$ in  (\ref{eqn:ray and p 1}) has
denominator $\langle p^{(1)}, \Sigma p^{(1)}\rangle$ because our $p^{(1)}$ is
not normalized.) To introduce their symmetry $\tilde a^2$ using their
(3.4) in \cite{AvramDaiHasenbein01}, we introduce 
\begin{equation}
  \label{eq:eandn}
  {\rm e}^2=
  \begin{pmatrix}
    1\\ 0
  \end{pmatrix} \quad \text{ and } n^2 =
  \begin{pmatrix}
    0 \\ 1
  \end{pmatrix}.
\end{equation}
One can check that the vector $a^2$ (or any other vector) has the
following orthogonal 
decomposition under inner product $\langle \Sigma^{-1}x, y\rangle$ for
$x, y\in \dd{R}^2$.
\begin{equation}
  \label{eq:a2decomposition}
 a^2 =\frac{ \langle \Sigma^{-1} a^2, {\rm e}^2\rangle }
{ \langle \Sigma^{-1} {\rm e}^2, {\rm e}^2\rangle } {\rm e}^2 +
  \frac{ \langle  a^2, n^2\rangle }
{ \langle \Sigma n^2, n^2\rangle} \Sigma n^2 .
\end{equation}
The dominators in (\ref{eq:a2decomposition}) are  due to the lacking
of normalization in  (\ref{eq:eandn}).
Their (3.4) defines $\tilde a^2$ via
\begin{equation}
  \label{eq:symmetrya2}
  \tilde a^2 =\frac{ \langle \Sigma^{-1} a^2, {\rm e}^2\rangle }
{ \langle \Sigma^{-1} {\rm e}^2, {\rm e}^2\rangle } {\rm e}^2 - 
  \frac{ \langle  a^2, n^2\rangle }
{ \langle \Sigma n^2, n^2\rangle} \Sigma n^2 .
\end{equation}
The following equation gives the geometric interpretation of $\tilde a^2$.
\begin{equation}
  \label{eq:tildea2normaltildetheta1r}
  \tilde a^2 = n^\Gamma (\tilde \theta^{(1,\rm r)}).
\end{equation}
We leave the proof of (\ref{eq:tildea2normaltildetheta1r}) to the end
of this section. 

Following Definition 3.2 and
equation (3.5) of \cite{AvramDaiHasenbein01}, the face $F_2$ is
reflective if and 
only if 
\begin{equation}
  \label{eq:1}
  \tilde a^{2}_2>0.
\end{equation}
The following lemma is immediate.
\begin{lemma}
Face $F_2$   is reflective if and only if condition
(\ref{eq:theta1maxinGamma1}) is not satisfied.
\end{lemma}

Now we are ready to prove \thr{geometric view 1}. 
\begin{proof}[Proof of Theorem~\ref{thr:geometric view 1}]
We prove \thr{geometric view 1} for $i=1$. 
First assume that condition (\ref{eq:theta1maxinGamma1}) is
satisfied. In this case, $\tilde a_2^2 \le 0$ and the face $F_2$ is not
reflective. It follows from part (b) of Theorem 6.3 in
\cite{AvramDaiHasenbein01}, for any $v\in 
\dd{R}^2_+$,
$I^{(1)}(v)$ is given by the expression (3.6) of
\cite{AvramDaiHasenbein01}, which is the right side of 
(\ref{eqn:I 0a}). Thus, in this case, 
\begin{displaymath}
I^{(1)}(v)=\langle v, \theta^{(v,
  \Gamma)}\rangle =  \sup\{ \langle v, \theta \rangle; \theta\in
\Gamma\}=
 \sup\{ \langle v, \theta \rangle; \theta\in
\Gamma_{\max}\},
\end{displaymath}
which is $ \sup\{ \langle v, \theta \rangle; \theta\in
\sr{D}^{(1)}\}$ because $\sr{D}^{(1)}=\Gamma_{\max}$.

Now assume that   
 condition~(\ref{eq:theta1maxinGamma1}) is not satisfied. In this case,
 $\tilde a_2^2 > 0$ and the face $F_2$ is 
 reflective.  If $\tilde\theta^{(1,\rm r)}$ is to the left of
 $\theta^{(2,\max)}$, then $\tilde a^2_1<0$. Thus, 
 for any $v\in \dd{R}^2_+$, $v$ is below $\tilde a^{2}$.
It follows from part (a) in Theorem 6.3 of \cite{AvramDaiHasenbein01}, for any $v\in \dd{R}^2_+$,
$I^{(1)}(v)$ is given by 
\begin{displaymath}
  \langle v,\tilde \theta^{(1,\rm r)} \rangle,
\end{displaymath}
which is  equal to $
 \sup\{ \langle v, \theta \rangle; \theta\in
\sr{D}^{(1)}\}$ by Lemma \ref{lem:geometric}.
Otherwise, we have that $\tilde\theta^{(1,\rm r)}$ is to the right of or
equal to $\theta^{(2,\max)}$. Thus, $\tilde a^2_1\ge 0$.
It follows from part (b) in Theorem 6.3 of \cite{AvramDaiHasenbein01}
for any $v\in \dd{R}^2_+$ that is on or above $\tilde a^2$, 
$I^{(1)}(v)$ is again given by the right side of (\ref{eqn:I 0a}),
which is equal to
$\langle v, \theta^{(v,
  \Gamma)}\rangle$, and therefore to $\sup\{ \langle v, \theta
\rangle; \theta\in 
\sr{D}^{(1)}\}$ by Lemma \ref{lem:geometric}. When $v\in
\dd{R}^2_+$ is below $\tilde a^2$, 
it follows again from part (a) in Theorem 6.3 of \cite{AvramDaiHasenbein01} that
$I^{(1)}(v)$ is given by 
\begin{displaymath}
  \langle v,\tilde \theta^{(1,\rm r)} \rangle,
\end{displaymath}
which is  equal to $
 \sup\{ \langle v, \theta \rangle; \theta\in
\sr{D}^{(1)}\}$ by Lemma \ref{lem:geometric}.
\end{proof}

\begin{proof}[Proof of equation (\ref{eq:tildea2normaltildetheta1r})]
 Because $\gamma(\theta^{(1, \rm r)}) = 0$, we have by (\ref{eq:a2normaltheta1r})
\begin{eqnarray*}
  \br{\Sigma^{-1} (a^{2} - \mu), (a^{2} - \mu)} + 2 \br{\mu, \Sigma^{-1} (a^{2} - \mu)} = 0,
\end{eqnarray*}
  and therefore
\begin{eqnarray}
\label{eqn:a i norm}
  \br{a^{2}, \Sigma^{-1} a^{2}} = \br{\mu, \mu}.
\end{eqnarray}
  From (\ref{eq:symmetrya2}), we have
\begin{eqnarray}
\label{eqn:tilde a ii}
  \br{ \tilde{a}^{2}, \Sigma^{-1} {\rm e}^{2}}=  \br{\Sigma^{-1} \tilde{a}^{2}, {\rm e}^{2}} = \frac{\br{\Sigma^{-1} a^{2},
    {\rm e}^{2}}}{\br {\Sigma^{-1}{\rm e}^2, {\rm e}^2}}   \br{\Sigma^{-1} {\rm e}^{2},{\rm e}^{2}}
  = \br{\Sigma^{-1} a^{2}, {\rm e}^{2}}. 
\end{eqnarray}
  Hence,
\begin{eqnarray*}
 \br{ \tilde{a}^{2}, \Sigma^{-1} \tilde{a}^{2}} &=& \frac{ \br{\Sigma^{-1}
   a^{2}, {\rm e}^{2}} \br{\tilde{a}^{2}, \Sigma^{-1} {\rm
     e}^{2}}}{\br {\Sigma^{-1}{\rm e}^2, {\rm e}^2}} - \frac{\br{a^{2},
   {n}^{2}} \br{\tilde{a}^{2}, 
   {n}^{2}}}{\br {\Sigma n^2, n^2}} \nonumber \\ 
  &=& 
\frac{ \br{\Sigma^{-1}   a^{2}, {\rm e}^{2}}^2}{\br {\Sigma^{-1}{\rm e}^2,
    {\rm e}^2}} + \frac{\br{a^{2}, 
   {n}^{2}}^2}{\br {\Sigma n^2, n^2}} \nonumber \\ 
& = & \br{a^{2}, \Sigma^{-1} a^{2}}. 
\end{eqnarray*}
  This and \eqn{a i norm} imply
\begin{eqnarray*}
  \br{\tilde{a}^{2}, \Sigma^{-1} \tilde{a}^{2}} = \br{\mu, \mu},
\end{eqnarray*}
  which is equivalent to $\gamma(\Sigma^{-1} (\tilde{a}^{2} - \mu)) =
  0$. Thus,  the point
  \begin{displaymath}
    \Sigma^{-1} (\tilde{a}^{2} - \mu)
  \end{displaymath}
is on the ellipse. From equations (\ref{eq:a2decomposition}) and
(\ref{eq:symmetrya2}), the first component of 
\begin{displaymath}
      \Sigma^{-1} (\tilde{a}^{2} - \mu) 
\end{displaymath}
is equal to the first component of $\Sigma^{-1} ({a}^{2} - \mu)$,
which is $\theta^{(1,\rm r)}_1$. Also, from (\ref{eq:a2decomposition}) and
(\ref{eq:symmetrya2}), $\tilde a^2= a^2 $ if and only
if $\langle a^2, n^2\rangle =0$. The latter is equivalent to the
normal direction 
$a^2$ at $\theta^{(1,\rm r)}$ being horizontal or
$\theta^{(1,\rm r)}=\theta^{(1,\max)}$.  Thus,
we have 
either  $\Sigma^{-1} (\tilde{a}^{2} - \mu) \neq
\theta^{(1,\rm r)}$ or  $\Sigma^{-1} (\tilde{a}^{2} - \mu) =
\theta^{(1,\rm r)}=\theta^{(1,\max)}$. Therefore, we have proved that
\begin{displaymath}
  \Sigma^{-1} (\tilde{a}^{2} - \mu) =\tilde \theta^{(1,\rm r)},
\end{displaymath}
which is equivalent to (\ref{eq:tildea2normaltildetheta1r}).
\end{proof}

\section{Product form stationary distribution}
\label{sect:Product form}
\setnewcounter
\citet{HarrisonWilliams87b} proved that a multi-dimensional SRBM has a
product form stationary distribution if and only if the SRBM data
$(\Sigma,\mu, R)$ satisfies $R^{-1}\mu<0$ and the skew symmetry condition
\begin{eqnarray}
\label{eqn:skew symmetric}
  2 \Sigma = R \Delta_{R}^{-1} \Delta_{\Sigma} + \Delta_{\Sigma}
  \Delta_{R}^{-1} R^{\rs{t}} 
\end{eqnarray}
is satisfied, where, for a matrix $A$, $\Delta_{A}$ is the diagonal
matrix whose diagonal entries are that of $A$.  In this section, we
prove the following theorem.
\begin{theorem}\label{thm:productform}
 Assume that (\ref{eqn:P matrix}) and (\ref{eqn:stability 1})
 hold. The two-dimensional SRBM has a product form stationary
 distribution if and only 
 if
 \begin{equation}
   \label{eq:productformgeometric}
\tilde   \theta^{(1,\rm r)} = \tilde \theta^{(2,\rm r)},
 \end{equation}
 holds, where $\tilde \theta^{(i,r)}$ is the symmetry of
 $\theta^{(i,r)}$ defined in (\ref{eq:thetaSymmetry}).
\end{theorem}
\begin{figure}[tbh]
 	\centering
	\includegraphics[height=2.2in]{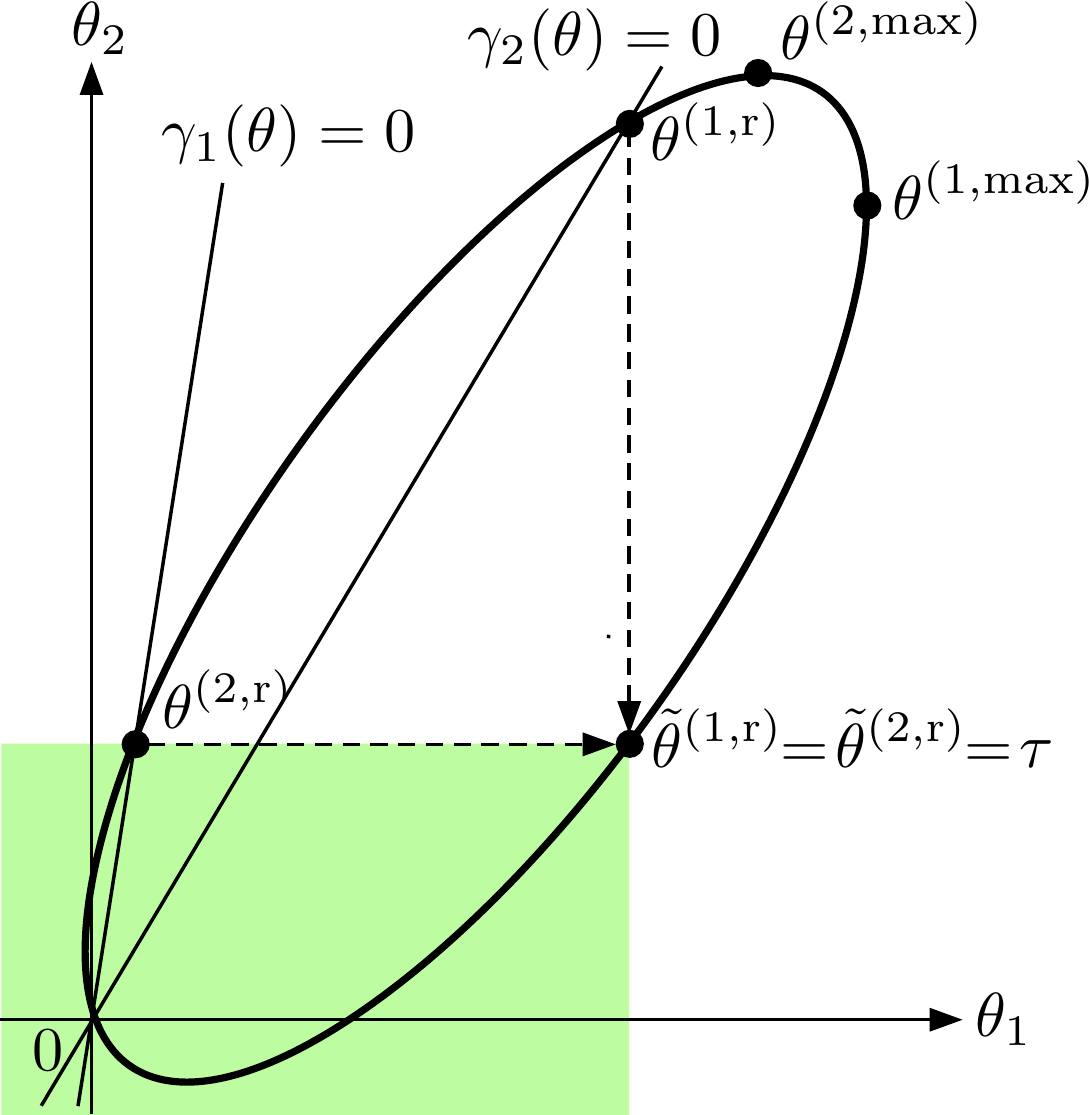} \hspace{3ex}
	\includegraphics[height=2.2in]{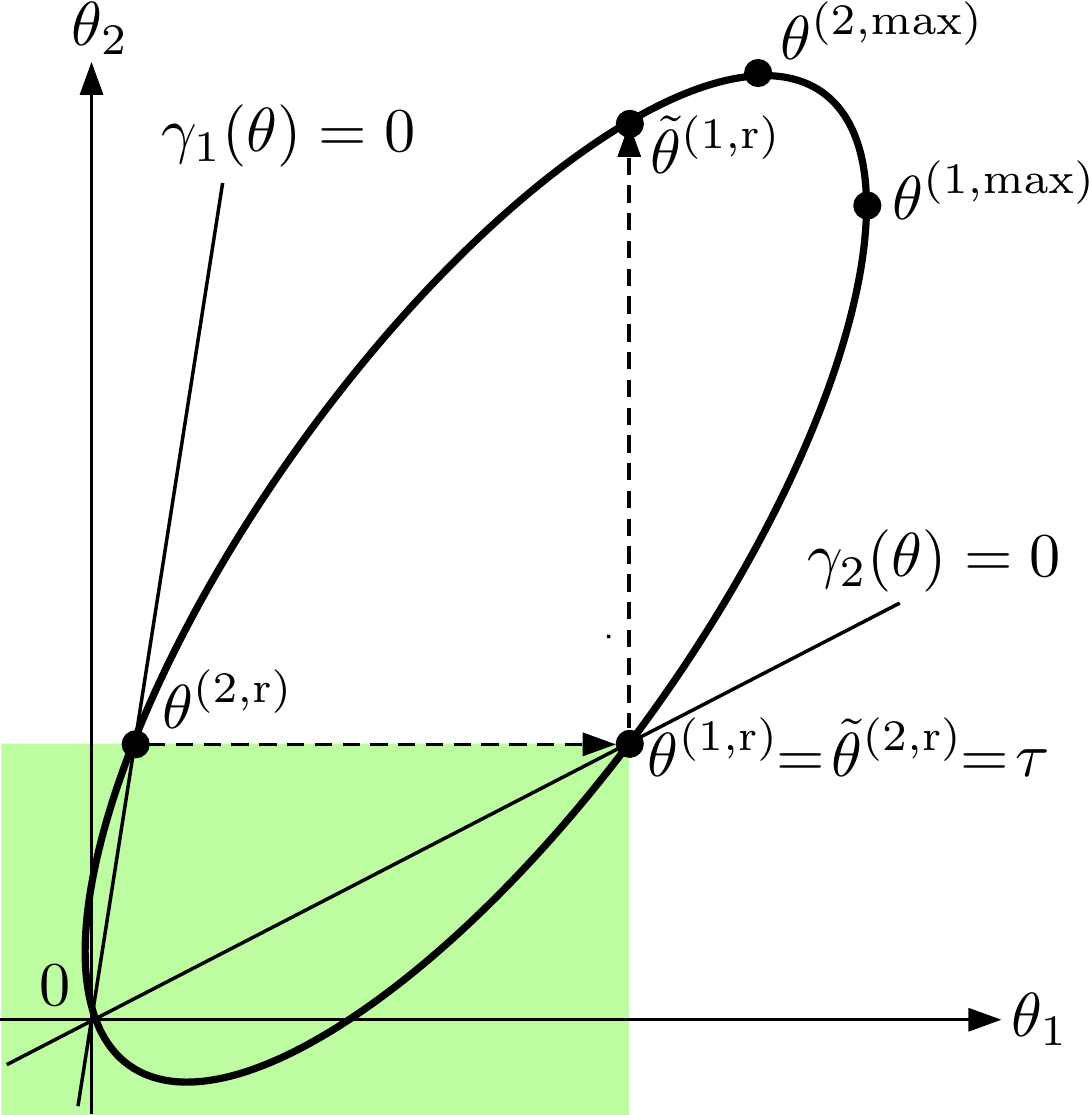}
	\caption{The locations $\theta^{(1,\rm r)}$ and $\theta^{(2,\rm r)}$;
          left panel is
          for  product form; the right panel is for non-product form} 
	\label{fig:Product form0}
\end{figure}
\fig{Product form0} gives examples for the locations of
$\tilde \theta^{(1,\rm r)}$ and $\tilde\theta^{(2,\rm r)}$;  left panel
illustrates an SRBM whose data satisfies
(\ref{eq:productformgeometric}), and the right panel illustrates an
SRBM whose data does not satisfy
(\ref{eq:productformgeometric}).
Equation
(\ref{eq:productformgeometric}) 
 provides a geometric condition on the SRBM
data for the stationary distribution to have a product form. Of
course, the algebraic condition (\ref{eqn:skew symmetric}) and
geometric condition
(\ref{eq:productformgeometric}) must be equivalent although its 
verification is not obvious.
\begin{remark}\label{rem:necessity}
\citet{AvramDaiHasenbein01} proved 
that the skew symmetry condition (\ref{eqn:skew symmetric}) implies
  \begin{equation}
    \label{eq:productformtildea}
    \tilde{a}^{1} = \tilde{a}^{2},
  \end{equation}
  where $\tilde a^{2}$ is defined in (\ref{eq:symmetrya2}) and
  $\tilde a^{1}$ is defined similarly.  Because
  $\tilde a^{2}=n^\Gamma(\tilde\theta^{(1,\rm r)})$ 
and   $\tilde a^{1}=n^\Gamma(\tilde\theta^{(2,\rm r)})$, where
$n^\Gamma(\theta)$ is the outward normal of the ellipse at $\theta$,
 condition (\ref{eq:productformtildea}) is equivalent to 
condition (\ref{eq:productformgeometric}).
Thus, geometric condition (\ref{eq:productformgeometric}) is necessary
for the stationary distribution to have a product form.
\end{remark}

We will provide a complete proof of Theorem
\ref{thm:productform} at the end of this section. Before that, we provide
some background discussion on the product form stationary distribution.
It is known that the stationary distribution
  has a density; see, for example, \cite{HarrisonWilliams87a} and \cite{daihar92}. 
  We use $\zeta(x,y)$ to denote the stationary density
  of the two-dimensional SRBM.  Thus, the stationary distribution 
  has a product form if and only if 
  \begin{equation}
    \label{eq:productformdensity}
  \zeta(x, y) = \zeta_{1}(x) \zeta_{2}(y) \quad \text{ for } x,y \in \dd{R}_{+},
  \end{equation}
  where $\zeta_{i}$'s are the marginal densities of $\zeta$. It
  follows from \cite{HarrisonWilliams87b} that when the
  stationary density is of the product form in (\ref{eq:productformdensity}), each marginal density
  $\zeta_{i}$ must be exponential. Hence, the product form
  holds if and only if there exist some $\alpha_{1}, \alpha_{2} > 0$
  such that 
\begin{eqnarray}
\label{eqn:product form 1}
  \zeta(x,y) = \alpha_{1} \alpha_{2} e^{-(\alpha_{1} x + \alpha_{2}
    y)} \quad \text{ for } x, y \ge 0. 
\end{eqnarray}
\citet{HarrisonWilliams87b} proved that when the stationary density is
of the product form, $\alpha$ is given by
\begin{eqnarray}
\label{eqn:alpha 1}
  \alpha = - 2 \Delta_{\Sigma}^{-1} \Delta_{R} R^{-1} \mu.
\end{eqnarray}

Associated with the stationary distribution are two boundary measures
$\nu_1$ and $\nu_2$. The measure $\nu_i$ has support on
$F_i=\{z\in \dd{R}^2_+; z_i=0\}$, $i=1, 2$; see, for example, Section 2 of
\cite{DaiMiyazawa2011a} for their definition. It is  known that the
stationary distribution, together with its associated boundary
measures, satisfies the basic adjoint relationship (BAR). For a statement of
BAR in two dimensions, see, for example, (4.1) of
\cite{DaiMiyazawa2011a}. 
\citet{HarrisonWilliams87a} proved the necessity of BAR  when the reflection
matrix $R$ is an ${\cal M}$ matrix and the stability condition
$R^{-1}\mu<0$ is satisfied. Assuming the stationary distribution
exists,  \citet{daihar92} proved the necessity of BAR for a general
reflection matrix $R$ that  is  completely-${\cal S}$.

Recall the two-dimensional moment generating function
$\varphi(\theta)$ defined in (\ref{eq:momentPhi}) for 
 the stationary distribution. 
We denote the moment
generating functions of boundary measures $\nu_1$ and $\nu_2$ by
$\varphi_{1}(\theta_2)$ and $\varphi_{2}(\theta_1)$, 
respectively. Then, \citet{DaiMiyazawa2011a} derived  from the basic
adjoint relationship the following key relationship among
moment generating functions: 
\begin{eqnarray}
\label{eqn:BAR}
\gamma(\theta) \varphi(\theta) = \gamma_{1}(\theta) \varphi_{1}(\theta_{2}) +  \gamma_{2}(\theta) \varphi_{2}(\theta_{1})
\end{eqnarray} for any $\theta\in \dd{R}^2$ as long as $\varphi(\theta)$ is finite.

\begin{proof}[Proof of Theorem~\ref{thm:productform}]
Remark \ref{rem:necessity} proves the necessity of the theorem.
Here, we provide a self-contained, alternative proof
for the necessity. The technique used in our necessity proof will be useful in
the sufficiency proof.

Assume that the stationary distribution has the product form
stationary density given by (\ref{eqn:product form 1}). It follows
from \cite{HarrisonWilliams87a} that boundary
measure $\nu_1$ has density
\begin{equation}
  \label{eq:boundaryDensity1}
\zeta_1(y)= \frac{\Sigma_{11}}{2r_{11}} \zeta(0,y)= \frac{\Sigma_{11}}{2r_{11}} \alpha_1\alpha_2 e^{-\alpha_2
  y} \quad \text{ for } y\ge 0
\end{equation}
such that 
\begin{equation}
  \label{eq:boundarynu1}
  \nu_1(\dd{R}_+\times B) = \int_{B} \zeta_1(y) dy \quad \text{ for
    any Borel set $B$ in $\dd{R}_+$}.  
\end{equation}
Similarly, boundary measure $\nu_2$ has
density 
\begin{equation}
  \label{eq:boundaryDensity2}
  \zeta_2(x)=\frac{\Sigma_{22}}{2r_{22}} \zeta(x,0)  = \frac{\Sigma_{22}}{2r_{22}} \alpha_1\alpha_2 e^{-\alpha_2
  x} \quad \text{ for } x\ge 0
\end{equation}
such that 
\begin{equation}
    \label{eq:boundarynu2}
    \nu_2(A\times\dd{R}_+) = \int_{A} \zeta_2(y) dy \quad \text{ for
    any Borel set $A$ in $\dd{R}_+$}.
\end{equation}
It follows from (\ref{eqn:product form 1}),
(\ref{eq:boundaryDensity1}) and (\ref{eq:boundaryDensity2}) that
\begin{eqnarray}
\label{eqn:product form 2}
&& \varphi(\theta) = \frac {\alpha_{1}} {\alpha_{1} - \theta_{1}} \frac
{\alpha_{2}} {\alpha_{2} - \theta_{2}} \quad \text{ for $\theta\in
  \dd{R}^2$ with } \theta<\alpha, \\
\label{eqn:exponential 1}
&&  \varphi_{1}(\theta_{2}) = \frac {\Sigma_{11} \alpha_{1}
  \alpha_{2}} {2r_{11}(\alpha_{2} - \theta_{2})} \quad \text{ for }
\theta_2< \alpha_2,  \\
\label{eqn:exponential 2}
&&  \varphi_{2}(\theta_{1}) = \frac {\Sigma_{22} \alpha_{1}
  \alpha_{2}} {2r_{22}(\alpha_{1} - \theta_{1})} \quad \text{ for } \theta_1<\alpha_1.
\end{eqnarray}
Plugging these expressions for $\varphi(\theta)$, $\varphi_1(\theta_2)$, and
$\varphi_2(\theta_1)$ into key relationship \eqn{BAR}, we have
\begin{eqnarray}
\label{eqn:PF gamma 1}
  \gamma(\theta) = \frac {\Sigma_{22}} {2 r_{22}} \gamma_{2}(\theta)
  (\alpha_{2} - \theta_{2}) + \frac {\Sigma_{11}} {2 r_{11}}
  \gamma_{1}(\theta) (\alpha_1 - \theta_{1})\quad \text{ for any }
  \theta< \alpha. 
\end{eqnarray}
Note that both sides of (\ref{eqn:PF gamma 1}) are quadratic functions of
$\theta$. Since (\ref{eqn:PF gamma 1}) holds for any $\theta < \alpha$
and a quadratic function of two variables is uniquely determined by  six
points,
it must hold for all $\theta\in \dd{R}^2$. In particular,  
we have
\begin{eqnarray}
\label{eqn:tau on Gamma}
  \gamma(\alpha) = 0
\end{eqnarray}
  That is, $\alpha$ must be on the ellipse $\partial
  \Gamma$. Plugging $\theta = \theta^{(1,\rm r)}$ into
  \eqn{PF gamma 1} implies that  
\begin{eqnarray}
\label{eqn:tau theta}
  \alpha_{1} = \theta^{(1,\rm r)}_{1} = \tilde \theta^{(1,\rm r)}_{1} 
\end{eqnarray}
  because $\gamma(\theta^{(1,\rm r)}) = 0$, $\gamma_{2}(\theta^{(1, r)})=0$
 and  $\gamma_{1}(\theta^{(1,\rm r)}) \ne 0$;  see Lemma 2.2 of
  \cite{DaiMiyazawa2011a} for the latter.
Plugging
$\theta =\tilde \theta^{(1,\rm r)}$ into \eqn{PF gamma 1}, we have
\begin{equation}
  \label{eq:gamma2theta1r}
    \frac{\Sigma_{22}}{2r_{22}} \gamma_2(\tilde
  \theta^{(1,\rm r)})(\alpha_2-\tilde\theta^{(1,\rm r)}_2)=0
\end{equation}
 because $\gamma(\tilde
\theta^{(1,\rm r)}) = 0$ and $\alpha_1 =\tilde\theta^{(1,\rm r)}_1$.
If $\tilde\theta^{(1,\rm r)} = \theta^{(1,\rm r)}$, then
$\theta^{(1,\rm r)}=\theta^{(1,\max)}$. This fact, together with 
(\ref{eqn:tau on Gamma}) and (\ref{eqn:tau theta}), implies that 
\begin{equation}
  \label{eq:alpha2equaltildetheta1r2}
  \alpha_2= \tilde \theta^{(1,\rm r)}_2.
\end{equation}
If $\tilde\theta^{(1,\rm r)} \neq  \theta^{(1,\rm r)}$, then 
$\gamma_2(\tilde
  \theta^{(1,\rm r)})\neq 0$, and (\ref{eq:gamma2theta1r}) implies again
  (\ref{eq:alpha2equaltildetheta1r2}). Equations
(\ref{eqn:tau theta}) and (\ref{eq:alpha2equaltildetheta1r2}) imply that
\begin{equation}
  \label{eq:alphaequaltildetheta1r}
  \alpha = \tilde \theta^{(1,\rm r)}.
\end{equation}
Similarly, we can prove
\begin{equation}
  \label{eq:alphaequaltildetheta2r}
  \alpha = \tilde \theta^{(2,\rm r)}.
\end{equation}
Equations (\ref{eq:alphaequaltildetheta1r}) and
(\ref{eq:alphaequaltildetheta2r}) imply that
(\ref{eq:productformgeometric}) holds, proving the necessity of the theorem.

\begin{figure}[h]
 	\centering
	\includegraphics[height=2.2in]{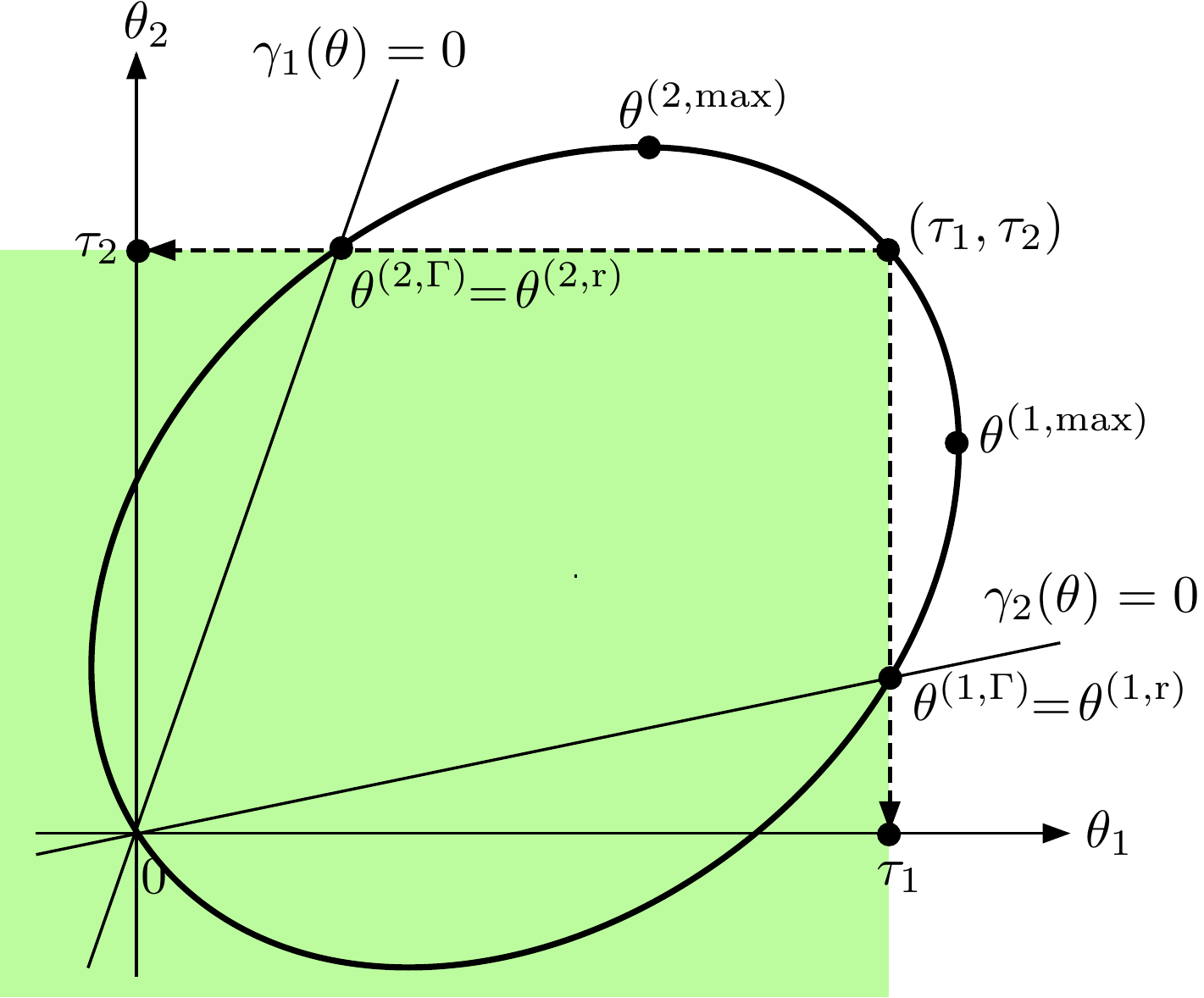}
	\includegraphics[height=2.2in]{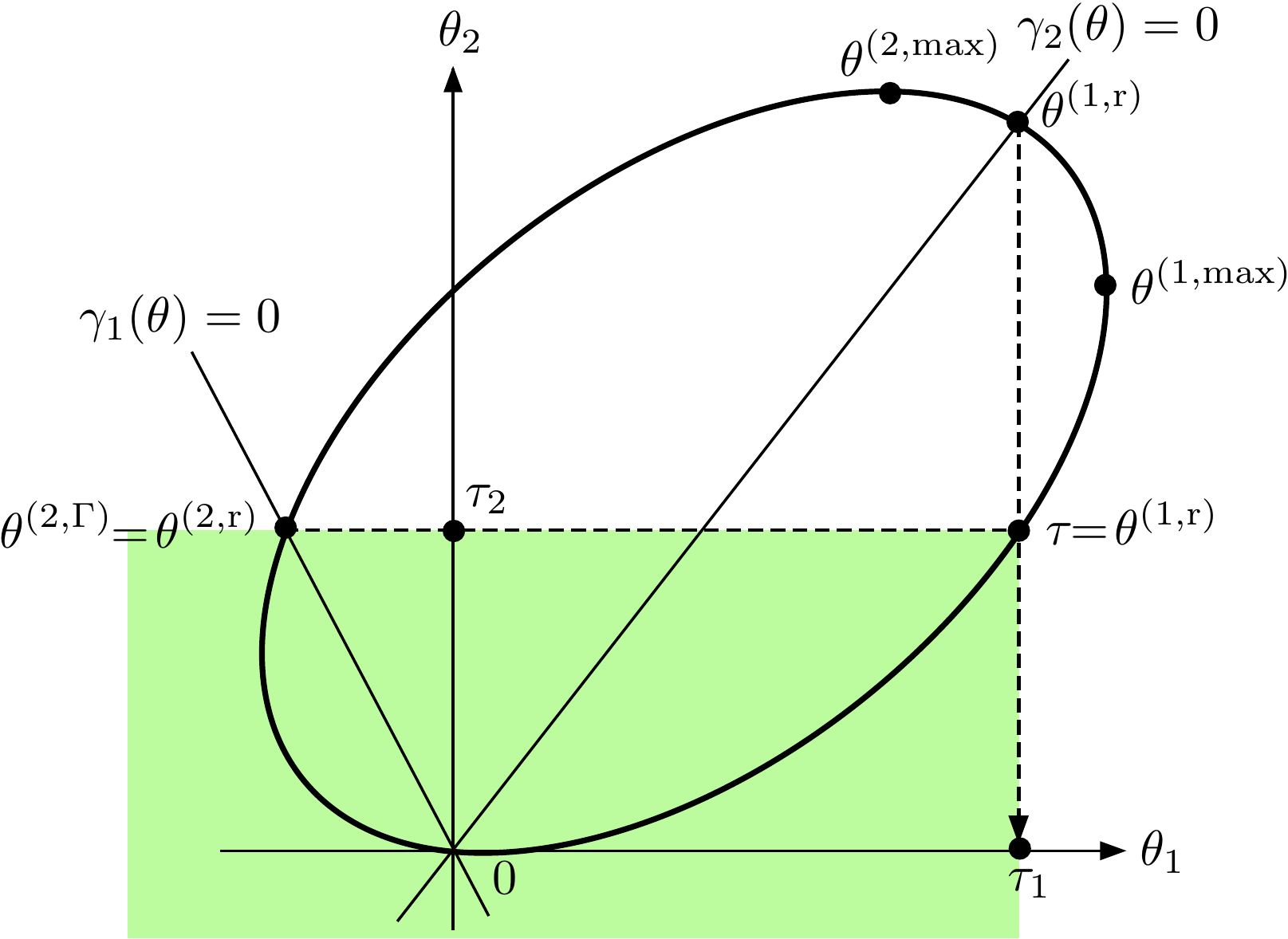}
	\caption{The locations of $\tau$, $\theta^{(1,\rm r)}$ and $\theta^{(2,\rm r)}$ for the product form}
	\label{fig:Product form}
\end{figure}

 We now  prove the sufficiency of the theorem. Assume
 (\ref{eq:productformgeometric}) holds. Let 
 \begin{displaymath}
   \alpha = \tilde \theta^{(1,\rm r)} = \tilde \theta^{(2,\rm r)}.
 \end{displaymath}
Then, $\alpha$ is on the ellipse, and thus $\gamma(\alpha)=0$. Therefore,
\begin{eqnarray}
  \gamma(\theta)& = & \gamma(\theta)-\gamma(\alpha) = \frac{1}{2} \Bigl[ \langle
  \alpha, \Sigma\alpha\rangle - \langle
  \theta, \Sigma \theta\rangle  \Bigr] + \langle \mu, \alpha\rangle -
  \langle \mu, \theta\rangle \nonumber \\
& = & \Bigl \langle
  \alpha-\theta, \frac{1}{2} \Sigma (\theta+\alpha)+\mu \Bigr\rangle
\nonumber  \\
&=& f_{1}(\theta) (\alpha_{1} - \theta_{1}) + f_{2}(\theta) (\alpha_{2} -
\theta_{2}), \label{eq:gammadecomposition}
\end{eqnarray}
where 
\begin{displaymath}
  \begin{pmatrix}
       f_1(\theta) \\ f_2(\theta)
  \end{pmatrix}
= \frac{1}{2} \Sigma (\theta+\alpha)+\mu.
\end{displaymath}
Also, $\gamma(0)=0$ gives 
\begin{displaymath}
\alpha_1 f_1(0) + \alpha_2 f_2(0)=0.
\end{displaymath}
Therefore, there exists a $k\in\dd{R}$ such that 
\begin{equation}
  \label{eq:f0alpha}
  \begin{pmatrix}
       f_1(0) \\ f_2(0)
  \end{pmatrix}
= k
\begin{pmatrix}
  \alpha_2 \\
-\alpha_1
\end{pmatrix}.
\end{equation}
It follows from (\ref{eq:gammadecomposition}) that
\begin{eqnarray}
  \gamma(\theta)
&=& \Bigl(f_{1}(\theta)-k (\alpha_2-\theta_2)\Bigr) (\alpha_{1} -
\theta_{1}) + \Bigl(f_{2}(\theta)+ k(\alpha_1-\theta_1)\Bigr)
(\alpha_{2} -  
\theta_{2}) \nonumber \\
&=& 
 \tilde f_{1}(\theta) (\alpha_{1} -\theta_{1}) + \tilde f_{2}(\theta)
(\alpha_{2} -  \theta_{2}), \label{eq:gammadecomposition2}
\end{eqnarray}
where
\begin{eqnarray*}
&&   \tilde f_{1}(\theta) = f_{1}(\theta)-k (\alpha_2-\theta_2), \\
&&   \tilde f_{2}(\theta) =f_{2}(\theta)+ k(\alpha_1-\theta_1).
\end{eqnarray*}
Both functions $\tilde f_{1}(\theta)$ and $\tilde f_{2}(\theta)$ are linear in
$\theta$ and (\ref{eq:f0alpha}) implies that $\tilde f_1(0)=0$ and $\tilde f_2(0)=0$.
If $\alpha \ne \theta^{(1, \rm r)}$, substituting $\theta =
\theta^{(1,\rm r)}$ into  equation (\ref{eq:gammadecomposition2}) yields 
\begin{eqnarray}
\label{eqn:f2 vanish}
  \tilde f_{2}(\theta^{(1,\rm r)}) = 0.
\end{eqnarray}
 If $\alpha = \theta^{(1, \rm r)}$, we have $\theta^{(1, \rm r)} =
 \theta^{(1, \max)}$,
\begin{figure}[t]
 	\centering
	\includegraphics[height=1.9in]{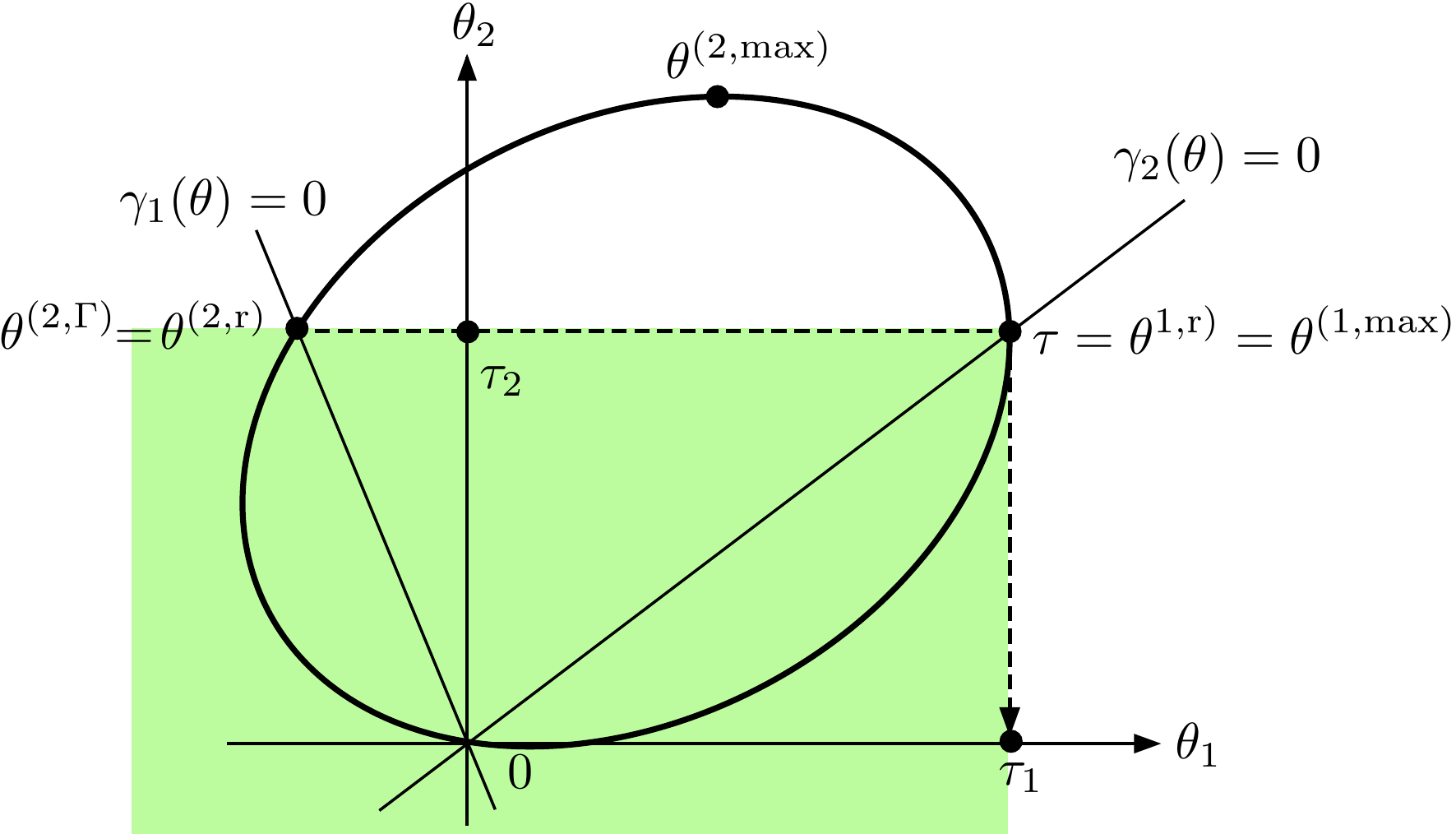}
	\caption{The case that $\theta^{(1, \rm r)} = \theta^{(1, \max)}$ for the product form}
	\label{fig:PF boundary case}
\end{figure}
which implies
\begin{eqnarray*}
  \left. \frac {\partial}{\partial \theta_{2}} \gamma(\theta) \right|_{\theta= \theta^{(1, \rm r)}} = 0,
\end{eqnarray*}
from  which \eqn{f2 vanish} follows again. Thus, we have proved that
  two linear equations,
  $\tilde f_2(\theta)=0$ and  $\gamma_{2}(\theta)=0$, both have $0$ and
  $\theta^{(1,\rm r)}$ as their roots. Hence, $\tilde f_{2}(\theta)$ must
  be proportional to $\gamma_{2}(\theta)$. Namely, there exists a
  constant $c_2$ such that 
  \begin{equation}
    \label{eq:f2prop}
    \tilde f_2(\theta)=c_2 \gamma_2(\theta) \quad \text{ for all }
    \theta\in \dd{R}^2.
  \end{equation}
Similarly, there exists a constant $c_1$ such that
\begin{equation}
  \label{eq:f1prop}
    \tilde f_1(\theta)=c_1 \gamma_1(\theta) \quad \text{ for all }
    \theta\in \dd{R}^2.
\end{equation}
It follows from (\ref{eq:gammadecomposition2}), (\ref{eq:f2prop}),
and (\ref{eq:f1prop}) 
 that 
that $\gamma(\theta)$ can be written as 
\begin{eqnarray}
\label{eqn:PF gamma 2}
  \gamma(\theta) = c_{1} \gamma_{1}(\theta) (\alpha_{1} - \theta_{1}) +
  c_{2} \gamma_{2}(\theta) (\alpha_{2} - \theta_{2}) \quad \text{ for
    all } \theta\in \dd{R}^2.
\end{eqnarray}
Comparing the coefficients of quadratic terms in (\ref{eqn:PF gamma
  2}),  we find
\begin{eqnarray*}
  c_{i} = \frac {\Sigma_{ii}} {2 r_{ii}}, \qquad i=1,2.
\end{eqnarray*}
  Hence, if we define $g(\theta)$, $g_{1}(\theta_{2})$ and $g_{2}(\theta_{1})$ as
\begin{eqnarray*}
  g(\theta) = \frac {\alpha_{1}} {\alpha_{1} - \theta_{1}} \frac
  {\alpha_{2}} {\alpha_{2} - \theta_{2}}, \qquad g_{2}(\theta_{2}) = \frac
  {c_{1} \alpha_{1} \alpha_{2}} {\alpha_{2} - \theta_{2}}, \qquad
  g_{1}(\theta_{1}) = \frac {c_{2} \alpha_{1} \alpha_{2}} {\alpha_{1} -
    \theta_{1}} \quad  \text{ for } \theta< \alpha, 
\end{eqnarray*}
  then $g, g_{1}, g_{2}$ are the solutions $\varphi, \varphi_{1},
  \varphi_{2}$ to the key relationship \eqn{BAR}. Let $\pi$ be the
  probability measure on $\dd{R}^2_+$ with $\zeta(x,y)$ in
  (\ref{eqn:product form 1}) as its density 
  function, and $\nu_1$ and $\nu_2$ be two boundary measures defined
  in (\ref{eq:boundarynu1}) and (\ref{eq:boundarynu2}). Thus, we have
  proved that $\pi$, $\nu_1$ and $\nu_2$ satisfy BAR (4.1) of
  \cite{DaiMiyazawa2011a} for all exponential functions 
  \begin{displaymath}
    f(x,y) = e^{\theta_1 x+ \theta_2 y} \quad x, y \ge 0
  \end{displaymath}
with $\theta<\alpha$.
From this, one can argue that BAR is satisfied for all $f\in
C^2_b(\dd{R}^2)$, the set of  functions which together with their
first and second order partial
derivatives are continuous and bounded.  It follows from \cite{daikur94a}
that 
BAR uniquely determines the stationary distribution. Thus, the SRBM
must have  $\zeta(x,y)$  in (\ref{eqn:product form 1}) as its
stationary density.
\end{proof}

\section{Exact asymptotics for boundary measures}
\label{sect:Exact asymptotics}
\setnewcounter

As before, we assume the SRBM satisfies conditions
(\ref{eqn:P matrix}) and (\ref{eqn:stability 1}) so that it has a
unique stationary distribution.
  In \sectn{Product form}, immediately below (\ref{eqn:alpha
    1}), we introduced two
  boundary measures $\nu_1$ and $\nu_2$ associated with the stationary
  distribution of the SRBM.  In this section, we study
  the exact asymptotics for the tail distribution of $\nu_i$. 

  The main result of this section is Theorem~\ref{thr:boundary measure
    asymptotics}, which will be stated shortly.  We emphasize that
  Lemmas~\ref{lem:tail equivalence 1}, \ref{lem:handy 1} and
  \ref{lem:handy 2} that are used in the proof of \thr{boundary
    measure asymptotics} constitute a significant contribution of this
  paper as well. These lemmas can potentially be used for other
  related problems when conditions in asymptotic inversion lemmas
    such as Lemmas \lemt{f asymptotics 1} and \lemt{f asymptotics 2}
    are difficult to check. 
  
  For convenience, we write $f_{1}(x) \sim f_{2}(x)$ as $x\to\infty$ for two functions $f_{1}, f_{2}$ on $[0,\infty)$ if
\begin{eqnarray*}
  \lim_{x \to \infty} f_{1}(x)/f_{2}(x) = 1.
\end{eqnarray*}
 Recall  $\tau$   defined in (\ref{eqn:tau i}) and the three
 categories defined through 
(\ref{eq:cateI})-(\ref{eq:cateIII}) in the proof of 
\thr{domain 1}. 
\begin{theorem} {\rm
\label{thr:boundary measure asymptotics}
Under conditions \eqn{P matrix} and \eqn{stability 1}, we have the
exact asymptotic:
\begin{eqnarray}
\label{eqn:}
  \lim_{x \to \infty} \frac{ \nu_{2}((x,\infty))}{ x^{\kappa} e^{- \tau_{1} x}} = b > 0,
\end{eqnarray}
  for some constant $b > 0$, where $\kappa$ is given below. 

\noindent (a) If $\tau_{1} < \theta^{(1,\max)}_{1}$, then
for Category I and Category III, 
 $\kappa=0$,  and for Category II,
\begin{eqnarray}
\label{eqn:boundary asymptotic 1}
  \kappa = \left\{\begin{array}{ll}
  0, \quad & \tau_{1} \ne \theta^{(1,\rm r)}_{1},\\
  1, \quad & \tau_{1} = \theta^{(1,\rm r)}_{1}.
  \end{array} \right.
\end{eqnarray}
  (b) If $\tau_{1} = \theta^{(1,\max)}_{1}$, then for Category I,
\begin{eqnarray}
\label{eqn:boundary asymptotic 2}
  \kappa = \left\{\begin{array}{ll}
  -\frac 12, \quad & \theta^{(1,\rm r)} = \theta^{(1,\max)},\\
  -\frac 32, \quad & \theta^{(1,\rm r)}\ne \theta^{(1,\max)},
  \end{array} \right.
\end{eqnarray}
  and for Category II,
\begin{eqnarray}
\label{eqn:boundary asymptotic 3}
  \kappa = \left\{\begin{array}{ll}
  0, \quad & \theta^{(1,\rm r)} = \theta^{(1,\max)},\\
  -\frac 12, \quad & \theta^{(1,\rm r)} \ne \theta^{(1,\max)}.
  \end{array} \right.
\end{eqnarray}
}\end{theorem}

We note that,  in Category III, one must have
$\tau_1<\theta^{(1,\max)}_1$. Thus, we do not need to consider
Category III when $\tau_1=\theta^{(1,\max)}_1$.
In Section~\ref{sect:Product form}, we introduced the moment
generating functions $\varphi_1(\theta_2)$ and $\varphi_2(\theta_1)$
for boundary measures $\nu_{1}$ and $\nu_{2}$.
\citet{DaiMiyazawa2011a} studied analytic 
behaviors of $\varphi_{1}(z)$ and $\varphi_{2}(z)$ as functions of
complex variable 
$z$. In particular, they proved that $\varphi_2(z)$ is analytic if
$\Re{z}<\tau_1$ and is singular at $z=\tau_1$. Furthermore, they
characterized the nature of the singularities at $z=\tau_1$.
See Lemmas 6.6, 6.7 and 6.8 of \cite{DaiMiyazawa2011a}. We will use the
singularity of $\varphi_2(z)$ at $\tau_1$ to prove Theorem
\ref{thr:boundary measure asymptotics}. Before proving the theorem, we
first introduce a series of
three lemmas.


For a nonnegative function  $f:\dd{R}_+\to \dd{R}_+$ that is
integrable on $\dd{R}_+$, let 
\begin{equation}
  \label{eq:mgfg}
  g(z) = \int_0^\infty e^{z x} f(x) dx
\end{equation}
be its moment generating function. Assume that $g$ is analytic for
$\Re z< \alpha$  and is singular at $z=\alpha$ for some $\alpha>0$. 
Assuming $f$ is continuous, one
can apply the complex 
inversion formula to represent $f$ as a contour integral of $g$. From
this integral representation, one then uses the singular behavior of
$g$ at $z=\alpha$ to
obtain the exact tail asymptotic of $f$.  \citet{Doetsch74} is a good
reference for the complex inversion formula. \citet{DaiMiyazawa2011a} have
successfully applied this technique to study the exact asymptotic for tail 
distribution of $\langle c, Z(\infty)\rangle $ in an arbitrarily given direction
$c\in \dd{R}_+^2$. In particular, they developed two asymptotic
inversion lemmas, Lemmas C.1 and C.2 of \cite{DaiMiyazawa2011a}, that
allow one to directly infer the form of exact tail asymptotics of $f$
from the singular behavior of $g$ at $z=\alpha$. With these two
lemmas, there is no 
need to get into contour integral and complex inversion formula. For
convenience, these two lemmas are reproduced in the appendix of this
paper.

In order to apply one of their asymptotic inversion lemmas,  one
needs to verify a certain set of conditions on the moment generating
function $g$. For example, to apply Lemma \ref{lem:f asymptotics 1}, 
one often needs to verify the sufficient condition (\ref{eqn:C1bc
  sufficient condition}). Namely, there are some positive constants 
$\beta$, $a$, $b$ and $\delta$ such that 
\begin{equation}
  \label{eq:inversion3}
  |g(z)| < \frac{a}{|z|^{1+\delta}} \text{ for } \Re z \in[0, \beta]
  \text{ and } |\Im z|> b.
\end{equation}
In the proof of our Theorem~\ref{thr:boundary measure asymptotics}, 
    \begin{displaymath}
    f(x)=\nu_2(x, \infty)  \quad \text{ for } x \ge 0          
    \end{displaymath}
 and the corresponding moment generating function is
\begin{eqnarray}
\label{eqn:g 2}
  g(z) = \frac{1}{z}( \varphi_2(z)-\varphi_2(0)) \quad \Re z < \tau_1.
\end{eqnarray}
  When $\tau_1<\theta^{(1,\max)}_1$, one can prove that $\varphi_2(z)$
  is bounded in a region given by $\Re z \in[0, \tau_1+\epsilon]$ and
  $|\Im z|> b$ for some $\epsilon>0$ and some $b>0$. Still we are
  unable to prove condition (\ref{eq:inversion3}) for some positive
  $\delta$. Furthermore, we do not know how to check conditions
  (\appt{Asymptotic inversion}1a)-(\appt{Asymptotic inversion}1c) of
  Lemma~\ref{lem:f asymptotics 1} directly.

To overcome the preceding difficulty, we introduce a new
technique. We develop a series of
three lemmas. These lemmas  may have independent interest.
 Let 
\begin{eqnarray*}
  T(f)(x) = \int_{x}^{\infty} f(u) du, \qquad x > 0,
\end{eqnarray*}
be the tail distribution of $f$.
Inductively, for $n \ge 1$, define
\begin{eqnarray*}
  T^{n}(f)(x) = \int_{x}^{\infty} T^{(n-1)}(f)(u) du, \qquad x > 0,
\end{eqnarray*}
  where $T^{0}(f) = f$. Denote the moment generating function of
  $T^{n}(f)$ by $\hat{T^{n}}(f)$. Then, it is easy to see that
\begin{eqnarray}
\label{eqn:T 1}
 && \hat{T}(f)(z) = z^{-1} \bigl(g(z) - g(0)\bigr),\\
\label{eqn:T 2}
 && \hat{T^{2}}(f)(z) = z^{-2} \bigl(g(z) - g(0) - g'(0)z\bigr),
\end{eqnarray}
where $g=\hat{T^0}(f)$ is the moment generating function of $f$ given
in (\ref{eq:mgfg}). Thus,
it is easier to get asymptotics for $T^{n}(f)$ with $n \ge 1$ because
it is easier for $\hat{T^{n}}(f)(z)$ to satisfy sufficient conditions
such as (\ref{eq:inversion3}). Thus, if $T^{n}(f)$ have the same
asymptotics as $f$, it will be more convenient to work with
${T^{n}}(f)(z)$. We will show that $f$ and $T^{n}(f)$ indeed have
the same asymptotics if the asymptotic is limited to a certain
exponential type.

We call a function $f$ to be {\em ultimately nonincreasing} if there
is an $x_{0}$ such that $f(x)$ is nonincreasing for $x \ge x_{0}$. We
have the following lemma showing the equivalence of tail asymptotics.
The lemma is a 
converse of Lemma~D.5 of \cite{DaiMiyazawa2011a}.
\begin{lemma} 
\label{lem:tail equivalence 1}
  Assume that the nonnegative function $f$ is ultimately nonincreasing. For each integer $n \ge 1$, if
\begin{mylist}{-3}
\item [(\lemt{tail equivalence 1}a)] $T^{n}(f)(x) \sim b \,
  x^{\kappa} e^{-\alpha x}$ as $x\to\infty$ for some real number
  $\kappa$, positive 
  number $b$ and nonnegative number $\alpha$,
\end{mylist}
 then $f(x) \sim
  \alpha^{n} b \, x^{\kappa} e^{-\alpha x}$ as $x\to\infty$. In particular, (\lemt{tail equivalence 1}a) implies that $T(f)(x) \sim
  \alpha^{n-1} b \, x^{\kappa} e^{-\alpha x}$ as $x\to\infty$ for any nonnegative and integrable function $f$ on $[0,\infty)$.
\end{lemma}

We defer the proof of this lemma to \app{Proof of Lemma 6.1}.  We
combine this lemma with the classical asymptotic inversion 
  results in \app{Asymptotic inversion}. The idea is very simple. We
  just use the moment generating function $\hat{T}(f)$ of (\ref{eqn:T 1}) or
  $\hat{T}^{2}(f)$ of \eqn{T 2} instead of the moment generating function  $g$.
For this, $g$ of \eqn{g 2} is redefined in the following lemma.

  \begin{lemma}    
\label{lem:handy 1}
  Let
  \begin{equation}
    \label{eq:ginlemmas}
    g(z) = \int_0^\infty e^{ zx} \nu(dx)    
  \end{equation}
 be the moment generating function of a finite measure on $[0,\infty)$. Assume that $g$ satisfies the 
  following two conditions
\begin{mylist}{0}
\item [(\lemt{handy 1}a)] there is a complex variable function
\begin{eqnarray*}
  g_{0}(z) = \frac {b} {(\alpha - z)^{k}} ,
\end{eqnarray*}
for some positive integers $k$, some complex number $b$ and positive
numbers $\alpha, \beta$ satisfying $0 < \alpha < \beta$ such that
$g(z) - g_{0}(z)$ is analytic for $\Re z < \beta$,  

\item [(\lemt{handy 1}b)] there exists an $\epsilon_0>0$ such that  $g(z) -
  g_{0}(z)$ is bounded for all $z \in \dd{C}$ 
   such that $0 \le \Re z \le \beta$ and $|z - \alpha| >
  \epsilon_{0}$.  
\end{mylist}
  Then, we have the following exact tail asymptotic for $\nu$.
\begin{eqnarray}
\label{eqn:handy 1}
  \nu(x, \infty) \sim \frac{b\alpha^{-1}}{\Gamma(k)} x^{k-1} e^{-\alpha
   x} \quad \text{ as 
  } x \to \infty,
\end{eqnarray}
  where, $\Gamma(\lambda)$ is the gamma function as defined in Section 53 of Volume II of \cite{Markushevich77} (see also Theorem 10.14 in its Section 54 for its integral representation), and
  by convention $\frac 1{\Gamma(\lambda)} = 0$ when $\lambda =
  0, -1, -2, \ldots$.
\end{lemma}

The asymptotic of this lemma is less sharp than that of \lem{f
  asymptotics 1} because the lower order term in \lem{f
  asymptotics 1} has a precise exponential form. However, we do not require for $\nu(x,\infty)$ to be continuous in $x$. Before proving this lemma, we
present  similarly a  version of \lem{f asymptotics 2}. In this case, the
asymptotic results are the same, but we need to restrict $\lambda$ to
be greater than $-1$.

\begin{lemma} 
\label{lem:handy 2}
Assume   $\nu$ and $g$ as in \lem{handy 1}.  Assume that there
  are  some $\alpha > 0$ and some $\delta \in [0,
  \frac {\pi}{2})$ such that  the following two conditions hold for $g$
\begin{mylist}{0}
\item [(\lemt{handy 2}a)] $g(z)$ is analytic on
on $\sr{G}_{\delta}(\alpha) \equiv \{z \in \dd{C}; z \ne \alpha,
|\arg(z-\alpha)| > \delta\}$ and is bounded on
$\sr{G}_{\delta}(\alpha) \cap \{ z \in \dd{C}; |z-\alpha| >
\epsilon_{0} \}$ for some $\epsilon_{0} > 0$, where for
$z\in\dd{C}\setminus \{0\}$,
$\arg(z)\in (-\pi,
\pi)$ is the angle of $z$;
\item [(\lemt{handy 2}b)] 
  for some complex numbers $a, b$ and some real number $\lambda$ satisfying $\lambda \ne 0$ and $\lambda > -1$, 
\begin{eqnarray}
\label{eqn:handy bridge}
  \lim_{ z \to \alpha\atop z\in \sr{G}_{\delta}(\alpha)} (\alpha
  - z)^{\lambda} \bigl(g(z) -a \bigr) = b.
\end{eqnarray}
\end{mylist}
  Then   
\begin{eqnarray}
\label{eqn:handy 2}
  \nu(x, \infty) \sim \frac{b \alpha^{-1}}{\Gamma(\lambda)} x^{\lambda-1} e^{-\alpha x}, \quad
\text{ as } x\to\infty.
\end{eqnarray}
\end{lemma}

In the reminder of this section, 
we first prove \thr{boundary measure asymptotics} using Lemmas
\lemt{tail equivalence 1}, \lemt{handy 1} and \lemt{handy 2}. We then
prove the latter two lemmas, leaving the proof of  \lem{tail
  equivalence 1} to \app{Proof of Lemma 6.1}. 

\begin{proof*}{The proof of \thr{boundary measure asymptotics}}
  We will use Lemmas \ref{lem:handy 1} and \ref{lem:handy 2} to prove
  the theorem.  Recall that $\varphi_2(z)$ is the moment generating
  function of boundary measure $\nu_2$ and it will play the role of
  $g(z)$ in these two lemmas. 

 (a) For Categories I and III, we have
  $\tau_{1} = \theta^{(1,\rm r)}_{1}< \theta^{(1,\max)}_1$.  By parts
  (a) and (c) of Lemma 6.6 of \cite{DaiMiyazawa2011a}, all conditions
  of \lem{handy 1} with $k=1$ are satisfied, and hence the lemma
  implies the theorem with $\kappa=0$ for this case.  For Category II, if $\tau_{1} \ne
  \theta^{(1,\rm r)}_{1}$, then $\varphi_{2}(z)$ has a simple pole at
  $z=\tau_{1}$, and all the conditions of \lem{handy 1} with $k=1$ are
  satisfied by parts (b) and (c) of Lemma 6.6 of
  \cite{DaiMiyazawa2011a}. Otherwise, $\varphi_{2}(z)$ has a double
  pole at $z=\tau_{1}$, all the conditions of \lem{handy 1} with $k=2$
  are similarly satisfied. In either case, Lemma \ref{lem:handy 1}
  implies the theorem with $\kappa$ being given in \eqn{boundary
    asymptotic 1}. 
  
  (b) We apply Lemma 6.8 of \cite{DaiMiyazawa2011a} because $\tau_{1} =
  \theta^{(1,\max)}_{1}$ in this case. We use parts (c) and (d) of
  this lemma for Category I and II, respectively. For Category I, it
  is not hard to 
  see that all the conditions of \lem{handy 2}  are satisfied  with
  $\lambda = \frac 
  12$ for $\theta^{(1,\rm r)}_{1} = \theta^{(1,\max)}_{1}$ and
  $\lambda = -\frac 12$ for $\theta^{(1,\rm r)}_{1} \not=
  \theta^{(1,\max)}_{1}$. Therefore the theorem is proved in this case 
with $\kappa$ being given in
  \eqn{boundary asymptotic 2}. Similarly, for Category
  II, all conditions of \lem{handy 2} are satisfied
  with $\lambda = 1$ for $\theta^{(1,\rm r)}_{1} =
  \theta^{(1,\max)}_{1}$ and $\lambda = \frac 12$ for $\theta^{(1,\rm
    r)}_{1} \not= \theta^{(1,\max)}_{1}$. Thus, the lemma implies the
  theorem with $\kappa$ being given by  \eqn{boundary
    asymptotic 3}.  
\end{proof*}

\begin{proof*}{The proof of \lem{handy 1}}
Let
  \begin{equation}
    \label{eq:fnu2}
    f(x)=\nu(x, \infty)  \quad \text{ for } x \ge 0.
  \end{equation}
  We apply \lem{f asymptotics 1} to $T(f)$. Let 
\begin{eqnarray*}
 &&  f_{0}(x) = \frac{b}{\Gamma(k)} x^{k-1} e^{-\alpha x} \qquad x \ge
 0, \\ 
&& g_0(z) = \frac{b}{(\alpha-z)^k}, \qquad z\in \dd{C}\setminus\{\alpha\}.
\end{eqnarray*}
It is clear that the moment generating function
of $f_0$ is $g_0$.  Since $g_{0}(0) = b \alpha^{-k}$, we have, from \eqn{T 1},
\begin{eqnarray*}
  \hat{T}(f_{0})(z) = z^{-1} \left( g_{0}(z) - g_{0}(0) \right) = \frac {b} {\alpha^{k+1}} \sum_{\ell=1}^{k} \frac {\alpha^{\ell}} {(\alpha - z)^{\ell}}
\end{eqnarray*}
This and $g_{0}'(0) = k b \alpha^{-(k+1)}$ yield
\begin{eqnarray*}
  \hat{T^{2}}(f_{0})(z) &=& z^{-2} \left( g_{0}(z) - g_{0}(0) - z g'(0)\right) \\
  &=& z^{-1} \left( \hat{T}(f_{0}) - g'(0)\right) \\
  &=& z^{-1} \frac {b} {\alpha^{k+1}} \sum_{\ell=1}^{k} \left( \frac {\alpha^{\ell}} {(\alpha - z)^{\ell}} - 1 \right) \\
  &=& \frac {b} {\alpha^{k+2}} \sum_{\ell=1}^{k} \sum_{m=1}^{\ell} \frac {\alpha^{m}} {(\alpha - z)^{m}}\\
  &=& \frac {b} {\alpha^{k+2}} \sum_{m=1}^{k} (k+1-m) \frac {\alpha^{m}} {(\alpha - z)^{m}}\\
  &=& \frac {b} {\alpha^{2}} \frac {1} {(\alpha - z)^{k}} + \frac {b} {\alpha^{k+2}} \sum_{m=1}^{k-1} (k+1-m) \frac {\alpha^{m}} {(\alpha - z)^{m}}.
\end{eqnarray*}
  We check (\appt{Asymptotic inversion}1a), (\appt{Asymptotic
    inversion}1b) and (\appt{Asymptotic inversion}1c) of \lem{f
    asymptotics 1} for $\hat T(f)$ and $\hat{T^{2}}(f_{0})$. Note that the continuity of $f$ of \lem{f asymptotics 1} is satisfied because $T(f)(x)$ is continuous in $x$ for $f$ of \eq{fnu2}.
(When \lem{f
    asymptotics 1} is applied, $\hat T(f)$ and $\hat{T^{2}}(f_{0})$ serve the roles of $g$ and
 $g_0$, respectively, in that lemma.) It is
  clear that 
\begin{eqnarray*}
  \hat{T}^{2}(f)(z)- \hat{T^{2}}(f_{0})(z) = z^{-2} \left( g(z) - g_{0}(z) - (g(0) - g_{0}(0)) - z (g'(0) - g_{0}'(0) )\right)
\end{eqnarray*}
  has a removable singularity at $z=0$, and is analytic for $\Re z \le
  \beta$ because of condition (\lemt{handy 1}a). Thus (\appt{Asymptotic inversion}1a) is
  satisfied. By (\lemt{handy 1}b) and \eqn{T 2}, (\appt{Asymptotic inversion}1b)
  is also satisfied. It remains to check (\appt{Asymptotic
    inversion}1c). From \eqn{C1bc sufficient condition} and  \eqn{T
    2}, we only need to check that 
\begin{eqnarray*}
  \int_{-\infty}^{+\infty} e^{- i y x} \frac 1{\beta+i y} dy
\end{eqnarray*}
  uniformly converges for $x > T$ for some $T > 0$, but this is already proved on page 237 of \cite{Doetsch74}. Hence, \lem{f asymptotics 1} yields
\begin{eqnarray*}
  T(f)(x) = \frac {b} {\alpha^{2}} \frac {1} {\Gamma(k)} x^{k-1} e^{-\alpha x} + \frac {b} {\alpha^{k+2}} \sum_{m=1}^{k-1} (k+1-m) \frac {\alpha^{m}} {\Gamma(m)} x^{m-1} e^{-\alpha x} + o(e^{-\beta x})
\end{eqnarray*}
  Since $k$ is a positive integer, this implies
\begin{eqnarray*}
  T(f)(x) \sim \frac {b} {\alpha^{2}} \frac {1} {\Gamma(k)}
  x^{k-1} e^{-\alpha x}, \quad \text{ as } x \to \infty. 
\end{eqnarray*}
Applying \lem{tail equivalence 1} with $n=1$, we conclude \eqn{handy
  1} and thus prove the lemma.
\end{proof*}

\bigskip

\begin{proof*}{The proof of \lem{handy 2}}
Let $f$ be defined in \eq{fnu2}.
We first apply \lem{f asymptotics 2} to $T(f)$ to obtain
  the exact tail asymptotic for $T(f)$.
 Because of \eq{fnu2}, we have
\begin{eqnarray}
\label{eqn:Tf 2}
 \hat{T}(f)(z) = z^{-2} \bigl(g(z) - g(0) - g'(0)z\bigr)
\end{eqnarray}
Obviously, this $\hat{T}(f)(z)$ satisfies the conditions (\appt{Asymptotic inversion}2a) and (\appt{Asymptotic inversion}2b) by (\lemt{handy 2}a) of this lemma. Thus, we only need to verify (\appt{Asymptotic inversion}2c).

We consider three cases separately, depending on $\lambda > 0$, $\lambda < 0$ or $\lambda=0$. We first assume that $\lambda > 0$. In this case, we put $d=0$, then we have, by \eqn{handy bridge} and \eqn{Tf 2},
\begin{eqnarray*} 
  \lefteqn{\lim_{ z \to \alpha\atop z\in \sr{G}_{\delta}(\alpha)} (\alpha -
    z)^{\lambda} \hat{T}(f)(z)}\\
    && = \lim_{ z \to \alpha\atop z\in \sr{G}_{\delta}(\alpha)} z^{-2} (\alpha - z)^{\lambda} \left( g(z) - a + a - g(0) - g'(0) z \right) \\
    &&= \lim_{ z \to \alpha\atop z\in
      \sr{G}_{\delta}(\alpha)} z^{-2} (\alpha - z)^{\lambda} \bigl(
    g(z) - a \bigr) = b \alpha^{-2}. 
\end{eqnarray*}
  We next assume that $-1 < \lambda < 0$. Then, we must have, by \eqn{handy bridge},
\begin{eqnarray*}
  a = g(\alpha).
\end{eqnarray*}
and therefore it follows from \eqn{Tf 2} that
\begin{eqnarray*}
  \lefteqn{\hat{T}(f)(z) - \hat{T}(f)(\alpha) = z^{-2} \bigl(g(z) - g(0) - g'(0)z\bigr) - \alpha^{-2} \bigl(g(\alpha) - g(0) - g'(0) \alpha\bigr)} \hspace{10ex}\\
  && = z^{-2} \bigl(g(z) - a \bigr) + \frac {(\alpha-z)(\alpha+z)}{z^{2} \alpha^{2}} \bigl(g(\alpha) - g(0)\bigr) - \frac {\alpha-z}{z \alpha} g'(0).
\end{eqnarray*}
Hence, letting $d = \hat{T}(f)(\alpha)$, we have, by \eqn{handy bridge} and the fact that $\lambda > -1$,
\begin{eqnarray*}
\lim_{ z \to \alpha\atop z\in \sr{G}_{\delta}(\alpha)} (\alpha - z)^{\lambda} \bigl( \hat{T}(f)(z) - d \bigr) = b \alpha^{-2},
\end{eqnarray*}
We finally assume that $\lambda = 0$. In this case, letting $d = \alpha^{-2} (a - g(0) - g'(0) \alpha)$, \eqn{handy bridge} and \eqn{Tf 2} yield
\begin{eqnarray*}
  \lim_{ z \to \alpha\atop z\in \sr{G}_{\delta}(\alpha)} \bigl( \hat{T}(f)(z) - d \bigr) = \alpha^{-2} (b + a - g(0) - g'(0) \alpha) - d = b \alpha^{-2}.
\end{eqnarray*}
Thus, \eqn{bridge case 1} is verified for $\hat{T}(f)$ and $c_{0} = b \alpha^{-2}$ in all cases.
Hence, by \lem{f asymptotics 2} 
  \begin{equation}
    \label{eq:lemma2Tfasym}
  T(f)(x) \sim \frac {b} {\alpha^{2}} \frac {1} {\Gamma(\lambda)}
  x^{\lambda-1} e^{-\alpha x}, \quad \text{ as } x \to \infty. 
  \end{equation}
(When \lem{f
    asymptotics 2} is applied, $\hat T(f)$  serves the role of $g$ in
  that lemma.)  
Finally,  (\ref{eqn:handy 2}) and thus the lemma follows from Lemma
\ref{lem:tail 
  equivalence 1} and (\ref{eq:lemma2Tfasym}).
\end{proof*}


\section*{Acknowledgements}  

This research was supported in part by NSF grants  CMMI-0825840,
CMMI-1030589, CNS-1248117,  and  by Japan Society for the Promotion of Science under
grant No.\ 24310115. 

\bigskip

\appendix

\noindent {\bf \Large Appendix}

\section{Proof of \lem{tail equivalence 1}}
\label{app:Proof of Lemma 6.1}
\setnewcounter

  We only prove the claim for $n=1$ because the case for $n \ge 2$ is iteratively obtained. Since $f$ is ultimately nonincreasing, we can find $x_{0} \ge 0$ such that $f(x)$ is nonincreasing for $x \ge x_{0}$. By the assumption, we have, for any $\epsilon > 0$, there exists some $x_{1} > x_{0}$ such that
\begin{eqnarray}
\label{eqn:asymptotic bound 1}
  (1-\epsilon) b x^{\kappa} e^{-\alpha x} < T(f)(x) = \int_{x}^{\infty} f(u) du < (1+\epsilon) b x^{\kappa} e^{-\alpha x}, \qquad \forall x \ge x_{1}.
\end{eqnarray}
  Hence, for each $\delta > 0$ and each $x \ge x_{1} + \delta$,
\begin{eqnarray*}
  \int_{x - \delta}^{x} f(u) du &<& (1+\epsilon) b (x-\delta)^{\kappa} e^{-\alpha (x-\delta)} - (1-\epsilon) b x^{\kappa} e^{-\alpha x}\\
  &=& \left(\left(\frac {x-\delta}{x}\right)^{\kappa} e^{\alpha \delta} - 1\right) b x^{\kappa} e^{-\alpha x} + \epsilon  \left( \left(\frac {x-\delta}{x}\right)^{\kappa} e^{\alpha \delta} + 1 \right) b x^{\kappa} e^{-\alpha x}
\end{eqnarray*}
  Since $f(u)$ is nonincreasing in $u$, we have
\begin{eqnarray*}
  \delta \limsup_{x \to \infty} \frac {f(x)} {x^{\kappa} e^{-\alpha x}} \le ( e^{\alpha \delta} - 1) b + \epsilon ( e^{\alpha \delta} + 1) b.
\end{eqnarray*}
  Since $\epsilon$ can be arbitrarily small, we obtain
\begin{eqnarray*}
  \limsup_{x \to \infty} \frac {f(x)} {x^{\kappa} e^{-\alpha x}} \le \frac 1{\delta} (e^{\alpha \delta} - 1) b.
\end{eqnarray*}
  Letting $\delta$ go to zero, we arrive at
\begin{eqnarray}
\label{eqn:upper bound}
  \limsup_{x \to \infty} \frac {f(x)} {x^{\kappa} e^{-\alpha x}} \le \alpha b.
\end{eqnarray}
  Similarly, for each $\delta > 0$ and each $x \ge x_{1}$,
\begin{eqnarray*}
  \int_{x}^{x+\delta} f(u) du &>& (1-\epsilon) b x^{\kappa} e^{-\alpha x} - (1+\epsilon) b (x+\delta)^{\kappa} e^{-\alpha (x+\delta)}\\
  &=& \left(1 - \left(\frac {x+\delta}{x}\right)^{\kappa} e^{-\alpha \delta} \right) b x^{\kappa} e^{-\alpha x} - \epsilon  \left( 1 + \left(\frac {x+\delta}{x}\right)^{\kappa} e^{-\alpha \delta} \right) b x^{\kappa} e^{-\alpha x},
\end{eqnarray*}
which  yields
\begin{eqnarray}
\label{eqn:lower bound}
  \liminf_{x \to \infty} \frac {f(x)} {x^{\kappa} e^{-\alpha x}} \ge \alpha b.
\end{eqnarray}
  Hence, \eqn{upper bound} and \eqn{lower bound} conclude $f(x) \sim \alpha b \, x^{\kappa} e^{-\alpha x}$.
  
\section{Asymptotic inversion lemmas}
\label{app:Asymptotic inversion}
\setnewcounter

In this appendix, we state Lemmas C.1 and C.2 in
\cite{DaiMiyazawa2011a}. Actually, we here present a refined version of
Lemma C.1, whereas Lemma C.2 is unchanged. This refinement is closer to
Theorems 35.1 of Doetsch \cite{Doetsch74}. The reason for modifying
Lemma C.1 is clear in the proof of Lemma \ref{lem:handy 1} in
\sectn{Exact asymptotics}. 
 For two functions $h_1$ and $h_2$, we say $h_1=o(h_2)$ as
  $x\to\infty$ if 
  \begin{displaymath}
    \lim_{x\to\infty} h_1(x)/h_2(x)=0.
  \end{displaymath}

\begin{lemma} 
\label{lem:f asymptotics 1}
  Let $g$ be the moment generating function in (\ref{eq:mgfg}) of a
  nonnegative, continuous and integrable function $f$. Assume the
  following conditions are satisfied for some integer $m \ge 1$:
\begin{mylist}{0}
\item [(\appt{Asymptotic inversion}1a)] there is a complex variable
  function 
\begin{eqnarray*}
  g_{0}(z) = \sum_{j=1}^{m} \sum_{\ell =1}^{k_{j}} \frac
  {c^{(j)}_{\ell}} {(\alpha_{j} - z)^{\ell}} 
\end{eqnarray*}
 for some positive integers $k_{j}$ and
  some  positive numbers $\beta, \alpha_{j}, c^{(j)}_{\ell}$ for
  $\ell=1, \ldots, k_j$ and $j= 1,\ldots m$
  with $0 < \alpha_{1} < \ldots < \alpha_{m-1} < \alpha_{m} <
  \beta$ such that  $g(z) - g_{0}(z)$ is analytic for $\Re z < \beta$,

\item [(\appt{Asymptotic inversion}1b)] $g(z)$ uniformly converges to
  $0$ as $z \to \infty$ for $0 \le \Re z \le \beta$, 

\item [(\appt{Asymptotic inversion}1c)]  for some $T > 0$, $\int_{-\infty}^{+\infty} e^{- i y x} g(\beta+iy) dy$ uniformly converges for $x > T$.
\end{mylist}
Then
\begin{eqnarray}
\label{eqn:f asymptotics 1}
  f(x) = \sum_{j=1}^{m} \sum_{\ell =1}^{k_{j}} \frac{c^{(j)}_{\ell}}{\Gamma(\ell)} x^{\ell-1} e^{-\alpha_{j} x} + o(e^{-\beta x}), \qquad x \to \infty.
\end{eqnarray}

\end{lemma}

  \begin{remark}
 Conditions (\appt{Asymptotic inversion}1b) and
  (\appt{Asymptotic inversion}1c) are satisfied if, for some
  constants  $a, b,
  \delta > 0$,  
\begin{eqnarray}
\label{eqn:C1bc sufficient condition}
  |g(z)| < \frac {a}{|z|^{1+\delta}}, \quad \text{ for } \Re z \in [0,
  \beta] \text{ and } |\Im z| > b.
\end{eqnarray}    
  \end{remark}

\begin{lemma} {\rm
\label{lem:f asymptotics 2} 
 Let $f$ and $g$ be two functions in \lem{f asymptotics 1}. Assume
 that the
 following two conditions hold for some $\alpha > 0$ and some $\delta 
 \in [0, \frac {\pi}{2})$:  
\begin{mylist}{0}
\item [(\appt{Asymptotic inversion}2a)]  $g(z)$ is analytic on
  $\sr{G}_{\delta}(\alpha) \equiv \{z \in \dd{C}; z \ne \alpha,
  |\arg(z-\alpha)| > \delta\}$, where for $z\in\dd{C}\setminus \{0\}$,
  $\arg(z)$ is the angle of $z$,

\item [(\appt{Asymptotic inversion}2b)] 
  $g(z) \to 0$ as $|z| \to \infty$ for $z \in
  \sr{G}_{\delta}(\alpha)$,

\item [(\appt{Asymptotic inversion}2c)] 
  for some complex number $d$, some $\lambda\in \dd{R}$ and $c_{0} \in \dd{R}$, 
\begin{eqnarray}
\label{eqn:bridge case 1}
  \lim_{ z \to \alpha\atop z\in \sr{G}_{\delta}(\alpha)} (\alpha
  - z)^{\lambda} \bigl(g(z) - d \bigr) = c_{0}.
\end{eqnarray}
\end{mylist}
  Then   
\begin{eqnarray}
\label{eqn:f asymptotics 2}
f(x)=\frac{c_{0}}{\Gamma(\lambda)} x^{\lambda-1} e^{-\alpha x}(1+o(1)) .
\end{eqnarray}
}\end{lemma}

\bibliographystyle{ims}
\bibliography{dai-2}

\def\cprime{$'$} \def\cprime{$'$} \def\cprime{$'$} \def\cprime{$'$}
  \def\cprime{$'$} \def\cprime{$'$} \def\cprime{$'$}
\begin{thebibliography}{17}
\expandafter\ifx\csname natexlab\endcsname\relax\def\natexlab#1{#1}\fi
\expandafter\ifx\csname url\endcsname\relax
  \def\url#1{\texttt{#1}}\fi
\expandafter\ifx\csname urlprefix\endcsname\relax\def\urlprefix{URL }\fi
\providecommand{\eprint}[2][]{\url{#2}}

\bibitem[{Avram et~al.(2001)Avram, Dai and Hasenbein}]{AvramDaiHasenbein01}
\textsc{Avram, F.}, \textsc{Dai, J.~G.} and \textsc{Hasenbein, J.~J.} (2001).
\newblock Explicit solutions for variational problems in the quadrant.
\newblock \textit{Queueing Systems}, \textbf{37} 259--289.

\bibitem[{Borovkov and Mogul{\cprime}ski{\u\i}(2001)}]{BorovkovMogulski01}
\textsc{Borovkov, A.~A.} and \textsc{Mogul{\cprime}ski{\u\i}, A.~A.} (2001).
\newblock Large deviations for {M}arkov chains in the positive quadrant.
\newblock \textit{Uspekhi Mat. Nauk}, \textbf{56} 3--116.
\newblock \urlprefix\url{http://dx.doi.org/10.1070/RM2001v056n05ABEH000398}.

\bibitem[{Dai and Harrison(1992)}]{daihar92}
\textsc{Dai, J.~G.} and \textsc{Harrison, J.~M.} (1992).
\newblock Reflected {B}rownian motion in an orthant: numerical methods for
  steady-state analysis.
\newblock \textit{Annals of Applied Probability}, \textbf{2} 65--86.

\bibitem[{Dai and Kurtz(1994)}]{daikur94a}
\textsc{Dai, J.~G.} and \textsc{Kurtz, T.~G.} (1994).
\newblock Characterization of the stationary distribution for a semimartingale
  reflecting {Brownian} motion in a convex polyhedron.
\newblock Preprint.

\bibitem[{Dai and Miyazawa(2011)}]{DaiMiyazawa2011a}
\textsc{Dai, J.~G.} and \textsc{Miyazawa, M.} (2011).
\newblock Reflecting brownian motion in two dimensions: Exact asymptotics for
  the stationary distribution.
\newblock \textit{Stochastic Systems}, \textbf{1} 146--208.

\bibitem[{Doetsch(1974)}]{Doetsch74}
\textsc{Doetsch, G.} (1974).
\newblock \textit{Introduction to the theory and application of the {L}aplace
  transformation}.
\newblock Springer-Verlag, New York.
\newblock Translated from the second German edition by Walter Nader.

\bibitem[{Dupuis and Ramanan(2002)}]{DupuisRamanan-2002}
\textsc{Dupuis, P.} and \textsc{Ramanan, K.} (2002).
\newblock A time-reversed representation for the tail probabilities of
  stationary reflected {Brownian} motion.
\newblock \textit{Stochastic Processes and their Applications}, \textbf{98}
  253--287.

\bibitem[{Harrison and Hasenbein(2009)}]{HarrisonHasenbein09}
\textsc{Harrison, J.~M.} and \textsc{Hasenbein, J.~J.} (2009).
\newblock Reflected {B}rownian motion in the quadrant: Tail behavior of the
  stationary distribution.
\newblock \textit{Queueing Systems}, \textbf{61} 113--138.

\bibitem[{Harrison and Williams(1987{\natexlab{a}})}]{HarrisonWilliams87a}
\textsc{Harrison, J.~M.} and \textsc{Williams, R.~J.} (1987{\natexlab{a}}).
\newblock Brownian models of open queueing networks with homogeneous customer
  populations.
\newblock \textit{Stochastics}, \textbf{22} 77--115.

\bibitem[{Harrison and Williams(1987{\natexlab{b}})}]{HarrisonWilliams87b}
\textsc{Harrison, J.~M.} and \textsc{Williams, R.~J.} (1987{\natexlab{b}}).
\newblock Multidimensional reflected {B}rownian motions having exponential
  stationary distributions.
\newblock \textit{Annals of Probability}, \textbf{15} 115--137.

\bibitem[{Hobson and Rogers(1993)}]{HobsonRogers93}
\textsc{Hobson, D.~G.} and \textsc{Rogers, L. C.~G.} (1993).
\newblock Recurrence and transience of reflecting {B}rownian motion in the
  quadrant.
\newblock \textit{Math. Proc. Cambridge Philos. Soc.}, \textbf{113} 387--399.

\bibitem[{Kozyakin et~al.(1993)Kozyakin, Mandelbaum and
  Vladimirov}]{KozyMandVlad1993}
\textsc{Kozyakin, V.}, \textsc{Mandelbaum, A.} and \textsc{Vladimirov, A.}
  (1993).
\newblock Absolute stability and dynamic complementarity.
\newblock Preprint.

\bibitem[{Majewski(1996)}]{maj96}
\textsc{Majewski, K.} (1996).
\newblock Large deviations of stationary reflected {B}rownian motions.
\newblock In \textit{Stochastic Networks: Theory and Applications} (F.~P.
  Kelly, S.~Zachary and I.~Ziedins, eds.). Oxford University Press.

\bibitem[{Markushevich(1977)}]{Markushevich77}
\textsc{Markushevich, A.~I.} (1977).
\newblock \textit{Theory of functions of a complex variable. {V}ol. {I}, {II},
  {III}}.
\newblock English ed. Chelsea Publishing Co., New York.
\newblock Translated and edited by Richard A. Silverman.

\bibitem[{Miyazawa(2009)}]{Miyazawa09}
\textsc{Miyazawa, M.} (2009).
\newblock Tail decay rates in double {QBD} processes and related reflected
  random walks.
\newblock \textit{Math. Oper. Res.}, \textbf{34} 547--575.
\newblock \urlprefix\url{http://dx.doi.org/10.1287/moor.1090.0375}.

\bibitem[{Williams(1995)}]{Williams95}
\textsc{Williams, R.~J.} (1995).
\newblock Semimartingale reflecting {B}rownian motions in the orthant.
\newblock In \textit{Stochastic Networks} (F.~P. Kelly and R.~J. Williams,
  eds.), vol.~71 of \textit{The {IMA} Volumes in Mathematics and its
  Applications}. Springer, New York, 125--137.

\bibitem[{Williams(1996)}]{Williams96}
\textsc{Williams, R.~J.} (1996).
\newblock On the approximation of queueing networks in heavy traffic.
\newblock In \textit{Stochastic Networks: Theory and Applications} (F.~P.
  Kelly, S.~Zachary and I.~Ziedins, eds.). Royal Statistical Society, Oxford
  University Press.

\end{thebibliography}

\end{document}